\documentclass[11pt]{amsart}
\usepackage{amsmath, amssymb,color, amsthm}
\usepackage[utf8]{inputenc}
\usepackage{bigints}
\usepackage{hyperref, xcolor, tikz}
\usepackage{enumitem, subcaption}
\usepackage{standalone}

\usepackage{geometry}

\usepackage{placeins}
\usepackage[normalem]{ulem}

\numberwithin{equation}{section}

\definecolor{goetheblau}{cmyk}{1.00 0.20 00 0.40}
\definecolor{hellgrau}{cmyk}{0.04 0.04 0.05 0.02}
\definecolor{sandgrau}{cmyk}{0.12 0.09 0.13 0}
\definecolor{dunkelgrau}{cmyk}{0.25 0.25 0.30 0.75}
\definecolor{purple}{cmyk}{0.08 1.00 0.30 0.36}
\definecolor{emorot}{cmyk}{0.04 1.00 0.80 0.07}
\definecolor{senfgelb}{cmyk}{0.01 0.25 1.00 0.05}
\definecolor{gruen}{cmyk}{0.62 0.40 0.87 0.09}
\definecolor{magenta}{cmyk}{0.08 0.86 0.12 0.12}
\definecolor{orange}{cmyk}{0 0.70 1.00 0.04}
\definecolor{sonnengelb}{cmyk}{0 0.12 0.95 0}
\definecolor{hellesgruen}{cmyk}{0.40 0.17 0.81 0.07}
\definecolor{lichtblau}{cmyk}{0.80 00 0.06 0.04}

\hypersetup{%
  colorlinks = true,
  linkcolor  = emorot,
  citecolor = hellesgruen
}

\newtheorem{theorem}{Theorem}[section]
\newtheorem{corollary}[theorem]{Corollary}
\newtheorem{definition}[theorem]{Definition}
\newtheorem{lemma}[theorem]{Lemma}
\newtheorem{proposition}[theorem]{Proposition}
\theoremstyle{remark}
\newtheorem{remark}[theorem]{Remark}
\newtheorem{example}[theorem]{Example}
\newcommand{\N}{\mathbb{N}}
\newcommand{\Z}{\mathbb{Z}}
\newcommand{\R}{\mathbb{R}}

\usepackage{commath}
\usepackage{esdiff}
\DeclareMathOperator{\Var}{\mathbf{Var}_{ \mathfrak y }}
\DeclareMathOperator{\Varo}{\mathbf{Var}}
\DeclareMathOperator{\EW}{\mathbf{E}}
\DeclareMathOperator{\WS}{\mathbf{P}}
\DeclareMathOperator{\Cov}{\mathbf{Cov}_{ \mathfrak y }}
\DeclareMathOperator{\Covbin}{\mathbf{Cov}_{ \mathfrak y_{\mathrm{bin}}}}

\newcommand{\appA}{Appendix~\ref{sec:slepianapplication}}
\newcommand{\appB}{Appendix~\ref{sec:proof:th:dbfbm}}
\newcommand{\appC}{Appendix~\ref{sec:analyticidentities}}
\newcommand{\appD}{Appendix~\ref{sec:proofprop:coalbranches}}
\newcommand{\appG}{Appendix~\ref{sec:gaussianstuff}}

\usepackage{graphicx}
\usepackage{mathrsfs}

\usepackage[]{todonotes}
\hypersetup{pdfauthor={Jan Lukas Igelbrink},%
  pdftitle={Notizen},%
  pdfsubject={Notizen},%
  pdfkeywords={one, two},%
  pdfproducer={LaTeX},%
  pdfcreator={pdfLaTeX}
}
\usepackage{autonum}
\usetikzlibrary{arrows.meta,bending}
\usetikzlibrary{patterns}
\usepackage{pict2e,color}
\usepackage{placeins}
\title[Branching fractional Brownian motion]{Branching fractional Brownian motion: \\discrete approximations and maximal displacement}
\thanks{We thank Matthias Birkner and Anton Wakolbinger for valuable discussions and suggestions, Nicola Kistler for educating us about generalized random energy models and their connection to branching random walks, Marius Schmidt for the arguments used in Section~\ref{sec:THmax2proof}, Lisa Hartung for pointing us to \cite{slowdownbbm}, and an anonymous reviewer for pointing out an alternative proof of Theorem~\ref{eq:THmax2}, see
Remark~\ref{rem:refereeproof}. 
 We also thank the \emph{Allianz f\"ur Hochleistungsrechnen Rheinland-Pfalz} for granting us access to the High Performance Computer \textsc{Elwetritsch}, on which the simulations underlying our figures have been performed.  }
\author[A. Gonz\'alez Casanova]{Adri\'an Gonz\'alez Casanova}
\address{Adri\'an Gonz\'alez Casanova \\ Universidad Nacional Aut\'onoma de M\' exico (UNAM), Instituto de Matem\'aticas and Department of Statistics, University of California, Berkeley.
}
\email{adrian.gonzalez@im.unam.mx, gonzalez.casanova@berkeley.edu}

\author[J.~L. Igelbrink]{Jan Lukas Igelbrink}
\address{Jan Lukas Igelbrink  \\ Institut für Mathematik, 
  Johannes Gutenberg-Universit\"at Mainz and Goethe-Universit\"at Frankfurt, Germany.}
\email{jigelbri@uni-mainz.de, igelbrin@math.uni-frankfurt.de}

\subjclass[2020]{Primary 60J80; secondary 60G18, 60G15}
\keywords{branching fractional Brownian motion, tree-indexed processes, power law Pólya urn, coalescing renewal processes}

\begin{document}
\maketitle

\noindent
\allowdisplaybreaks
\begin{abstract}
    We construct and study branching fractional Brownian motion with Hurst parameter $H\in(1/2,1)$. The construction relies on a generalization of the discrete approximation of fractional Brownian motion (Hammond and Sheffield \cite{HS}) to power law Pólya urns indexed by trees. We show that the first order of the speed of branching fractional Brownian motion with Hurst parameter $H$ is $ct^{H+1/2}$ where $c$ is explicit and only depends on the Hurst parameter. A notion of ``branching property'' for processes with memory emerges naturally from our construction. 
  \end{abstract}
  \newpage
\tableofcontents
\newpage
\section{Introduction}\label{sec:Intro}
Fractional Brownian Motion with Hurst parameter $H$, also known as $H$-fractional Brownian motion, is the unique centered Gaussian process $B^H$ with stationary increments and variance $\Varo\big[ B^H_t \big]=t^{2H}$. We will restrict to the case $H \in (\frac12,1)$, for which  fractional Brownian Motion has positively correlated increments (and lacks the Markov property).  By a well-known result of Mandelbrot and van Ness \cite{mandelbrotvanness}, fractional Brownian motion (FBM) has a kernel representation in terms of Wiener increments $\dif W(s)$: 
\begin{equation}\label{eq:FBM2_introduction}
  B^H(t)= \int_{(-\infty, t]} K(s,t) \dif{W(s)}, \quad t\in \R.
\end{equation}
The definition of the kernel $K= K^H$ will be  recalled in Section~\ref{sec:bfbm}. 
In the following we will construct and analyse
a {\em branching} version of fractional Brownian motion, more specifically {\em fractional Brownian motion indexed by a Yule tree~$\mathfrak Y$}. This construction works in two steps: the first is for a given realisation $\mathfrak y$ of $\mathfrak Y$, the second is a randomisation over $\mathfrak y$. 

%
 \begin{figure}[h]
    \includegraphics[scale=0.75]{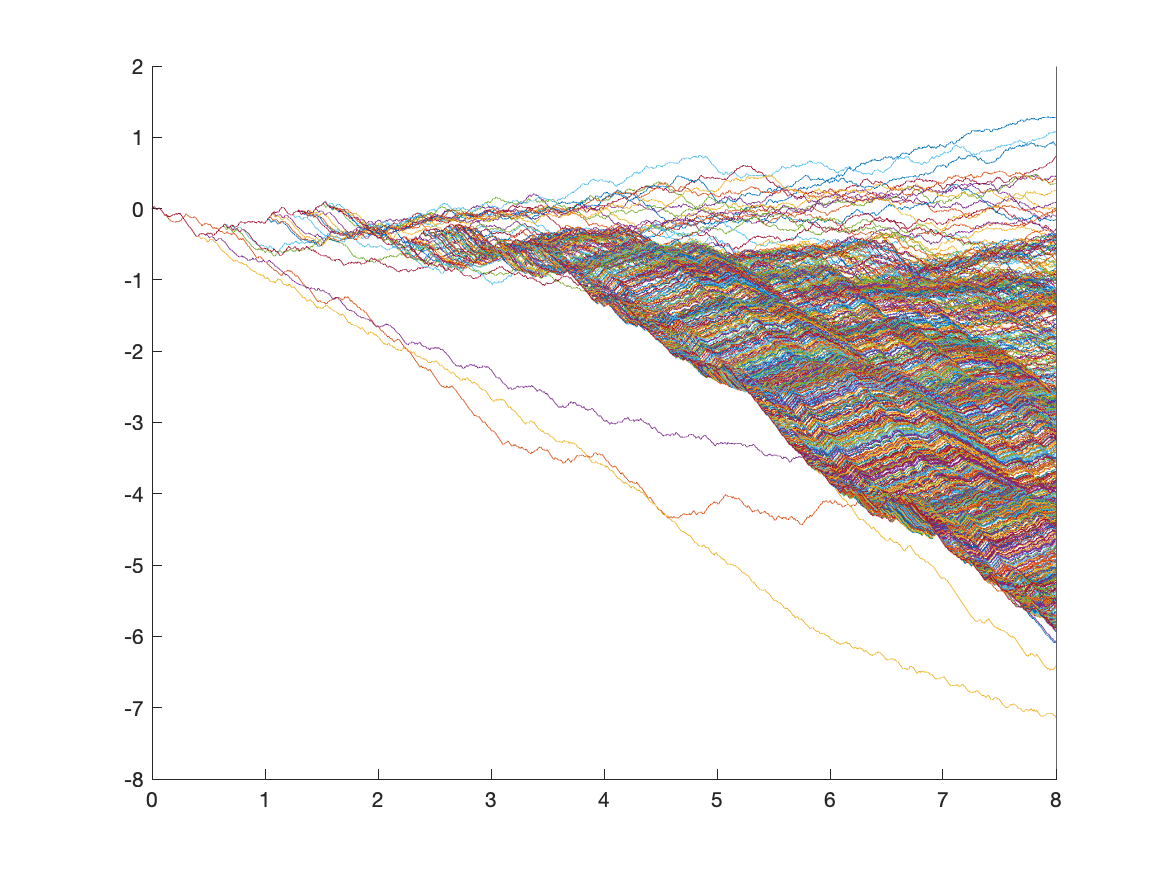}
    \caption{This shows a simulation of the branching fractional Brownian motion with Hurst parameter $H=0.85$. The $y$-axis is measured in units of $\left( \sum_{l\geq 0}q_l^2 \right)^{-\frac12}$ for $q_l$ defined by \eqref{defqn} and $\alpha = H-\tfrac12$.
    The simulation is based on the  discrete approximation described in Section~\ref{sec:treeindexhsurns}
    }
    \label{fig:introBFBM}
\end{figure}
 The representation \eqref{eq:FBM2_introduction} lends itself for a construction of the $\mathfrak y$-indexed FBM. To this purpose, think of $\mathfrak y$ as the disjoint union  of countably many edges $\{h\} \times (\ell_h, r_h]$, where $h\in \bigcup_{n\in \N_0}\{0,1\}^n$ is the Ulam-Harris name (see \cite[Section~2.1]{harris} and Section~\ref{sec:treeindexhsurns}) of the edge and $(\ell_h, r_h]$ is the life time interval of the edge. For each $v = (h,t) \in \mathfrak y$, let $t(v):=t$.  In this way, the {\em ancestral lineage} $\mathfrak a(v) \subset \mathfrak y$  of each $v \in \mathfrak y$ is in isometric  correspondence  with $(-\infty, t(v)]$. Let $dW_{\mathfrak y}(v), \, v \in \mathfrak y$, be a standard Gaussian white noise on $\mathfrak y$. Then the {\em $\mathfrak y$-indexed FBM}  has the representation
 \begin{equation}
   \label{eq:BFBM2_introduction}
  B_{\mathfrak y}^H(v):= \int_{\mathfrak a(v)} K(t(u),t(v)) \dif{W_{\mathfrak y}(u)}, \quad v \in \mathfrak y.
\end{equation}
This construction is implicit in work by Adler and Samorodnitsky, see \cite{Adler}, who introduced branching fractional Brownian as a starting point for their construction of super fractional Brownian motion as a process of measures on the path space that can be seen as a high density limit of branching fractional Brownian motion.  For related approaches to tree-indexed processes with memory see \cite{treeGaussianLLN},  who studied the long-term behaviour of their cloud of particles with a focus on the bulk. 
 \\

Our paper has two main contributions. Firstly, in Theorem~\ref{th:dbfbm}  we will construct discrete approximations of $\mathfrak y$-indexed FBM $B_{\mathfrak y}^H$ (and thus also the branching fractional Brownian motion) for $H > \frac 12$ via a discrete approximation. This approach is based on the power law P\'olya's urn of Hammond and Sheffield  \cite{HS} and a recent analysis of the random genealogy that underlies this urn, see \cite{IW}.

In a nutshell, the approximation of $B^H$ works as follows. For $\alpha := H-\frac 12$, consider  a family $\left(R_i\right)_{i\in\Z}$ of independent $\N$-valued random variables  with distribution
\begin{equation}
  \mu\left( \left\{n, n+1, \ldots \right\} \right)=n^{-\alpha}, \qquad n\in\N.
\end{equation}
The types (say $-1$ or $+1$) of the individuals  $i\in \Z$ are determined recursively:  Each individual $i$ selects $i-R_i$ as its parent and inherits the parent's type. As it turns out, this leads to a forest of countably many trees, see Section~\ref{sec:HSdesc}. If each of these trees is assigned the type $-1$ or $+1$ by a fair coin tossing, 
it further turns out that the rescaled sums over the types of the individuals $1,\ldots, \lfloor tn \rfloor $, $t\ge 0$, converge to fractional Brownian motion with Hurst parameter $H$ as $n \to \infty$. 

The proof of Theorem~\ref{th:dbfbm}, which will be carried out in  Section~\ref{sec:proofthdbfbm}, reveals interesting analytic identities. While the representation \eqref{eq:FBM2_introduction} ( as well as \eqref{eq:BFBM2_introduction} ) has the flavour of a moving average, the Hammond-Sheffield approximations of FBM is more in the spirit of an autoregression. Indeed, as an aside of our proof of Theorem~\ref{th:dbfbm}, we will obtain a ``microscopic'' interpretation of the prediction formula of Gripenberg and Norros \cite{fBMCondEW}.

The simulation depicted in Figure~\ref{fig:introBFBM} is based on the discrete approximation provided by Theorem~\ref{th:dbfbm}.

Our second main contribution in this paper is the analysis of the speed of the right-most particle of branching fractional Brownian motion,  in the spirit of McKean's celebrated work \cite{mckean1975} about branching Brownian motion, which was later substantially refined by Bramson  \cite{bramson} and Lalley and Sellke \cite{sellke}.
Specifically we will show that 
  the maximum $M^H(t)$ of a BFBM 
  with Hurst parameter $H>\tfrac12$ behaves asymptotically like
  \begin{equation}\label{eq:mtdef2intro}
   m(t):= 
    t^{H+\frac12} \sqrt{ \frac{ \sqrt{\pi} 2^{2H+1} H  }{ \Gamma(1-H)  \Gamma\left(H+\frac12\right) \left(H+\frac12\right)^2 } }  
  \end{equation}
  in the sense that for all $\varepsilon>0$
  \begin{equation}
   \mathcal P \left(\left\vert \frac{ M^H(T) }{m(t) } -1 \right\vert > \varepsilon \right) \rightarrow 0 \quad \mbox { for }t\rightarrow \infty.
  \end{equation}
 See 
 Figure~\ref{fig:mt} for an illustration of $m(t)$ for $t$ fixed and varying $H$. 
 Note that for $H=\tfrac12$ $m(t)$ equals $\sqrt{2}t$, which is consistent with the results about branching Brownian motion.\
 \begin{figure}[h]
 \begin{center}
     \includegraphics[scale=0.5]{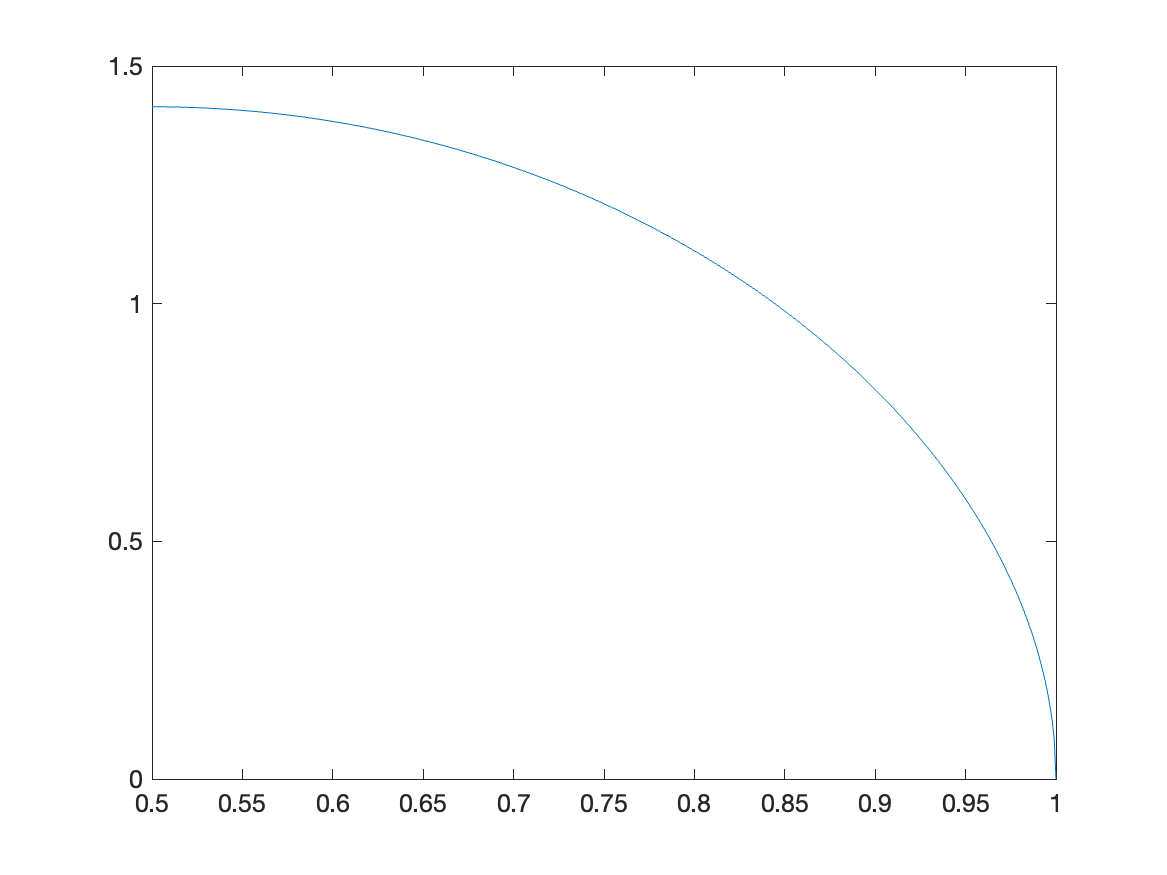}
 \end{center}
    
    \caption{This is an illustration of the behaviour of $m(t)/t^{H+\frac12}$ as a function of~$H$.}
    \label{fig:mt}
\end{figure}

Up to the constant factor in \eqref{eq:mtdef2intro}, the leading order $t^{H+\frac 12}$ of the maximum can be easily explained. Indeed, for a number $\lfloor e^t\rfloor$ of   independent normally distributed random variables with variance $t^{2H}$ the leading order of the maximum would be given by 
$$ \mathfrak{m}(t) := \sqrt 2 \,\, t^{H+\frac12 } $$
as suggested by the estimate 
$$  e^t  \WS\left( \mathcal N(0, t^{2H}) \geq \mathfrak{m}(t) \right) \approx e^t  \exp\left(-\frac{\mathfrak{m}(t)^2}{2t^{2H}} \right) =1. $$\FloatBarrier
Due to correlations we obtain \eqref{eq:BFBM2_introduction} which only differs by the constant prefactor depending on $H$. 
The statement on $M^H(t)$ is proved in Theorem~\ref{eq:THmax2}, which in turn is strongly connected to Theorem~\ref{eq:THmax}, see Remark~\ref{eq:remlog2}. 
The proof techniques of these theorems rely heavily on exploring the connection
between branching random walks and the so-called generalised random energy models (GREM). Those models were introduced and discussed by Derrida in the 1980's in \cite{Derrida}, \cite{Derrida1985AGO} and \cite{PhysRevB.24.2613}. An introduction to those can be found in Lecture Notes by Kistler \cite{MR3380419}. 

The connection between GREMs and branching random walks has first been explored by Arguin, Bovier and Kistler in \cite{MR2838339} for branching Brownian motion. In contrast to Brownian motion fractional Brownian motion is a process with memory. The corresponding GREM is then a continuous GREM with decreasing variance. 
Of such a GREM, Bovier and Kurkova analyse the leading order of the maximum in \cite{GREM2}. Their results lead to Theorem~\ref{eq:THmax}, where the underlying tree is a deterministic binary branching tree, see Section~\ref{sec:THmaxproof}. For subleading orders work by Maillard and Zeitouni \cite{slowdownbbm} allows some conjectures in our regime, see Remark~\ref{rem:secondorder}. For the proof of Theorem~\ref{eq:THmax2}, where the underlying tree is a Yule tree, we rely on methods developed first in \cite{MR2838339} and later in \cite{MR3358969} by Kistler and Schmidt. 
In Remark~\ref{remark:history} we give a more detailed account of (some of) the literature on the GREM, with an emphasis on work that is relevant for the ideas that are used in the present paper.
%

Our proof of Theorem~\ref{eq:THmax2} is conceptual in the sense that it makes use of the connection between GREMs
and branching random walks, Section~\ref{sec:THmax2proof}.  As pointed out to us by an anonymous reviewer of a previous version of this paper, 
Theorem~\ref{eq:THmax2} can also be deduced in an expedited manner from known results on branching Brownian motion by using the representation \eqref{eq:BFBM2_introduction} and the Payley-Wiener partial integration formula, see Remark~\ref{rem:refereeproof}.

%
%
%
%

\section{Preliminaries}
\subsection{The Hammond-Sheffield random walk
}\label{sec:HSdesc}
We start by briefly recalling the urn model of \cite{HS} along the lines of \cite{IW}. 

For $\alpha \in \left(0, \tfrac12 \right)$ and a slowly varying function $L: \R \rightarrow \R^+$ let $\mu:=\mu_{\alpha, L}$ be a probability measure on $\N$ having regularly varying tails
\begin{equation}\label{basic}
  \mathbf \mu(\{n,n+1,\ldots\}) \sim n^{-\alpha} L(n),
\end{equation}
and let $R$ be an $\mathbb N$-valued random variable with distribution $\mu$.

Let $\mathcal G_\mu$ be a random directed graph  with vertex set $\Z$ and edge set $E\left( \mathcal G_\mu \right)$ generated in the following way: Let $\left(R_i\right)_{i\in\Z}$ be a family of independent copies of $R$. The random set of edges $E\left( \mathcal G_\mu \right)$ is then given by
\begin{equation}
  E\left( \mathcal G_\mu \right):= \left\{ (i, i- R_i): i\in\Z \right\}.
\end{equation}
 This induces the (random) equivalence relation
\begin{equation}
  i \sim j :\Longleftrightarrow i \mbox{ and }j \mbox { belong to the same connected component of } \mathcal G_\mu.
\end{equation}
If $\alpha>1/2$ there is only one connected component and this relation is trivial. 

The individuals' types arise as follows: Assume that each component of $\mathcal G_\mu$ gets its type by an independent copy of a real-valued random variable $Y$ with 
\begin{equation}\label{fourth}
  \mathbf E[Y] = 0 \mbox{ and } 0< \mathbf E[Y^4] < \infty. \end{equation}
In the situation of  \cite{HS}, $Y$ is a centered Rademacher$(p)$ variable, i.e.
\begin{equation}
  \label{Rade} Y = \xi-(2p-1) \mbox{ with } \WS(\xi=+1)=p,\, \WS(\xi=-1)=1-p.
\end{equation}
Denote by $\mathscr C_i$ the component which contains $i$ (note that $\mathscr C_i = \mathscr C_j $ if $i\sim j$). For $i\in \mathbb Z$ the type of the component  $\mathscr C_i$ will be denoted by $Y_i$. 
Define the ``random walk'' (with dependent increments)
\begin{equation}\label{eq:Sndef}
  S_n := \sum_{i=1}^n Y_i,  \quad n=0,1,\ldots.
\end{equation}
Let us analyze the covariance structure of this random walk to motivate its relation to fractional Brownian motion: By construction,
\begin{equation} \label{sigman}\sigma_n^2 := \mathbf {Var}[S_n] = \sum_{i,j \in [n]} \mathbf{Cov}[Y_i,Y_j] = \mathbf{E}[Y^2] \, \sum_{i,j \in [n]} \mathbf P(i \sim j).
\end{equation}
\cite[Lemma 3.1]{HS} shows by Fourier and Tauberian arguments that 
\begin{equation} \label{asvar}
  \sum\limits_{i,j \in [n]} \mathbf P(i\sim  j) \sim \frac{C_1}{\alpha(2\alpha+1)} \cdot \frac{n^{2\alpha +1}}{L(n)^2} \quad \mbox{ as } n\to \infty;
\end{equation}
with $C_1$ as in \eqref{eq:defc1}, see \cite[(1.8)]{IW}. 

For $\frac {i-1}n \le t \le \frac{i}n$, $i,n \in \mathbb N$,  let
\begin{equation}\label{eq:defSnintro}
  \mbox{$S^{(n)}(t)$ be the linear interpolation of $S_i/\sigma_n$ and $S_{i+1}/\sigma_n$.}
\end{equation}
Because $(S_n)$ has stationary increments by construction, it follows from \eqref{sigman} and \eqref{asvar} that $S^{(n)}$ has, asymptotically as $n \to \infty$, the covariance structure of fractional Brownian motion with Hurst parameter $H:=\tfrac12+\alpha$. In order to prove that $S^{(n)}$ converges (in the sense of finite dimensional distributions) to 
fractional Brownian motion, it is shown in \cite{IW}  that the finite dimensional distributions of $S^{(n)}$ are asymptotically Gaussian. This is provided by \cite[Theorem~1.1]{IW}, see also \cite[Theorem~1.1]{HS}.\\

The renewal function,
\begin{equation}\label{defqn}
  q_n := \mathbf P \left( \tilde R_1+\ldots+ \tilde R_j =n \mbox{ for some }j\geq 0 \right)
\end{equation}
with $\tilde R_1, \tilde R_2,\ldots$ being  independent copies of $R$, an $\N$-valued random variable with distribution $\mu$ as in \eqref{basic},
is the main ingredient to describe the random graph $G_\mu$. As in \cite{IW} we will work under the condition 
\begin{equation}\label{qas}
  q_n \sim \frac{1}{\Gamma(\alpha)\Gamma(1-\alpha)} \frac{n^{\alpha -1}}{L(n)} \quad \mbox{as } n\rightarrow \infty,
\end{equation}
see \cite[Theorem~1.1(B)]{IW}. 
This is equivalent to the validity of the Strong Renewal Theorem for the renewal process with increment distribution \eqref{basic}, see Caravenna and Doney \cite{CaravennaDoney}, whose Theorem~1.4 gives necessary and sufficient conditions in terms of $\mu$ for the validity of \eqref{qas}. A well-known sufficient criterion for \eqref{qas} is Doney's \cite{D} criterion 
\begin{equation} \label{Doney}
  \sup_{n \ge  1} \frac{n\mathbf P(R = n)}{\mathbf P(R > n)}<\infty.
\end{equation}
We will work under the more restrictive assumption of $L\equiv 1$ and \eqref{Rade} with $p=\tfrac12$ to simplify notation. This means that we assign a type to each component of the random graph $G_{\mu}$ either $+1$ or $-1$ with probability $\tfrac12$.\\

Another key ingredient to analyse this model is the pair coalescence probability $\WS\left(0\sim n\right)$. 
Under the above assumption we get
\begin{equation}\label{eq:coalprob}
 \WS\left( 0 \sim n \right)\sim \frac{1}{ \sum_{l\geq 0}q_l^2 } \cdot   \frac{\Gamma(1-2\alpha)}{\Gamma(\alpha)\Gamma(1-\alpha)^3} \cdot n^{2\alpha-1} 
\end{equation}
for $n\rightarrow \infty$, see \cite[Proposition~2.1]{IW}.\normalcolor

\subsection{Tree-indexed Hammond-Sheffield-urns}\label{sec:treeindexhsurns}
 Now we introduce a branching version of the HS-model. Denote by $\mathcal Y$ the set of binary branching $\R$-trees with countably many branches. 
Let $\mathfrak y\in \mathcal Y$ (for example a realisation of Yule tree, in which every individual splits into two after an $\mathrm{Exp}(1)$-distributed time ), see Figure~\ref{fig:fig1}. We now recall some notation introduced in Section~\ref{sec:Intro} and add some more needed: We denote by $\mathcal B$ the collection of branches of the tree, such that $\mathfrak y:= \bigcup_{b\in \mathcal B} b$. 
 Assume that at a branch $b\in \mathcal B$ there is a branching event at time $s$, then we name this new branch ${(b,s)}$. For the sake of notational ease we shorten this to $bs$. The main branch is named $0$, see Figure~\ref{fig:fig1}. For two branches $b, \tilde b$ we denote by $b \wedge \tilde b$  the time when they split, for example $0\wedge 0rs=r $ and $0r \wedge 0rs = s$, as depicted in Figure~\ref{fig:fig1}. The point on the branch $b$ at time $t$ on the tree is denoted by $(b,t)_{\mathfrak y}$. If $bs$ and $b$ are two branches, we say that $bs\backslash s=b$.
 We say that the branch $b$ is older than the branch $\tilde b$ if the last digit ( $\in \R$ ) of $\tilde b$ is bigger than the last digit of $b$. For example, $0$ is the oldest branch and $b$ is older than $bs$.
 
%
  \begin{figure}[h]
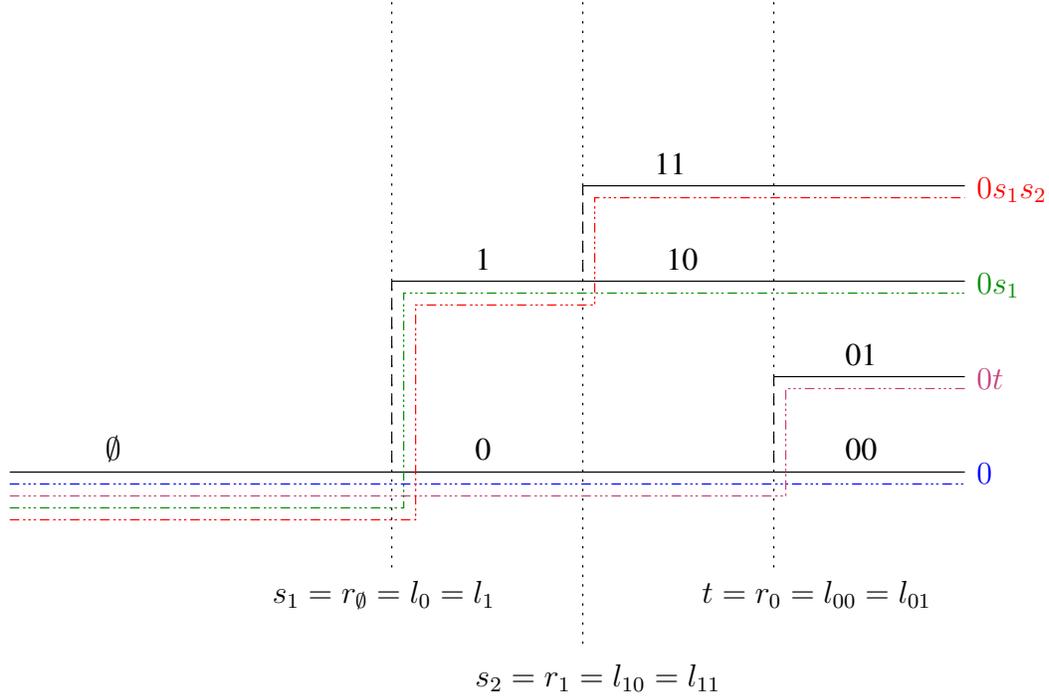

   \begin{center}
 \hspace*{-1.1cm}{\pgfkeys{/pgf/fpu/.try=false}%
\ifx\XFigwidth\undefined\dimen1=0pt\else\dimen1\XFigwidth\fi
\divide\dimen1 by 8424
\ifx\XFigheight\undefined\dimen3=0pt\else\dimen3\XFigheight\fi
\divide\dimen3 by 5424
\ifdim\dimen1=0pt\ifdim\dimen3=0pt\dimen1=3946sp\dimen3\dimen1
  \else\dimen1\dimen3\fi\else\ifdim\dimen3=0pt\dimen3\dimen1\fi\fi
\tikzpicture[x=+\dimen1, y=+\dimen3]
{\ifx\XFigu\undefined\catcode`\@11
\def\temp{\alloc@1\dimen\dimendef\insc@unt}\temp\XFigu\catcode`\@12\fi}
\XFigu3946sp
\ifdim\XFigu<0pt\XFigu-\XFigu\fi
\definecolor{green1}{rgb}{0,0.56,0}
\clip(588,-6612) rectangle (9012,-1188);
\tikzset{inner sep=+0pt, outer sep=+0pt}
\pgfsetfillcolor{black}
\pgftext[base,left,at=\pgfqpointxy{5550}{-5625}] {\fontsize{12}{14.4}\usefont{T1}{ptm}{m}{n}$t=r_{0}=l_{00}=l_{01}$}
\pgfsetlinewidth{+7.5\XFigu}
\pgfsetstrokecolor{black}
\draw (4800,-3000)--(7200,-3000);
\draw (6000,-4200)--(7200,-4200);
\pgfsetdash{{+60\XFigu}{+60\XFigu}}{++0pt}
\draw (3600,-3600)--(3600,-4800);
\draw (6000,-4200)--(6000,-4800);
\draw (4800,-3000)--(4800,-3600);
\pgfsetdash{{+15\XFigu}{+45\XFigu}}{+15\XFigu}
\draw (4800,-1800)--(4800,-5925);
\draw (6000,-1800)--(6000,-5400);
\pgfsetdash{}{+0pt}
\draw (1200,-4800)--(7200,-4800);
\pgfsetstrokecolor{blue}
\pgfsetdash{{+60\XFigu}{+24\XFigu}{+15\XFigu}{+18\XFigu}{+15\XFigu}{+18\XFigu}{+15\XFigu}{+24\XFigu}}{+0pt}
\draw (1200,-4875)--(7200,-4875);
\pgfsetstrokecolor{magenta}
\draw (1200,-4950)--(6075,-4950)--(6075,-4275)--(7200,-4275);
\pgfsetstrokecolor{green1}
\draw (1200,-5025)--(3675,-5025)--(3675,-3675)--(7200,-3675);
\pgfsetstrokecolor{red}
\draw (1200,-5100)--(3750,-5100)--(3750,-3750)--(4875,-3750)--(4875,-3075)--(7200,-3075);
\pgfsetstrokecolor{white}
\draw (600,-1200)--(600,-6600)--(9000,-6600)--(9000,-1200)--(600,-1200);
\pgfsetstrokecolor{black}
\pgfsetdash{{+15\XFigu}{+45\XFigu}}{+15\XFigu}
\draw (3600,-1800)--(3600,-5400);
\pgfsetfillcolor{red}
\pgftext[base,left,at=\pgfqpointxy{7275}{-3075}] {\fontsize{12}{14.4}\usefont{T1}{ptm}{m}{n}$0s_1s_2$}
\pgfsetfillcolor{green1}
\pgftext[base,left,at=\pgfqpointxy{7275}{-3675}] {\fontsize{12}{14.4}\usefont{T1}{ptm}{m}{n}$0s_1$}
\pgfsetfillcolor{magenta}
\pgftext[base,left,at=\pgfqpointxy{7275}{-4275}] {\fontsize{12}{14.4}\usefont{T1}{ptm}{m}{n}$0t$}
\pgfsetfillcolor{blue}
\pgftext[base,left,at=\pgfqpointxy{7275}{-4875}] {\fontsize{12}{14.4}\usefont{T1}{ptm}{m}{n}$0$}
\pgfsetfillcolor{black}
\pgftext[base,left,at=\pgfqpointxy{1800}{-4725}] {\fontsize{12}{14.4}\usefont{T1}{ptm}{m}{n}$\emptyset$}
\pgftext[base,left,at=\pgfqpointxy{4125}{-3525}] {\fontsize{12}{14.4}\usefont{T1}{ptm}{m}{n}1}
\pgftext[base,left,at=\pgfqpointxy{4125}{-4725}] {\fontsize{12}{14.4}\usefont{T1}{ptm}{m}{n}0}
\pgftext[base,left,at=\pgfqpointxy{5325}{-3525}] {\fontsize{12}{14.4}\usefont{T1}{ptm}{m}{n}10}
\pgftext[base,left,at=\pgfqpointxy{5250}{-2925}] {\fontsize{12}{14.4}\usefont{T1}{ptm}{m}{n}11}
\pgftext[base,left,at=\pgfqpointxy{6450}{-4125}] {\fontsize{12}{14.4}\usefont{T1}{ptm}{m}{n}01}
\pgftext[base,left,at=\pgfqpointxy{6450}{-4725}] {\fontsize{12}{14.4}\usefont{T1}{ptm}{m}{n}00}
\pgftext[base,left,at=\pgfqpointxy{2850}{-5625}] {\fontsize{12}{14.4}\usefont{T1}{ptm}{m}{n}$s_1=r_{\emptyset}=l_{0}=l_{1}$}
\pgftext[base,left,at=\pgfqpointxy{4125}{-6150}] {\fontsize{12}{14.4}\usefont{T1}{ptm}{m}{n}$s_2=r_{1}=l_{10}=l_{11}$}
\pgfsetdash{}{+0pt}
\draw (3600,-3600)--(7200,-3600);
\endtikzpicture}%
   \end{center}
   
   \caption{ This figure contains a translation between the notations introduced in Section~\ref{sec:treeindexhsurns} and in the Introduction. The Ulam-Harris names of the edges are written in black. The birth and death times are $s_1, s_2$ and $t$. The names of the branches in our notation and the corresponding ancestral lines are displayed in multiple colours. See Remark~\ref{remark:branchnames} for a detailed explanation of the translation between the two notations.}
   \label{fig:antonnotation}
 \end{figure}
  \begin{remark}\label{remark:branchnames} Let us connect our notation (introduced in the paragraph above) with the classic Ulam Harris notation, which we discussed before Equation \eqref{eq:BFBM2_introduction}: The branch $0 s_1$ in Figure~\ref{fig:antonnotation} consists of the edges with Ulam-Harris names $f=\emptyset, g=1, h=(h_1, h_2)=(1,0)$ (which we write as 10). It branches of  from the main branch at time $s_1=l_{h_1}=l_1$.\\Generally speaking: A branch containing the edge with Ulam-Harris name $h=(h_1, h_2, \ldots, )$ splits-off at times  $$\left\{ l_{h_1, h_2, ..., h_k}: h_k=1 \right\}.$$
  Our notation consists of the branch times at which a 1 is added to the Ulam-Harris names of an edge. A branch $b$ consists then of the edges between branching events. 
  See Figure~\ref{fig:antonnotation} for a translation between those notations in a concrete example.
\end{remark}
\FloatBarrier
\begin{figure}[h]
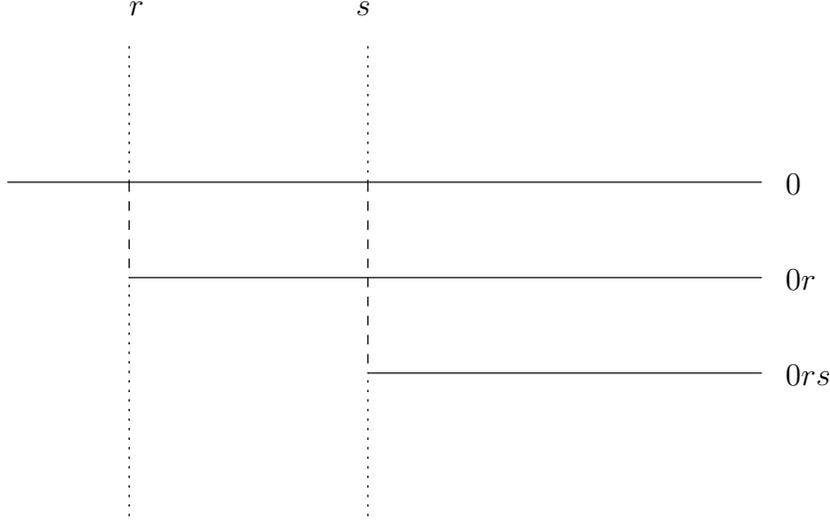

  \centering
  \scalebox{1}{{\pgfkeys{/pgf/fpu/.try=false}%
\ifx\XFigwidth\undefined\dimen1=0pt\else\dimen1\XFigwidth\fi
\divide\dimen1 by 5739
\ifx\XFigheight\undefined\dimen3=0pt\else\dimen3\XFigheight\fi
\divide\dimen3 by 3999
\ifdim\dimen1=0pt\ifdim\dimen3=0pt\dimen1=3946sp\dimen3\dimen1
  \else\dimen1\dimen3\fi\else\ifdim\dimen3=0pt\dimen3\dimen1\fi\fi
\tikzpicture[x=+\dimen1, y=+\dimen3]
{\ifx\XFigu\undefined\catcode`\@11
\def\temp{\alloc@1\dimen\dimendef\insc@unt}\temp\XFigu\catcode`\@12\fi}
\XFigu3946sp
\ifdim\XFigu<0pt\XFigu-\XFigu\fi
\clip(5673,-7062) rectangle (11412,-3063);
\tikzset{inner sep=+0pt, outer sep=+0pt}
\pgfsetfillcolor{black}
\pgftext[base,left,at=\pgfqpointxy{10575}{-6075}] {\fontsize{12}{14.4}\usefont{T1}{ptm}{m}{n}$0rs$}
\pgfsetlinewidth{+7.5\XFigu}
\pgfsetstrokecolor{black}
\draw (5685,-4800)--(10425,-4800);
\pgfsetdash{{+15\XFigu}{+45\XFigu}}{+15\XFigu}
\draw (6450,-5400)--(6450,-6900);
\pgfsetdash{{+60\XFigu}{+60\XFigu}}{++0pt}
\draw (6450,-4800)--(6450,-5400);
\draw (7950,-4800)--(7950,-6000);
\pgfsetdash{{+15\XFigu}{+45\XFigu}}{+15\XFigu}
\draw (6450,-3900)--(6450,-4800);
\pgfsetdash{}{+0pt}
\draw (10425,-5400)--(6450,-5400);
\draw (10425,-6000)--(7950,-6000);
\draw (6446,-5400)--(6454,-5400);
\draw (7946,-6000)--(7954,-6000);
\pgfsetdash{{+15\XFigu}{+45\XFigu}}{+15\XFigu}
\draw (7950,-3900)--(7950,-4800);
\draw (7950,-6000)--(7950,-6900);
\pgftext[base,left,at=\pgfqpointxy{6450}{-3750}] {\fontsize{12}{14.4}\usefont{T1}{ptm}{m}{n}$r$}
\pgftext[base,left,at=\pgfqpointxy{7875}{-3750}] {\fontsize{12}{14.4}\usefont{T1}{ptm}{m}{n}$s$}
\pgftext[base,left,at=\pgfqpointxy{10575}{-4875}] {\fontsize{12}{14.4}\usefont{T1}{ptm}{m}{n}$0$}
\pgftext[base,left,at=\pgfqpointxy{10575}{-5475}] {\fontsize{12}{14.4}\usefont{T1}{ptm}{m}{n}$0r$}
\pgfsetstrokecolor{white}
\pgfsetdash{}{+0pt}
\draw (11400,-3075)--(11400,-7050);
\endtikzpicture}%
}
  \caption{ 
   Three branches of the tree $\mathfrak y$ are shown. The ancestral branch is denoted by 0. The branch splitting off from the ancestral branch at time $r$ is denoted by $0r$, and the branch splitting off from the branch $0r$ at time $s$ is denoted by $0rs$. Time runs from left to right, and distances in $\mathfrak y$ are measured horizontally. Each branch $b$ is conceived as a copy of $\R$, with common ancestries being glued together.  
  }
  \label{fig:fig1}
\end{figure}
We are now ready to construct the $n$-th discrete approximation of branching $H$-fractional Brownian motion: 
\begin{enumerate}
    \item Sample a Yule tree $\mathfrak y$.
    \item Sample a HS-model over the discretisation of the branch $0$. More specifically, sample an HS-model over $\Z$, and identify the integer $i\in\Z$ with the point in the tree $(0,i/n)_{\mathfrak y}$. This is done as described in Section~\ref{sec:HSdesc},
    using the branch $0$ instead of the real numbers.
  \item We proceed recursively. Let the $R_{i}^{(b)}$ be independent and have distribution $\mu$. Note that each individual $(b,i)_{(n)}$ has an ancestor at distance $ R_i^{(b)} $ to the left to which it is connected regardless of the branch it lies in.\\
    Formally, assume that the HS model has been sampled on the oldest $m$ branches. Furthermore, assume that $b=s_1\ldots s_{m+1}$ is the $(m+1)$ oldest branch and that its last digit is $s_{m+1}>0$.
    For $i\in n\Z, i \geq \lfloor s_{m+1}n \rfloor$ set $$r:=\max\left\{r\leq m+1: (i-R_i^{(b)}-s_{m+1}n)/n + \lfloor n(s_m-s_r)\rfloor /n >0 \right\}\vee 0 .$$
    The point $\left(b,+i/n\right)_{\mathfrak y}$ is connected to  $\left( s_1\ldots s_r,(i-R_{i}^{(b)})/n\right)_{\mathfrak y}$.
   See Figure~\ref{fig:fig2}.
\end{enumerate}
\begin{figure}
  \centering
      \centering
  \subcaptionbox{An illustration of a branching HS-model and the multiple names an individual can have. \label{fig:fig2a}}{\scalebox{0.85}{{\pgfkeys{/pgf/fpu/.try=false}%
\ifx\XFigwidth\undefined\dimen1=0pt\else\dimen1\XFigwidth\fi
\divide\dimen1 by 7032
\ifx\XFigheight\undefined\dimen3=0pt\else\dimen3\XFigheight\fi
\divide\dimen3 by 5196
\ifdim\dimen1=0pt\ifdim\dimen3=0pt\dimen1=3946sp\dimen3\dimen1
  \else\dimen1\dimen3\fi\else\ifdim\dimen3=0pt\dimen3\dimen1\fi\fi
\tikzpicture[x=+\dimen1, y=+\dimen3]
{\ifx\XFigu\undefined\catcode`\@11
\def\temp{\alloc@1\dimen\dimendef\insc@unt}\temp\XFigu\catcode`\@12\fi}
\XFigu3946sp
\ifdim\XFigu<0pt\XFigu-\XFigu\fi
\pgfdeclarearrow{
  name = xfiga0,
  parameters = {
    \the\pgfarrowlinewidth \the\pgfarrowlength \the\pgfarrowwidth},
  defaults = {
	  line width=+7.5\XFigu, length=+120\XFigu, width=+60\XFigu},
  setup code = {
    \dimen7 2.15\pgfarrowlength\pgfmathveclen{\the\dimen7}{\the\pgfarrowwidth}
    \dimen7 2\pgfarrowwidth\pgfmathdivide{\pgfmathresult}{\the\dimen7}
    \dimen7 \pgfmathresult\pgfarrowlinewidth
    \pgfarrowssettipend{+\dimen7}
    \pgfarrowssetbackend{+-\pgfarrowlength}
    \dimen9 -0.5\pgfarrowlinewidth
    \pgfarrowssetvisualbackend{+\dimen9}
    \pgfarrowssetlineend{+-0.5\pgfarrowlinewidth}
    \pgfarrowshullpoint{+\dimen7}{+0pt}
    \pgfarrowsupperhullpoint{+-\pgfarrowlength}{+0.5\pgfarrowwidth}
    \pgfarrowssavethe\pgfarrowlinewidth
    \pgfarrowssavethe\pgfarrowlength
    \pgfarrowssavethe\pgfarrowwidth
  },
  drawing code = {\pgfsetdash{}{+0pt}
    \ifdim\pgfarrowlinewidth=\pgflinewidth\else\pgfsetlinewidth{+\pgfarrowlinewidth}\fi
    \pgfpathmoveto{\pgfqpoint{-\pgfarrowlength}{0.5\pgfarrowwidth}}
    \pgfpathlineto{\pgfqpoint{0pt}{0pt}}
    \pgfpathlineto{\pgfqpoint{-\pgfarrowlength}{-0.5\pgfarrowwidth}}
    \pgfusepathqstroke
  }
}
\clip(3210,-9012) rectangle (10242,-3816);
\tikzset{inner sep=+0pt, outer sep=+0pt}
\pgfsetfillcolor{black}
\pgftext[base,left,at=\pgfqpointxy{6450}{-3975}] {\fontsize{12}{14.4}\usefont{T1}{ptm}{m}{n}$sn$}
\pgfsetlinewidth{+7.5\XFigu}
\pgfsetcolor{black}
\filldraw  (7800,-5400) circle [radius=+75];
\filldraw  (9000,-5400) circle [radius=+75];
\filldraw  (9000,-6600) circle [radius=+75];
\filldraw  (9000,-7800) circle [radius=+75];
\filldraw  (6525,-6600) circle [radius=+75];
\filldraw  (5400,-5400) circle [radius=+75];
\filldraw  (7800,-6600) circle [radius=+75];
\pgfsetroundcap
\pgfsetdash{{+15\XFigu}{+45\XFigu}}{+15\XFigu}
\filldraw  (6525,-5400) circle [radius=+75];
\pgfsetbuttcap
\pgfsetdash{}{+0pt}
\draw (7800,-7800)--(9000,-7800);
\pgfsetdash{{+15\XFigu}{+45\XFigu}}{+15\XFigu}
\draw (9000,-4200)--(9000,-9000);
\draw (7800,-4200)--(7800,-6600);
\pgfsetstrokecolor{white}
\pgfsetarrows{[line width=7.5\XFigu]}
\pgfsetarrowsend{xfiga0}
\draw (10200,-4200)--(10200,-8400);
\pgfsetstrokecolor{black}
\pgfsetarrowsend{}
\draw (4200,-4200)--(4200,-9000);
\pgfsetdash{}{+0pt}
\draw (4200,-5400)--(4800,-5400)--(9000,-5400);
\pgfsetdash{{+15\XFigu}{+45\XFigu}}{+15\XFigu}
\draw (6525,-4200)--(6525,-9000);
\draw (5400,-4200)--(5400,-5400);
\pgfsetdash{}{+0pt}
\draw (5400,-6600)--(9000,-6600);
\pgfsetdash{{+60\XFigu}{+60\XFigu}}{++0pt}
\draw (5400,-5400)--(5400,-6600);
\draw (7800,-6600)--(7800,-7800);
\pgfsetdash{{+15\XFigu}{+45\XFigu}}{+15\XFigu}
\draw (5400,-6600)--(5400,-9000);
\draw (7800,-7800)--(7800,-9000);
\pgftext[base,left,at=\pgfqpointxy{4050}{-3975}] {\fontsize{12}{14.4}\usefont{T1}{ptm}{m}{n}$0$}
\pgftext[base,left,at=\pgfqpointxy{5400}{-3975}] {\fontsize{12}{14.4}\usefont{T1}{ptm}{m}{n}$rn$}
\pgftext[base,left,at=\pgfqpointxy{3225}{-5250}] {\fontsize{12}{14.4}\usefont{T1}{ptm}{m}{n}$(0,0)=(0r, 0)=\ldots$}
\pgftext[base,left,at=\pgfqpointxy{7800}{-3975}] {\fontsize{12}{14.4}\usefont{T1}{ptm}{m}{n}$tn$}
\pgftext[base,left,at=\pgfqpointxy{8850}{-3975}] {\fontsize{12}{14.4}\usefont{T1}{ptm}{m}{n}$un$}
\pgftext[base,left,at=\pgfqpointxy{7875}{-5250}] {\fontsize{12}{14.4}\usefont{T1}{ptm}{m}{n}$(0,tn)$}
\pgftext[base,left,at=\pgfqpointxy{9075}{-5250}] {\fontsize{12}{14.4}\usefont{T1}{ptm}{m}{n}$(0,rn)$}
\pgftext[base,left,at=\pgfqpointxy{7050}{-6375}] {\fontsize{12}{14.4}\usefont{T1}{ptm}{m}{n}$(0r, rn)=(0rt, tn)$}
\pgftext[base,left,at=\pgfqpointxy{9150}{-6675}] {\fontsize{12}{14.4}\usefont{T1}{ptm}{m}{n}$(0r, un)$}
\pgftext[base,left,at=\pgfqpointxy{9150}{-7875}] {\fontsize{12}{14.4}\usefont{T1}{ptm}{m}{n}$(0rt, un)$}
\pgftext[base,left,at=\pgfqpointxy{5700}{-6900}] {\fontsize{12}{14.4}\usefont{T1}{ptm}{m}{n}$(0r, sn)=(0rt, sn)$}
\pgftext[base,left,at=\pgfqpointxy{6600}{-5700}] {\fontsize{12}{14.4}\usefont{T1}{ptm}{m}{n}$(0, sn)$}
\pgftext[base,left,at=\pgfqpointxy{5025}{-5250}] {\fontsize{12}{14.4}\usefont{T1}{ptm}{m}{n}$(0, rn)=(0r, rn)$}
\pgfsetdash{}{+0pt}
\filldraw  (4200,-5400) circle [radius=+75];
\endtikzpicture}%
}}
  
\centering
\centering
  \subcaptionbox{An illustration of a branching HS-model in which we follow three ancestral lines. Two of them coalesce. Red and purple belong to the same component. \label{fig:fig2b}}{\scalebox{0.85}{{\pgfkeys{/pgf/fpu/.try=false}%
\ifx\XFigwidth\undefined\dimen1=0pt\else\dimen1\XFigwidth\fi
\divide\dimen1 by 4977
\ifx\XFigheight\undefined\dimen3=0pt\else\dimen3\XFigheight\fi
\divide\dimen3 by 5196
\ifdim\dimen1=0pt\ifdim\dimen3=0pt\dimen1=3946sp\dimen3\dimen1
  \else\dimen1\dimen3\fi\else\ifdim\dimen3=0pt\dimen3\dimen1\fi\fi
\tikzpicture[x=+\dimen1, y=+\dimen3]
{\ifx\XFigu\undefined\catcode`\@11
\def\temp{\alloc@1\dimen\dimendef\insc@unt}\temp\XFigu\catcode`\@12\fi}
\XFigu3946sp
\ifdim\XFigu<0pt\XFigu-\XFigu\fi
\catcode`\@11
\pgfutil@ifundefined{pgf@pattern@name@xfigp0}{
\pgfdeclarepatternformonly{xfigp0}
{\pgfqpoint{-1bp}{-1bp}}{\pgfqpoint{9bp}{5bp}}{\pgfqpoint{8bp}{4bp}}
{	\pgfsetdash{}{0pt}\pgfsetlinewidth{0.45bp}
	\pgfpathqmoveto{-1bp}{4.5bp}\pgfpathqlineto{9bp}{-0.5bp}
	\pgfusepathqstroke
}
}{}
\catcode`\@12
\pgfdeclarearrow{
  name = xfiga0,
  parameters = {
    \the\pgfarrowlinewidth \the\pgfarrowlength \the\pgfarrowwidth},
  defaults = {
	  line width=+7.5\XFigu, length=+120\XFigu, width=+60\XFigu},
  setup code = {
    \dimen7 2.15\pgfarrowlength\pgfmathveclen{\the\dimen7}{\the\pgfarrowwidth}
    \dimen7 2\pgfarrowwidth\pgfmathdivide{\pgfmathresult}{\the\dimen7}
    \dimen7 \pgfmathresult\pgfarrowlinewidth
    \pgfarrowssettipend{+\dimen7}
    \pgfarrowssetbackend{+-\pgfarrowlength}
    \dimen9 -0.5\pgfarrowlinewidth
    \pgfarrowssetvisualbackend{+\dimen9}
    \pgfarrowssetlineend{+-0.5\pgfarrowlinewidth}
    \pgfarrowshullpoint{+\dimen7}{+0pt}
    \pgfarrowsupperhullpoint{+-\pgfarrowlength}{+0.5\pgfarrowwidth}
    \pgfarrowssavethe\pgfarrowlinewidth
    \pgfarrowssavethe\pgfarrowlength
    \pgfarrowssavethe\pgfarrowwidth
  },
  drawing code = {\pgfsetdash{}{+0pt}
    \ifdim\pgfarrowlinewidth=\pgflinewidth\else\pgfsetlinewidth{+\pgfarrowlinewidth}\fi
    \pgfpathmoveto{\pgfqpoint{-\pgfarrowlength}{0.5\pgfarrowwidth}}
    \pgfpathlineto{\pgfqpoint{0pt}{0pt}}
    \pgfpathlineto{\pgfqpoint{-\pgfarrowlength}{-0.5\pgfarrowwidth}}
    \pgfusepathqstroke
  }
}
\definecolor{green1}{rgb}{0,0.56,0}
\clip(4035,-9012) rectangle (9012,-3816);
\tikzset{inner sep=+0pt, outer sep=+0pt}
\pgfsetarrows{[line width=7.5\XFigu]}
\pgfsetarrowsend{xfiga0}
\pgfsetlinewidth{+7.5\XFigu}
\pgfsetdash{}{+0pt}
\pgfsetstrokecolor{blue}
\draw (8550,-5400) arc[start angle=+33.69, end angle=+146.31, radius=+676];
\draw (7425,-5400) arc[start angle=+8.1, end angle=+171.9, radius=+265.2];
\draw (6900,-5400) arc[start angle=+33.69, end angle=+146.31, radius=+676];
\draw (5775,-5400) arc[start angle=+39.47, end angle=+140.53, radius=+825.9];
\pgfsetstrokecolor{red}
\draw (8100,-6600) arc[start angle=+-59.12, end angle=+-120.88, radius=+2119.1];
\draw (5925,-6600) arc[start angle=+270.95, end angle=+171.42, radius=+1044.4];
\pgfsetstrokecolor{magenta}
\draw (8775,-7800) arc[start angle=+-22.2, end angle=+-157.8, radius=+364.4];
\draw (8100,-7800) arc[start angle=+-98.02, end angle=+-175.68, radius=+1311.7];
\draw (6975,-6600) arc[start angle=+-27.38, end angle=+-152.62, radius=+591.2];
\pgfsetdash{}{+0pt}
\pgfsetstrokecolor{green1}
\draw (8475,-7800) arc[start angle=+-71.97, end angle=+-167.52, radius=+3266.3];
\pgfsetdash{}{+0pt}
\pgfsetcolor{blue}
\filldraw  (8550,-5400) circle [radius=+75];
\filldraw  (7425,-5400) circle [radius=+75];
\filldraw  (6900,-5400) circle [radius=+75];
\filldraw  (5775,-5400) circle [radius=+75];
\filldraw  (4500,-5400) circle [radius=+75];
\pgfsetcolor{red}
\filldraw  (8100,-6600) circle [radius=+75];
\filldraw  (5925,-6600) circle [radius=+75];
\filldraw  (4875,-5400) circle [radius=+75];
\pgfsetcolor{magenta}
\filldraw  (8775,-7800) circle [radius=+75];
\filldraw  (8100,-7800) circle [radius=+75];
\filldraw  (6975,-6600) circle [radius=+75];
\pgfsetdash{}{+0pt}
\pgfsetstrokecolor{green1}
\pgfsetfillpattern{xfigp0}{green1}
\draw[pattern,preaction={fill=green1}]  (8475,-7800) circle [radius=+75];
\draw[pattern,preaction={fill=green1}]  (4275,-5400) circle [radius=+75];
\pgfsetstrokecolor{black}
\pgfsetarrowsend{}
\draw (7800,-7800)--(9000,-7800);
\draw (5400,-6600)--(9000,-6600);
\pgfsetdash{{+15\XFigu}{+45\XFigu}}{+15\XFigu}
\draw (9000,-4200)--(9000,-9000);
\draw (7800,-4200)--(7800,-6600);
\draw (4200,-4200)--(4200,-9000);
\pgfsetdash{}{+0pt}
\draw (4200,-5400)--(4800,-5400)--(9000,-5400);
\pgfsetdash{{+15\XFigu}{+45\XFigu}}{+15\XFigu}
\draw (5400,-4200)--(5400,-5400);
\pgfsetdash{{+60\XFigu}{+60\XFigu}}{++0pt}
\draw (5400,-5400)--(5400,-6600);
\draw (7800,-6600)--(7800,-7800);
\pgfsetdash{{+15\XFigu}{+45\XFigu}}{+15\XFigu}
\draw (5400,-6600)--(5400,-9000);
\draw (7800,-7800)--(7800,-9000);
\pgfsetfillcolor{black}
\pgftext[base,left,at=\pgfqpointxy{8850}{-3975}] {\fontsize{12}{14.4}\usefont{T1}{ptm}{m}{n}$tn$}
\pgftext[base,left,at=\pgfqpointxy{4050}{-3975}] {\fontsize{12}{14.4}\usefont{T1}{ptm}{m}{n}$0$}
\pgftext[base,left,at=\pgfqpointxy{5400}{-3975}] {\fontsize{12}{14.4}\usefont{T1}{ptm}{m}{n}$rn$}
\pgftext[base,left,at=\pgfqpointxy{7800}{-3975}] {\fontsize{12}{14.4}\usefont{T1}{ptm}{m}{n}$sn$}
\endtikzpicture}%
}}
  
\caption{}\label{fig:fig2}
\end{figure}
 Note that the above procedure produces a random graph with almost surely infinitely many connected components, see Figure~\ref{fig:fig2b}.
We assign a type $\pm1$ to each component independently with probability $\frac12$. By $(bs, k)_{(n)}=(bs, k/n)_{\mathfrak y}$ we denote individual number $k\in\Z$  
in the branch $bs$, which branched off from branch $b$ at time $s$, in the $n$-th discrete approximation.
 We will mostly omit the subscripts and just write $(bs,k)$. 
If $k\leq \left\lfloor sn \right\rfloor$ this is the corresponding individual in branch $b$, if $k> \left\lfloor sn \right\rfloor$ it is the $\left[k-\left\lfloor sn \right\rfloor\right]$-th individual after the branch point at which $bs$ branched off  $b$. Especially this means that $(b, 0)$ corresponds to the same individual for all branches $b\in \mathcal B$, see Figure~\ref{fig:fig2a}. 
For an individual $(b,k)$ we denote by $Y_{(b,k)}$ the type of its component, where we often will omit the $b$ if $b=0$. \\

The random walk for the $n$-th discrete approximation along the main branch $0$ is defined by
\begin{equation}\label{eq:Snbranch0}
%
  S_{0}^{(n)}(t):=\frac{1}{c(n)}\left[\sum_{l=1}^{ \left\lfloor tn \right\rfloor } Y_l+ \left[ tn-\left\lfloor tn \right\rfloor  \right] Y_{ \left\lfloor tn \right\rfloor  } + \left[  \left\lceil tn \right\rceil - tn \right]Y_{ \left\lceil tn \right\rceil }\right], \qquad t\geq 0
\end{equation}
for the scaling function
\begin{equation}
  c(n):= \left(  n^{2\alpha+1} \frac{1}{\sum_{l\geq 0}q_l^2}\cdot  \frac{1}{\alpha (2\alpha+1)}\cdot   \frac{ \Gamma(1-2\alpha) }{\Gamma(\alpha)\Gamma(1-\alpha)^3} \right)^{\frac12}.
\end{equation}
Now define the random walk along a branch $bs$ inductively by
\begin{equation}\label{eq:HSRWbranching}
  S_{bs}^{(n)}(t):= \mathbf{1}_{t\leq s}  S_{b}^{(n)}(t) +  \mathbf{1}_{t> s} \left[S_{b}^{(n)}(s) + \tfrac{ \sum_{l=  \lfloor ns \rfloor +1 }^{ \left\lfloor tn \right\rfloor }Y_{(bs,l)} + \left[ tn-\left\lfloor tn \right\rfloor  \right] Y_{ \left(bs, \left\lfloor tn \right\rfloor  \right)} + \left[ \left\lceil tn \right\rceil - tn \right]Y_{\left(bs,   \left\lceil tn \right\rceil\right) } }{c(n)}\right].
\end{equation}
 Observe that this means that for two branches $b$ and $\tilde b$ the processes $S^{(n)}_b$ and $S^{(n)}_{bs}$ are equal till $b\wedge \tilde b$ and share some common memory afterwards.
  \begin{figure}[h]
  \centering
    \subcaptionbox{$n=10$\label{subfigure:1}}{\includegraphics[scale=0.65]{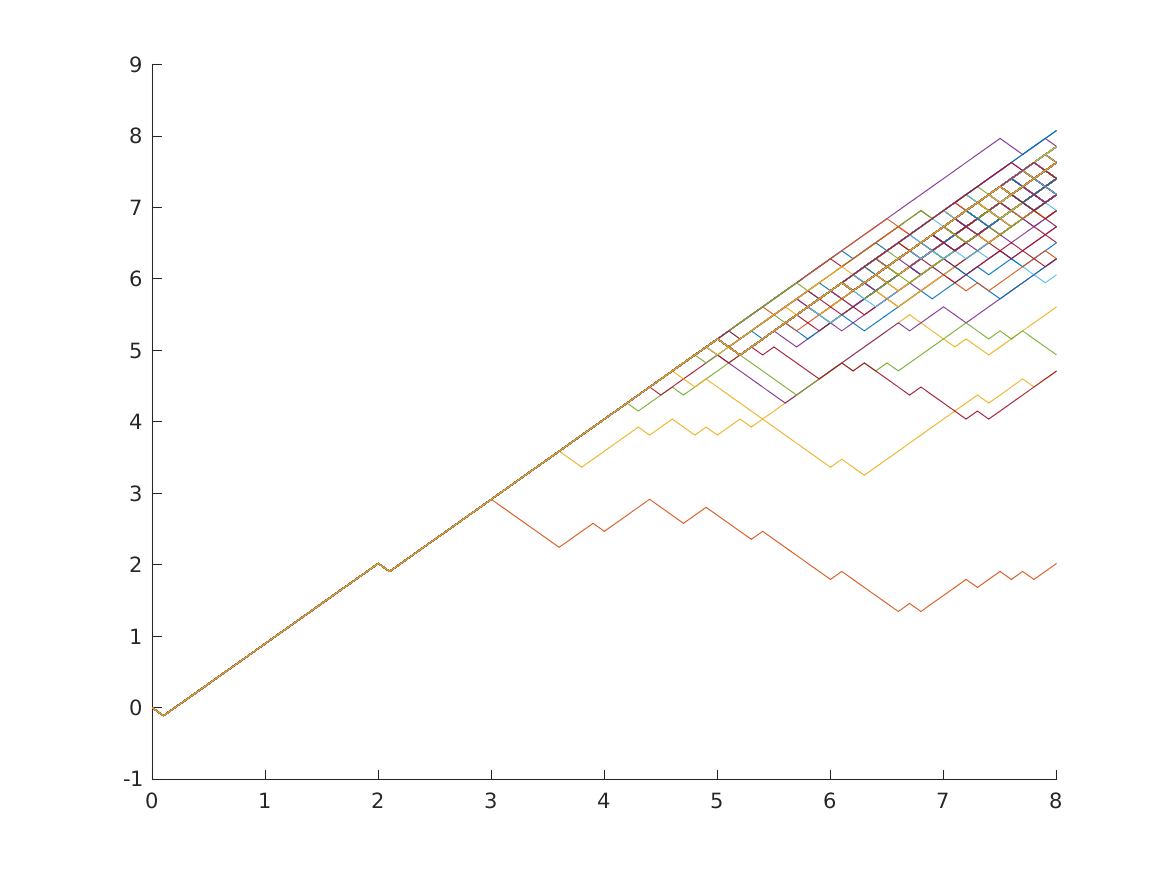}}
 \subcaptionbox{$n=300$\label{subfigure:2}}{
    \includegraphics[scale=0.65]{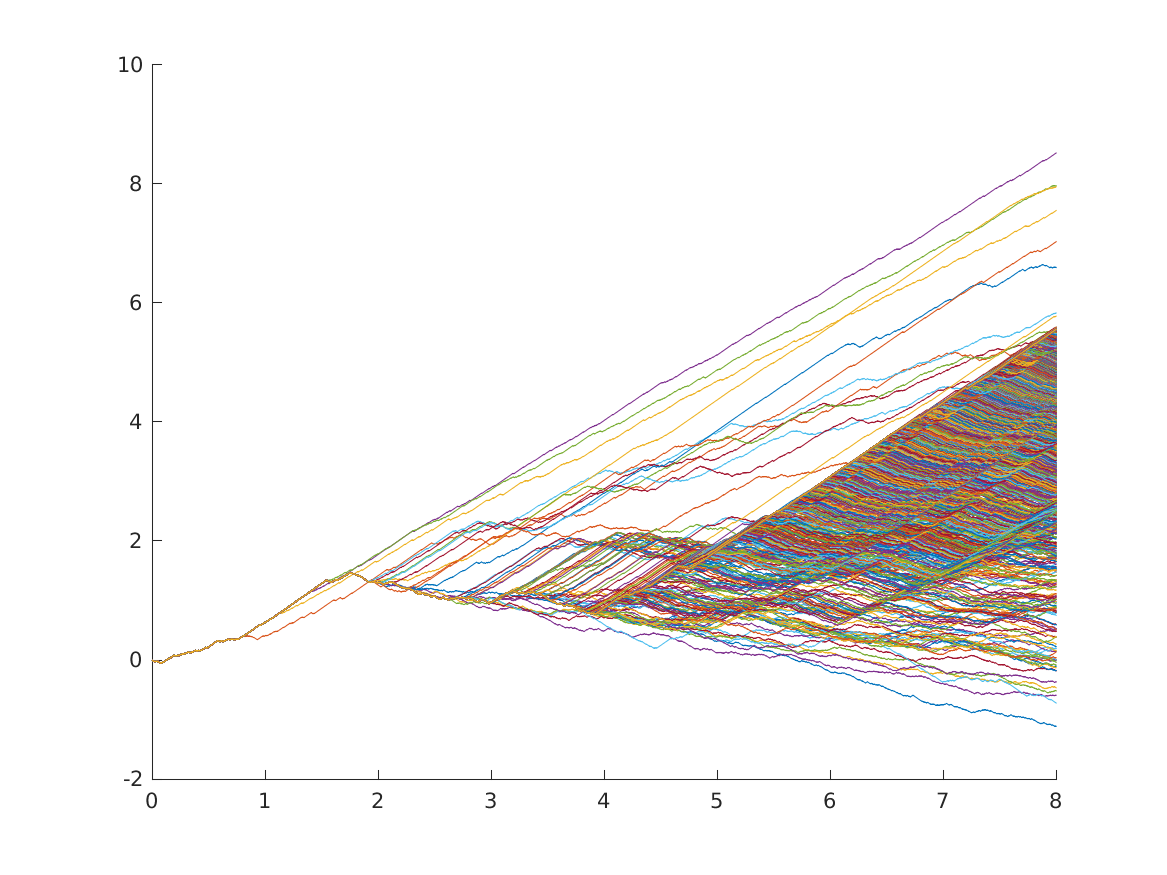}}
  %
  \caption{This is a simulation of $S_{\mathfrak y}^{(n)}:= \left( \left( S_b(t) \right)_{0\leq t \leq 8} \right)_{b\in \mathcal B}$ for $\alpha=0.45$. The $y$-axis is measured in units of $\left( \sum_{l\geq 0}q_l^2 \right)^{-\frac12}$ for $q_l$ defined by \eqref{defqn} and $\alpha = H-\tfrac12$. }
  \label{fig:sT}
\end{figure}
See Figure~\ref{fig:sT} for a simulation using different values of $n$.\\

From now on always denote by $\WS_{ \mathfrak y }, \EW_{ \mathfrak y }$ the law of the  $\mathfrak y$-indexed random walks given the tree $\mathfrak y$. \
Thus we immediately get the following:
\begin{proposition}\label{prop:iwtheorem11cor}
  Let $\mathfrak y \in \mathcal Y$. Then  the sequence of processes $\big(S_{b}^{(n)}\big)_{n\geq 1}$ converges for all branches $b\in \mathcal B$ in distribution to fractional Brownian motion with Hurst-parameter $H=\frac12+ \alpha$ as $n\rightarrow \infty$.
\end{proposition}
This follows directly by \cite[Theorem~1.1]{IW} since each branch itself is a HS-model. 
A first step in showing the validity of our discrete construction is the following proposition.
\begin{proposition}\label{prop:condGauss} 
  Let $\mathfrak y\in \mathcal Y$. For not necessarily different branches $b_1, \ldots, b_m \in \mathcal B$ and $t_1, \ldots, t_m \in \R$ the vector
  \begin{equation}
    \left( S^{(n)}_{b_i}(t_i) \right)_{i=1,\ldots, m}
  \end{equation} 
  is asymptotically jointly Gaussian as $n\rightarrow \infty$. 
\end{proposition}
Proposition~\ref{prop:condGauss} will be shown by the results of Section~\ref{sec:branchingtindexed}~and~\ref{sec:proofthdbfbm}, which contain a notion of "branching property" for processes with memory.
%
%
%
\FloatBarrier
  \section{Main results}
  In this section we will describe the main contributions of this paper, which include a discrete approximation of branching fractional Brownian motion and the derivation of the speed of its maximum. 
  
  %
\subsection{Tree-indexed fractional Brownian motion}\label{sec:bfbm}
Remember the classical construction of fractional Brownian motion already stated in the Introduction, see \eqref{eq:FBM2_introduction}:
Given a Brownian motion $W$, Gaussian white noise on $\R$, we construct a fractional Brownian Motion with Hurst parameter $H$ by setting
\begin{equation}\label{eq:FBM2}
  B(t):= \int_{\R} K(s,t) \dif{W(s)} 
\end{equation}
with
\begin{equation}\label{eq:kernelKst}
  K(s,t):= K^H(s,t):= \frac{1}{C_H}\left[ \mathbf{1}_{ s\leq 0 } \left( (t-s)^{H-\frac12}-(-s)^{H-\frac12} \right) + \mathbf{1}_{t\geq s>0} (t-s)^{H-\frac12} \right], 
\end{equation}
and
\begin{equation}\label{eq:fbmconst}
  C_H:= \left( - \frac{ 2^{-2H} \Gamma(-H) \Gamma\left(H+\frac12\right) }{\sqrt{\pi}}  \right)^{\frac12},
\end{equation}
see \cite[Definition~2.1, Corollary~3.4]{mandelbrotvanness},  where this process is called \emph{reduced fractional Brownian motion}.    Note, that the normalization corresponds to the one used by Samorodnitsky and Taqqu in \cite[Chapter~7.2]{stbook} and differs from the one used in \cite{mandelbrotvanness}. 
In the same way we can proceed and construct on top of branching Brownian motion $\left(\left( W_{b}(t) \right)_{[0,T]}\right)_{b\in \mathcal B}$, Gaussian white noise on an $\R$-tree $\mathfrak y\in \mathcal T$, fractional branching Brownian motion:\FloatBarrier
\begin{definition} Let $\left(\left( W_{b}(t) \right)_{[0,T]}\right)_{b\in \mathcal B}$be Gaussian white noise on an $\R$-tree $\mathfrak y\in \mathcal Y$. The process
    $\left(\left( B_{b}(t) \right)_{[0,T]}\right)_{b\in \mathcal B}$ defined by
\begin{equation}\label{eq:bbfbmkernel}
  B_{b}(t):= \int_{\R} K(s,t) \dif{W_b(s)}
\end{equation}
is called \emph{$\mathfrak y$-indexed branching fractional Brownian motion} (BFBM).
\end{definition}
 Note that the above definition corresponds to the process defined in \eqref{eq:BFBM2_introduction} and the construction of \cite[Section~6]{treeGaussianLLN}.
If one now computes the covariance
\begin{equation}
  \Cov\left[ B_{b}\left(t_1\right), B_{\tilde b}\left(t_2\right) \right]
\end{equation}
one obtains for branches $b$ and $\tilde b$ with $b\wedge \tilde b = s$
\begin{eqnarray}
  &&\rho^{\mathrm{K}}\left(t_1, t_2, s\right)\nonumber\\
  &:=& \EW_{ \mathfrak y }\left[ B_{b}(t_1)B_{\tilde b}(t_2) \right]\nonumber\\
  &=&\frac{1}{C_H^2}\EW_{ \mathfrak y }\left[ \left( \int_{-\infty}^0 (t_1-\xi)^{H-\frac12}-(-\xi)^{H-\frac12} \dif{B_{b}(\xi)} + \int_0^{t_1} (t_1-\xi)^{H-\frac12} \dif{B_b(\xi)}  \right) \right.\nonumber\\
  &&\left. \cdot \left( \int_{-\infty}^0 (t_2-\xi)^{H-\frac12}-(-\xi)^{H-\frac12} \dif{B_{\tilde b}(\xi)} + \int_0^{t_2} (t_2-\xi)^{H-\frac12} \dif{B_{\tilde b}(\xi)}  \right)\right]\nonumber\\
  &=&  \frac{1}{C_H^2}\left(\int_{-\infty}^0 \left((t_1-\xi)^{H-\frac12}-(-\xi)^{H-\frac12}  \right)\left((t_2-\xi)^{H-\frac12}-(-\xi)^{H-\frac12}  \right) \dif{\xi} \right) \nonumber\\
  && + \frac{1}{C_H^2}\left( \int_0^{s} (t_1-\xi)^{H-\frac12} (t_2-\xi)^{H-\frac12} \dif{\xi}   \right).\label{eq:rhoKtt}
\end{eqnarray}
 For $t_1=t_2 \equiv t > s$ this gives
\begin{eqnarray}
  &&\rho^{\mathrm{K}}\left( t, t, s \right)\nonumber\\
  &=&  \frac{1}{C_H^2} \left(\int_{-\infty}^0 \left((t-\xi)^{H-\frac12}-(-\xi)^{H-\frac12}  \right)^2 \dif{\xi} \right) + \frac{1}{C_H^2} \left( \int_0^{s} (t-\xi)^{2H-1}  \dif{\xi}   \right)\nonumber\\
  &=& t^{2H} -(t-s)^{2H} \frac{ \sqrt{\pi} 2^{2H-1} }{ \Gamma(1-H)  \Gamma\left(H+\frac12\right) }=:t^{2H} -C_{\rho}(t-s)^{2H} \label{eq:formcov}
\end{eqnarray}
 Observe that the prefactor one of $t^{2H}$ in \eqref{eq:formcov} is due to the normalization used in \eqref{eq:kernelKst}.

\subsection{A discrete approximation via tree-indexed Hammond-Sheffield random walks}
\label{sec:bfbmdiscrete}
Let us now discuss the discrete approximation \eqref{eq:HSRWbranching} to branching fractional Brownian motion: \FloatBarrier
\begin{theorem}\label{th:dbfbm}
  Let $T>0$. 
  Let $\mathfrak y \in \mathcal Y$. Then
  \begin{equation}
   S_{\mathfrak y}^{(n)}:= \left( \left( S_b^{(n)}(t) \right)_{ [0,T] } \right)_{b\in \mathcal B}
  \end{equation}
  converges in distribution to BFBM.
\end{theorem}
Note that $S_{\mathfrak y}^{(n)}$ is a random function from the $\R$-tree $\mathfrak y$ into the real numbers $\R$. The above convergence in distribution occurs with respect to the topology induced by uniform convergence. 

The proof can be found in Section~\ref{sec:proofthdbfbm}.\\

Since again pair coalescence probabilities are a main ingredient to understand the structure of the model and its limit we state a simple proposition similar to \cite[Proposition~2.1]{IW}:
\begin{proposition}\label{prop:coalbranches}
  For branches $b$ and $\tilde b$ with $b\wedge \tilde b = s$
  the probability that $\left(b,i\right)$ and $(\tilde b, j)$ (which for $i,j>sn$ lie in the two different branches) lie in the same component of the \emph{branching} HS-model is given by 
  \begin{equation}
    \WS_{ \mathfrak y }\left( \left(b,i\right) \sim (\tilde b, j) \right) = \frac{1}{\sum_{l\geq 0}q_l^2} \sum_{ r \geq (i\wedge j)-sn } q_r q_{r+|j-i|} =\frac{1}{\sum_{l\geq 0}q_l^2} \sum_{ r\geq 0 } q_{ i-sn +r }q_{j-sn+r} 
  \end{equation}
  for $i,j> sn$. 
\end{proposition}
The proof can be found in \appD.
This can now be used to compute the covariance structure of the discrete approximations in Theorem~\ref{th:dbfbm}: 
 For branches $b, \tilde b$ with $b\wedge \tilde b=s$ we have
  \begin{eqnarray}
    \Cov\left[ S_{b}^{(n)}\left(t_1\right), S_{\tilde b}^{(n)}\left(t_2\right)  \right] 
    &\sim&  \frac12 \left[ t_1^{2\alpha+1} - (t_1-s)^{2\alpha+1}
         +     t_2^{2\alpha+1} -  (t_2-s)^{2\alpha+1}\right] \nonumber\\
    &&+   \frac{ \int\limits_0^{t_1-s} \int\limits_0^{t_2-s} \int\limits_0^\infty (y_3+y_1)^{\alpha-1} (y_3+y_2)^{\alpha-1} \dif{y_3} \dif{y_2} \dif{y_1} }{ \Gamma(1-2\alpha) \Gamma(\alpha) \Gamma(1-\alpha)^{-1} (\alpha (2\alpha+1))^{-1} }.\nonumber\\
       &=:&\rho^{\mathrm{HS}}\left(t_1, t_2, s\right)\label{eq:covhsbranching}
  \end{eqnarray}
  This formula and its proof, which can be found in \appB, help to identify the sources of positive correlation. 
  We will show that
  the covariance structure  obtained with the discrete approximations and the kernel construction are indeed the same, see Corollary~\ref{cor:uniquedistGauss}. This implies that the two constructions of BFBM lead to the same object.
  
  In particular it turns out that for $H\equiv \alpha + \frac12$
\begin{equation}\label{eq:covequality}
  \rho^{\mathrm{K}}\left(t_1, t_2, s\right) = \rho^{\mathrm{HS}}\left(t_1, t_2, s\right)=: \rho \left(t_1, t_2, s\right)
\end{equation}
for $\rho^{\mathrm{K}}\left(t_1, t_2, s\right)$ defined by  \eqref{eq:rhoKtt}  and $\rho^{\mathrm{HS}}\left(t_1, t_2, s\right)$ defined by \eqref{eq:covhsbranching}.
This equality is elusive  to us   without the stochastic interpretation presented here.
Interesting analytic identities emerge of this and are the content of \appC.
\subsection{A prediction formula for fractional Brownian motion}
To better understand the connection between the representations  \eqref{eq:HSRWbranching} and \eqref{eq:FBM2} let us go back to fractional Brownian motion and provide additional insight on the following result by Gripenberg and Norros:
\begin{proposition}[{\cite[Theorem~3.1]{fBMCondEW}}]\label{prop:condew}
  For a fractional Brownian motion $B$ with Hurst parameter $\tfrac12 < H < 1$ for all $t>0$
  \begin{equation}\label{eq:propcondewintegral}
    \EW\left[ B_t | \sigma\left(B_s, s\leq 0  \right) \right] = \int_{-\infty}^0 g(t, s) \dif{B_s}
  \end{equation}
  with
  \begin{equation}\label{eq:gkernelfbmCondEW}
    g(t, -s) := \frac{ \sin\left(\pi \left(H-\frac12\right)\right) }{\pi} t^{-H+\frac12} \int_0^t \frac{ \xi^{H-\frac12} }{\xi +s} \dif{\xi}.
  \end{equation}
\end{proposition}
 
\begin{remark}
   Note, that the integral in \eqref{eq:gkernelfbmCondEW} has no singularity at 0 since since $H>\frac12$. 
  For the same reason we have that for all $t\geq 0$ the integrals $\int_{-\infty}^0 |g(t,s)| \dif{s}$ and $\int_{-\infty}^0 |g(t,s)^2| \dif{s}$ are finite. By \cite[p.~404]{fBMCondEW} this is enough for the integral in \eqref{eq:propcondewintegral} to exist. 
\end{remark}
 
The proof technique in \cite{fBMCondEW} consists in using the representation formulae \eqref{eq:FBM2}, see \cite[Proof of Theorem~3.1]{fBMCondEW}. 
The HS-model allows to understand the kernel $g$ in \eqref{eq:gkernelfbmCondEW} as the probability that a certain \emph{increment from the past will get copied into the present}. See Section~\ref{sec:condew} for a proof of Proposition~\ref{prop:condew} which makes use of this.\\

\subsection{The maximal displacement of branching fractional Brownian motion}\label{sec:maxbfbm}
In the following we identify the leading order of $ \max\limits_{b\in \mathcal B} B_{b}(t)$.
\begin{theorem}\label{eq:THmax} Let $\mathfrak y_{ \mathrm{bin} }$ be a deterministic binary branching tree (every branch branches into two after time 1, see Figure~\ref{fig:fig3}), and recall the definition of $C_H$ from \eqref{eq:fbmconst}. 
  The leading order of the maximum of a BFBM
  \begin{equation}
    \left( \left( B_{b}(t) \right)_{t\geq 0} \right)_{b\in\mathcal B}
  \end{equation}
  with Hurst parameter $H\in\left( \tfrac12,1 \right)$ is
  \begin{equation}\label{eq:mtdef}
    m(t):=
    t^{H+\frac12}\sqrt{ \frac{ \log(2)\sqrt{\pi} 2^{2H+1} H  }{ \Gamma(1-H)  \Gamma\left(H+\frac12\right) \left(H+\frac12\right)^2 } }
    = t^{H+\frac12} \cdot \frac{\sqrt{2\log(2)}}{C_H}\cdot \frac1{(H+\frac12)}
  \end{equation}
  in the sense that
  \begin{equation}
    \EW_{ \mathfrak y_{ \mathrm{bin} } }\left[ \frac{ \max\limits_{b\in \mathcal B} B_{b}(t)  }{m(t) }   \right] \rightarrow 1 \quad \mbox { for }t\rightarrow \infty.
  \end{equation}
\end{theorem}
See Section~\ref{sec:THmaxproof} for a proof. 
\begin{remark}\label{rem:newremarkanton}
    Arguments in the flavour of Arguin, Bovier and Kistler \cite{MR2838339}, which were continued in work by Kistler and Schmidt \cite{MR3358969}, see
Section~\ref{sec:THmax2proof}, give us a similar statement for branching fractional Brownian motion on Yule trees, see the next theorem. 
    The idea of the proof is to first  show that the maximum of a BFBM can only be attained by a trajectory staying very close to the maximum all along the way. For this purpose we will show exactly this for a series of Generalized random energy models (GREMs) approximating BFBM. Finally we observe that trajectories staying close to the maximum are indeed of the desired order.\\
    In Remark~\ref{rem:refereeproof} we give another proof of Theorem~\ref{eq:THmax2} which combines results of \cite{MR2838339}  with the  Mandelbrot-van-Ness-representation
of fractional Brownian motion in terms of Wiener integrals and  the Payley-Wiener partial integration formula. 
\end{remark}

\begin{theorem}\label{eq:THmax2}
  Let $\mathfrak Y$ be a (binary branching) Yule tree with branching rate $1$ and let $\eta$ denote its distribution. Then
  the leading order of the maximum of a BFBM
  \begin{equation}
    \left( \left( B_{b}(t) \right)_{t\geq 0} \right)_{b\in \mathcal B}
  \end{equation}
  with Hurst parameter $H\in \left( \tfrac12,1 \right)$ 
  is
  \begin{equation}\label{eq:mtdef2}
    m(t):= t^{H+\frac12} \sqrt{ \frac{ \sqrt{\pi} 2^{2H+1} H  }{ \Gamma(1-H)  \Gamma\left(H+\frac12\right) \left(H+\frac12\right)^2 } }
    = t^{H+\frac12} \cdot \frac{\sqrt{2}}{C_H}\cdot \frac1{(H+\frac12)}
  \end{equation}
 for $C_H$ defined by \eqref{eq:fbmconst}
  in the sense that
  \begin{equation}
    \bigintss_{\mathbb{T}} \WS_{ \mathfrak y }\left( \left \vert \frac{ \max\limits_{b\in \mathcal B} B_{b}(t)  }{m(t) } -1 \right\vert > \varepsilon \right)  \eta\left(\dif \mathfrak y\right) \rightarrow 0 \quad \mbox { for }t\rightarrow \infty
  \end{equation}
  and all $\varepsilon>0$. 
\end{theorem}
Note that this result means 
\begin{equation}
    \mathcal P \left( \left \vert \frac{ \max\limits_{b\in \mathcal B} B_{b}(t)  }{m(t) } -1 \right\vert > \varepsilon  \right) \rightarrow 0,
\end{equation}
where $\mathcal P$ averages over the Yule tree and the  fBM increments along the branches.\\

\begin{remark}\label{eq:remlog2}
    Observe that the only difference between the formulas \eqref{eq:mtdef2} and \eqref{eq:mtdef2} for the leading order of the maximum $m(t)$ is the factor $\sqrt{\log 2}$. This comes from the fact that the Yule tree  has of order $e^t$ many leaves at time $t$, while the corresponding quantity for the deterministic binary branching tree is  $2^t$.  While this explanation can be made rigorous using the ideas by Kistler and Schmidt \cite{MR3358969}, a proof of Theorem \ref{eq:THmax2} can be given along the following lines: In the situation of the Theorem, the distribution of $\max\limits_{b\in \mathcal B} B_{b}(t)$ can be approximated by a \emph{generalized random energy model} (GREM, see \cite{Derrida}) in which the variance along all branches is decaying. The latter fact together with the Gaussian increments along all branches gives that the \emph{optimal strategy} is to always follow the maximum path. This allows to compute the order of the maximal displacement directly. For details we refer to  Section~\ref{sec:THmax2proof}
\end{remark}

\begin{remark}
    We conjecture that the above results hold almost surely.
\end{remark}
\begin{remark}\label{rem:refereeproof}
The proof of Theorem~\ref{eq:THmax2}, which we will give in Section~\ref{sec:THmax2proof}, is conceptual in the sense that it relates the fractional branching Brownian motion rather directly 
with a generalised random energy model, see Remark~\ref{rem:newremarkanton}. Here we give an alternative proof, which makes use of the Mandelbrot-van-Ness-representation
of fractional Brownian motion in terms of Wiener integrals and transforms these via the Payley-Wiener partial integration formula into weighted Riemann-Integrals of the paths
of branching Brownian motion and controls their asymptotics using results from \cite{MR2838339}:
    Remember that we can represent branching fractional Brownian motion by the Wiener integral
    \begin{equation}\label{eq:FBM2}
  B_b(t):= \int_{\R} \frac{1}{C_H}\left[ \mathbf{1}_{ s\leq 0 } \left( (t-s)^{H-\frac12}-(-s)^{H-\frac12} \right) + \mathbf{1}_{t\geq s>0} (t-s)^{H-\frac12} \right]\dif{W_b(s)}.
\end{equation}
Now note that all branches have the integral over $\R_{\leq 0}$ in common. Writing 
\begin{equation}
    Z_t := \int_{\R_{-}} \frac{1}{C_H} \left( (t-s)^{H-\frac12}-(-s)^{H-\frac12} \right)  \dif{W_0(s)}
\end{equation}
we thus obtain
\begin{equation}\label{eq:btrepresentreferee}
     B_b(t)= Z_t+\int_{0}^t \frac{1}{C_H} (t-s)^{H-\frac12} \dif{W_b(s)}.
\end{equation}
By integration by parts
the latter integral is equal to 
\begin{equation}
   \int_{0}^t \frac{H-\frac12}{C_H} (t-s)^{H-\frac32} W_b(s)\dif{s},
\end{equation}
and similarly 
\begin{equation}
    Z_t = \int_{\R_-} \frac{H-\frac12}{C_H} \left( (t-s)^{H-\frac32}-(-s)^{H-\frac32} \right) W_0(s)\dif{s}.
\end{equation}
Since $W_0(s)$ is just an ordinary Brownian motion $Z_t$ is almost surely at most of order $t^{H}$. Let $\varepsilon>0$, then by \cite[Corollary~2.6]{MR2838339} (see also \cite[Theorem~3.8]{MR3382171} and \cite[Proposition~2.5]{MR3101852}) the integral in \eqref{eq:btrepresentreferee} fulfills
\begin{equation}
    \mathcal P\left( \left\vert \max\limits_{b\in \mathcal B} \frac{\int_{0}^t \frac{H-\frac12}{C_H} (t-s)^{H-\frac32} W_b(s)\dif{s} }{ \int_{0}^t \frac{H-\frac12}{C_H} (t-s)^{H-\frac32} \sqrt{2}s\dif{s}  } -1  \right\vert>\varepsilon\right) \rightarrow0 \mbox{ as }t\rightarrow\infty. 
\end{equation}
Substituting $u=\tfrac{s}{t}$ gives
\begin{eqnarray*}
    &&\frac{H-\frac12}{C_H}\int_{0}^t  (t-s)^{H-\frac32} \sqrt{2}s\dif{s}\\
    &=&t^{H-\frac12}\cdot \sqrt{2}\cdot \frac{H-\frac12}{C_H}\int_0^1  u (1-u)^{H-3/2}  \dif{u} \\
    &=& t^{H+\frac12}\cdot\sqrt{2}\cdot\frac{H-\frac12}{C_H}\int_0^1  (1-u) u^{H-3/2}  \dif{u}  \\
    &=& t^{H+\frac12}\cdot\sqrt{2}\cdot \frac{H-\frac12}{C_H}\left[\int_0^1 u^{H-3/2} \dif{u} - \int_0^1 u^{H-1/2} \dif{u}\right]  \\
    &=& t^{H+\frac12}\cdot\frac{\sqrt{2}}{C_H}\cdot \frac1{(H+\frac12)},
\end{eqnarray*}
which gives the desired result.  

 \cite[Corollary~2.6]{MR2838339} can also be used for a weak result in the second order term since it also gives that the maximum path of branching Brownian motion at time $t$ is of order $\sqrt{s} \wedge \sqrt{t-s} $ smaller than the maximum path at time $s$. So since 
\begin{equation}
    \int_0^t (t-s)^{H-\frac32} \left(\sqrt{s} \wedge \sqrt{t-s}\right)\dif{s}
\end{equation}
is of order $t^H$ this entails that
\begin{equation}
     \limsup\limits_{t\rightarrow\infty} \mathcal P\left(\frac{ \max\limits_{b\in \mathcal B} B_b(t)-t^{H+\frac12}\cdot \frac{\sqrt{2}}{C_H}\cdot \frac1{(H+\frac12)}  }{t^H } >-K\right) \rightarrow1 \mbox{ as }K\rightarrow\infty.
\end{equation}
The next remark suggests that this bound is presumably not sharp. 
\end{remark}
\begin{remark}\label{rem:secondorder}
In \cite{slowdownbbm}
  Maillard and Zeitouni continue the study on the distribution of the maximum of branching Brownian motion with time-inhomogeneous variance, which corresponds to the study of the time-inhomogeneous F-KPP equation. (See also \cite{zeitounifirst}  and \cite{zeitounimiddle} by Fang and Zeitouni as discussed in Remark~\ref{remark:history}.)  They prove that the maximum particle of a branching Brownian motion $W^{\sigma^2, T}$ on $[0,T]$ with time inhomogeneous strictly decreasing variance $\sigma^2\left(\tfrac{s}{T}\right)$ satisfying  $\sigma \in C^2, \inf_t \vert\sigma'(t)\vert>0, \sigma(1)>0$, which can be written as
  \begin{equation}
   W^{\sigma^2, T}(t)= \int_0^t \sigma\left( \frac{s}{T} \right) \dif{W_s} \mbox{ for standard Wiener noise }W\mbox{ and }0\leq t \leq T,
  \end{equation}
   behaves like 
\begin{equation}
    T \cdot \int_0^1 \sigma(s)\dif{s} - T^{\frac13}\cdot 2^{-\frac13} \alpha_1 \int_0^1 \sigma(s)^{\frac13} \left|\sigma'(s)\right|^{\frac23}\dif{s}- \log(T)\cdot \sigma(1),
\end{equation}
where $\alpha_1$ is the largest zero of the Airy function of the first kind.

We can naively fit branching fractional Brownian motion into their setting to obtain a candidate for the order of the second order term of the maximum of fractional branching Brownian motion:   Using Equation \eqref{eq:formcov} we can write 
\begin{eqnarray*}
    \rho(s,t,t)&=& t^{2H} \left[ 1- C_{\rho} \left( 1- \frac{s}{t}\right)^{2H}\right]\\
    &=& t^{2H} C_{\rho} \left[ 1-  \left( 1- \frac{s}{t}\right)^{2H}\right]+ \left[ 1- C_\rho\right] t^{2H}\\
    &=&  t^{2H} C_{\rho}\cdot\left[  \int_0^{s/t} \sigma^2(x)\dif{x} \right]+ \left[ 1- C_\rho\right] t^{2H}
\end{eqnarray*}
for
\begin{equation}
    \sigma^2 (x) =  2H (1-x)^{2H-1}.
\end{equation}
Let now $W^{\sigma^2, T}$ be a Brownian motion on $[0,T]$ with time inhomogeneous variance $\sigma^2\left(\tfrac{s}{T}\right)$ and let $Z$ be an independent normally distributed random variable. Then the above representation of $\rho$ gives the equality in distribution
\begin{equation}
    W^{\sigma^2, T}\left(T\right) + \sqrt{ 1- C_\rho} t^{H}Z  \overset{(d)}{=} B^H(t)
\end{equation}
for $t$ fixed and $T=t^{2H} C_{\rho} $.
This also holds for the branching systems, in the sense that for branching Brownian motion $\left( W^{\sigma^2, T}_b\right)_{b\in \mathcal B}$ with time inhomogeneous variance $\sigma^2\left(\tfrac{s}{T}\right)$ we obtain
\begin{equation}\label{eq:mzequal}
     \max\limits_{b\in \mathcal B} W^{\sigma^2, T}_b\left(T\right) + \sqrt{ 1- C_\rho} t^{H}Z  \overset{(d)}{=} \max\limits_{b\in \mathcal B} B^H(t)_b. 
   \end{equation}
   This is due to the fact that two branches in \cite{slowdownbbm} which have split at time $xT, x\in (0,1)$ have covariance $T \int_0^x \sigma^2(u)\dif{u}$. 
We conjecture that the results  by Maillard and Zeitouni, see \cite[Theorem~1.1]{slowdownbbm}, give the second order term of the speed of the maximum of a BFBM even if the function $\sigma$ is not fulfilling the  conditions $\inf_t |\sigma'(t)|>0$,  $\sigma(1)>0$ and $\sigma \in C^2$ . By \eqref{eq:mzequal} This suggests that the second order term is of order $ t^{\frac23 H }$.  We do not believe that the third order term will be the same as in \cite[Theorem~1.1]{slowdownbbm}, where it is of logarithmic type.
\end{remark}
In the following we give a few remarks concerning the connection between branching random walks and the GREM; for more background see the Lecture Notes by Kistler \cite{MR3380419}.
\begin{remark}\label{remark:history}
A precursor (and a special case) of the GREM is the REM, Random Energy Model, introduced by Derrida in \cite{Derrida}, continued in \cite{PhysRevB.24.2613}.
This is the Gaussian random field $\left(X_{\sigma}\right)_{\sigma =1,\ldots,2^N}$ for independent $X_\sigma$.
An overview and analysis of this can be found in lecture notes by Bolthausen and Sznitman, see \cite{MR1890289}.\\
In \cite{Derrida1985AGO} Derrida then introduced the GREM, a REM with multiple levels, which introduced a hierarchical structure.
For example a GREM with two levels is a correlated random field
$$\left(X_{\sigma}\right)_{\sigma\in \left\{ 1,\ldots, 2^{N/2} \right\}\times \left\{ 1,\ldots, 2^{N/2} \right\} }$$ with
$$X_\sigma \equiv Y_{\sigma_1}^{(1)}+Y_{\sigma_1, \sigma_2}^{(2)}$$
for independent collections of independent random variables $$\left(Y_{\sigma_1}^{(1)}  \right)_{\sigma_1 \in \left\{ 1,\ldots , 2^{N/2} \right\}}\qquad \mbox{ and } \qquad \left(Y_{\sigma_1, \sigma_2}^{(2)}  \right)_{\sigma_1 \in \left\{ 1,\ldots , 2^{N/2} \right\}, \sigma_2 \in \left\{ 1,\ldots , 2^{N/2} \right\}}.$$
The covariance of the random field then depends on the so called overlap. See \cite[Section~2.2]{MR3380419} fore a more formal introduction. 
An analysis of the leading order of the maximum and a general overview can be found in \cite{MR2761979}. \\
In \cite{MR2070334} Bovier and Kurkova give a detailed analysis of GREMs with finitely many hierarchies, which is then extended to continuous hierarchies by \cite{GREM2}. The latter is used by us to prove Theorem~\ref{eq:THmax}. The analysis of \cite{MR2070334} 
 and \cite{GREM2} gives much finer results on the leading order than what one gets by just applying Slepian's Lemma \cite[Section~2.10]{slepian}, see Lemma~\ref{th:slepian}. 
\\
In \cite{MR2209333} Bolthausen and Kistler studied a non-hierarchical version of the GREM and showed that the GREM is in some sense able to overcome correlations in certain regimes since this behaves as a suitable constructed hierarchical GREM. In \cite{MR2531095} they studied the corresponding Gibbs measures. This analysis has been extended in recent work of Kistler and Sebastiani \cite{MR4600154}. 

The connection of the GREM to branching random walks, especially branching Brownian motion, got first explored by Arguin, Bovier and Kistler in \cite{MR2838339}.
They were able to analyse the full extremal process, including the maximal particle, the second maximal, and so on. They showed that extremal particles descend from ancestors that split either shortly after zero or just before the observed time.
In \cite{MR3129797} they continued the study and showed that the extremal process converges in law to a Poisson cluster point process. 
In \cite{MR3065863} the authors analysed the empirical distribution of the maximal displacement and showed that a  Gumbel distribution with a random shift occurs as a limit thus
proving a conjecture of Lalley and  Sellke. 

After this ground breaking work variations of BBM have been studied, for example by Bovier and Hartung in \cite{MR3164771} and \cite{MR3351476}.
 Fang and Zeitouni \cite{zeitounifirst} studied the lower orders of the maximum of a branching random walk with time-inhomogeneous variance on a deterministic binary branching tree, which is extended to the study of branching Brownian motion with time-inhomogenous variance in \cite{zeitounimiddle} by Fang and Zeitouni. 
 In more recent work \cite{slowdownbbm}
Maillard and Zeitouni then studied the maximum of BBM with decaying variance which is close to our setup. See Remark~\ref{rem:secondorder} for a more detailed explanation.

In Theorem~\ref{eq:THmax2} we consider an object that is closely related to a GREM, with the difference being that 
underlying tree in this object is not a deterministic binary tree but a Yule tree like in the BBM case. 

The GREM and its refinements have also found applications in topics beyond branching random walks, see
 \cite[Section~2]{MR3380419} for an outline of applications which include branching diffusions \cite{MR3382171}. A recent application is one to TAP equations by Kistler, Schmidt and Sebastiani \cite{MR4609283}. 
\end{remark}

\section{Proof of Proposition~\ref{prop:condew}: Conditioning  fractional Brownian motion on its past}\label{sec:condew}
We now give a conceptually new proof of Proposition~\ref{prop:condew} making use of the discrete approximation of fractional Brownian motion introduced by \cite{HS}.
\begin{proof}
  We set $\alpha = H-\frac12$. Let $S^{(n)}$ be the discrete approximation to fractional Brownian motion defined by \eqref{eq:defSnintro}.
  By $\mathcal F_{\leq 0}$ we denote the $\sigma$-algebra containing any information about $G_\mu \cap\Z_{\leq 0}$ and the colours of the individuals $i\in \Z_{\leq 0}$.  First observe that the family $\big( S^{(n)}(t)\big)_{n} $  for $t$ fixed  is uniformly integrable since $\big( S^{(n)}(t)\big)_{n} $ is bounded in $L_2$. \ 

%
%
Write
\begin{equation}
    E_n   :=  \bigcap_{m=1}^k \left\{ S^{(n)}(s_m)\in \mathfrak B_m\right\} 
\end{equation}
  for  $\mathfrak B_1, \mathfrak B_2,\ldots, \mathfrak B_k \in\mathcal{B}(\R), 0>s_1>s_2>\ldots > s_k$, then for $t> 0$
  \begin{equation}
    \lim\limits_{n\rightarrow\infty}\EW\left[ S^{(n)}(t)\mathbf{1}_{  E_n   } \right]
    = \EW\left[ B(t)\mathbf{1}_{B_{s_1}\in\mathfrak  B_1, \ldots, B_{s_k}\in\mathfrak B_k  }  \right]. 
  \end{equation} 
  by \cite[Corollary~1.2]{IW} and  uniform integrability.
  Now set
  \begin{equation}
    b_{n, -k}:= \WS\left( \max \left( A_n \cap \left\{ -\infty, \ldots, -2, -1 \right\} \right)=-k \right). 
  \end{equation}
  The event
  \begin{equation}
    \left\{ \max \left( A_n \cap \left\{ -\infty, \ldots, -2, -1 \right\} \right)=-k \right\}
  \end{equation}
   can be decomposed with respect to the last individual right of zero which still belongs to $A_n$, so
  \begin{equation}
    b_{n, -k}= \sum_{l=0}^{n-1} \WS\left( n-l \in A_n \right) \mu\left(k+n-l\right)=\sum_{l=1}^{n} \WS\left( l \in A_n \right) \mu\left(k+l\right).
  \end{equation}
  Using
  \begin{equation}
    \WS\left( 0 \in A_n \right) \overset{n\rightarrow \infty}{\sim} \frac{ 1 }{\Gamma\left(\alpha\right)\Gamma\left(1-\alpha\right)} \cdot n^{\alpha-1}
  \end{equation}
  and  choosing a $\mu$ satisfying 
  \begin{equation}
    \frac{ \mu(n) }{n^{-\alpha-1}}\sim \alpha \qquad \mbox{ for }n\rightarrow \infty
  \end{equation}
  we obtain for $n\rightarrow \infty$
  \begin{equation}
    b_{n, -k}=\sum_{l=1}^{n} \WS\left( l \in A_n \right) \mu\left(k+l\right)\sim\frac{\alpha}{\Gamma\left(\alpha\right)\Gamma\left(1-\alpha\right)} \sum_{l=1}^{n} \left(n-l\right)^{\alpha-1} \left( k+l \right)^{-\alpha-1}.
  \end{equation}
  For $k=\xi n, \xi>0$ we further get
  \begin{eqnarray*}
    b_{n, -k}&\overset{n\rightarrow \infty}{\sim}& \frac{\alpha}{\Gamma\left(\alpha\right)\Gamma\left(1-\alpha\right)}   \sum_{l=1}^{n} \left(n-l\right)^{\alpha-1} \left( k+l \right)^{-\alpha-1}\\
             &=& \frac{\alpha}{\Gamma\left(\alpha\right)\Gamma\left(1-\alpha\right)} n^{-1} \cdot \frac{1}{n} \sum_{l=1}^{n} \left( 1- \frac{l}{n} \right)^{\alpha-1} \left( \xi + \frac{l}{n} \right)^{-\alpha-1}\\
             &\overset{n\rightarrow \infty}{\sim}& \frac{\alpha}{\Gamma\left(\alpha\right)\Gamma\left(1-\alpha\right)n} \int_0^1 \left( 1-x \right)^{\alpha-1} \left(\xi+x\right)^{-\alpha-1} \dif{x}.
  \end{eqnarray*}
   Applying the substitution $u=\tfrac{1}{(x+\xi)^{\alpha}}$ we get the identity
  
  \begin{equation}
    \int_0^1 \left( 1-x \right)^{\alpha-1} \left(\xi+x\right)^{-\alpha-1} \dif{x}=  \frac{ \xi^{-\alpha} }{ \alpha + \alpha \xi }, 
  \end{equation}
   such that in total  
  \begin{equation}\label{eq:bnk}
    b_{n, -k}\overset{n\rightarrow \infty}{\sim} \frac{\alpha}{\Gamma\left(\alpha\right)\Gamma\left(1-\alpha\right)n}\cdot  \frac{ \xi^{-\alpha} }{ \alpha + \alpha \xi }.
  \end{equation}
  Using \eqref{eq:Snbranch0} we have 
  \begin{equation}
     c(n)  S^{(n)}(t)  = \sum_{l=1}^{ \lfloor nt \rfloor } Y_l +  \left[ tn-\left\lfloor tn \right\rfloor  \right] Y_{ \left\lfloor tn \right\rfloor  } + \left[  \left\lceil tn \right\rceil - tn \right]Y_{ \left\lceil tn \right\rceil }.
  \end{equation}
 If we now define
  \begin{equation}
     \sum_{l=1}^{  nt  } a_l:= \sum_{l=1}^{ \lfloor nt \rfloor } a_l +  \left[ tn-\left\lfloor tn \right\rfloor  \right] a_{ \left\lfloor tn \right\rfloor  } + \left[  \left\lceil tn \right\rceil - tn \right]a_{ \left\lceil tn \right\rceil }
   \end{equation}
   for a sequence $\left( a_l \right)_{l}$
 we can write
  \begin{equation}
     c(n)\EW\left[ S^{(n)}(t)\mathbf{1}_{  E_n  } \right]
    =  \EW\left[ \sum_{l=1}^{nt} \EW\left[ Y_l | \mathcal F_{\leq 0}  \right]  \mathbf{1}_{  E_n  } \right] .
  \end{equation}
  Observing
  \begin{equation}
    \EW\left[ Y_l | \mathcal F_{\leq 0}  \right] = \sum_{s\geq 0} Y_{-s} b_{l, - \frac{s}{l}\cdot l}
  \end{equation}
  gives us
  \begin{equation}
     c(n)\EW\left[ S^{(n)}(t)\mathbf{1}_{  E_n  } \right]
    =  \EW\left[ \mathbf{1}_{  E_n  }  \sum_{s\geq 0} \sum_{l=1}^{nt} Y_{-s} b_{l, - \frac{s}{l}\cdot l}  \right].
  \end{equation}
  If we now plug in \eqref{eq:bnk}, we obtain
  \begin{equation}
      \EW\left[ S^{(n)}(t)\mathbf{1}_{  E_n  } \right]
    \overset{n\rightarrow \infty}{\sim}
    \frac{1}{c(n)}\EW\left[ \mathbf{1}_{  E_n  }\sum_{s\geq 0}Y_{-s}  \sum_{l=1}^{nt} l^{-1} \frac{1}{\Gamma\left(\alpha\right)\Gamma\left(1-\alpha\right)} \cdot \frac{ \left( \frac{s}{l} \right)^{-\alpha} }{1+ \left( \frac{s}{l} \right)}  \right].
  \end{equation}
    Setting $\Delta S^{(n)}\left(-\tfrac{s}{n}\right):= S^{(n)}\left( -\tfrac{s}{n}\right)-S^{(n)}\left( -\tfrac{s-1}{n}\right)= \tfrac{1}{c(n)}Y_{-s}$ the right hand side can be rewritten  as
  \begin{eqnarray*}
    &&\frac{1}{c(n)} \EW\left[ \mathbf{1}_{  E_n  }\frac{1}{\Gamma\left(\alpha\right)\Gamma\left(1-\alpha\right)}\sum_{s\geq 0} Y_{-s} \sum_{l=1}^{nt} l^{\alpha} s^{-\alpha} \left(l+s\right)^{-1} \right]\\
    &=& \EW\left[ \mathbf{1}_{  E_n  } \frac{1}{\Gamma\left(\alpha\right)\Gamma\left(1-\alpha\right)}\sum_{s\geq 0}  \left( \frac{s}{n} \right)^{-\alpha} \Delta S^{(n)}\left(-\frac{s}{n}\right) \frac{1}{n} \sum_{l=1}^{nt} \left( \frac{l}{n} \right)^{\alpha} \left( \frac{l+s}{n} \right)^{-1} \right]. 
  \end{eqnarray*}
  Write 
  \begin{equation}
      f^{(n)}\left( \frac{s}{n}, t\right) :=  \left( \frac{s}{n} \right)^{-\alpha} \frac{1}{n} \sum_{l=1}^{nt} \left( \frac{l}{n} \right)^{\alpha} \left( \frac{l+s}{n} \right)^{-1}, \qquad 
      f\left( \frac{s}{n},t\right) :=  \left( \frac{s}{n} \right)^{-\alpha}  \int_0^t y^{\alpha} \left(y+\frac{s}{n}\right)^{-1} \dif{y},  
  \end{equation}
  then 
  \begin{eqnarray*}
      &&\sum_{s\geq 0}  \left( \frac{s}{n} \right)^{-\alpha} \Delta S^{(n)}\left(-\frac{s}{n}\right) \frac{1}{n} \sum_{l=1}^{nt} \left( \frac{l}{n} \right)^{\alpha} \left( \frac{l+s}{n} \right)^{-1}  \\
      &=& \sum_{s\geq 0} \Delta S^{(n)}\left(-\frac{s}{n}\right) \left[ f\left( \frac{s}{n},t\right)+ f^{(n)}\left( \frac{s}{n}, t\right) -f\left( \frac{s}{n},t\right) \right].
  \end{eqnarray*}
  Since $\left\vert\Delta S^{(n)}\left(-\tfrac{s}{n}\right) \right\vert \leq \mathrm{const}\cdot n^{-\alpha - \frac12}$ and 
  \begin{equation}
      \left\vert f^{(n)}\left( \frac{s}{n}, t\right) -f\left( \frac{s}{n},t\right) \right\vert
      \leq \mathrm{const}\cdot  \left( \frac{s}{n} \right)^{-\alpha} \cdot \frac{1}{n} \cdot \int_0^t y^{\alpha}\left(y+ \frac{s}{n} \right)^{-1} \dif{y} 
  \end{equation}
  we obtain
  \begin{equation}
      \left\vert\sum_{s\geq 0} \Delta S^{(n)}\left(-\frac{s}{n}\right) \left[  f^{(n)}\left( \frac{s}{n}, t\right) -f\left( \frac{s}{n},t\right) \right]
      \leq \mathrm{const}\cdot \sum_{s\geq 0}  n^{-\frac{3}{2}} s^{-\alpha} \int_0^t y^{\alpha}\left(y+\frac{s}{n}\right)^{-1}\dif{y}\right\vert.  
  \end{equation}
  We now split the sum: 
  For $\varepsilon>0$ small enough we have
  \begin{equation}
      \mathrm{const}\cdot \sum_{s= 0}^{n^{1+\varepsilon}}  n^{-\frac{3}{2}} s^{-\alpha} \int_0^t y^{\alpha}\left(y+\frac{s}{n}\right)^{-1}\dif{y}
      \leq \mathrm{const}\cdot  n^{-\frac{3}{2}} \left( n^{1+\varepsilon} \right)^{1-\alpha }  \rightarrow 0
  \end{equation}
  as $n\rightarrow \infty$, as well as 
  \begin{equation}
      \mathrm{const}\cdot \sum_{s\geq n^{1+\varepsilon}}  n^{-\frac{3}{2}} s^{-\alpha} \int_0^t y^{\alpha}\left(y+\frac{s}{n}\right)^{-1}\dif{y}
      \leq \mathrm{const}\cdot \sum_{s \geq n^{ 1+\varepsilon }} n^{-\frac{3}{2}}s^{-\alpha} \cdot \frac{n}{s} 
      \rightarrow 0
  \end{equation}
  as $n\rightarrow \infty$. 
  So
  \begin{equation}
    \frac{1}{\Gamma\left(\alpha\right)\Gamma\left(1-\alpha\right)}\sum_{s\geq 0}  \left( \frac{s}{n} \right)^{-\alpha} \Delta S^{(n)}\left(-\frac{s}{n}\right) \frac{1}{n} \sum_{l=1}^{nt} \left( \frac{l}{n} \right)^{\alpha} \left( \frac{l+s}{n} \right)^{-1}  
  \end{equation}
  converges in distribution to
  \begin{equation}
    \frac{1}{\Gamma\left(\alpha\right)\Gamma\left(1-\alpha\right)} \int_0^\infty s^{-\alpha} \int_0^t y^{\alpha} \left(y+s\right)^{-1} \dif{y}\, \dif{B}_{-s}.
  \end{equation}
  This gives the desired result since by Euler's reflection formula
  \begin{equation}
    \frac{1}{\Gamma\left(\alpha\right)\Gamma\left(1-\alpha\right)} = \frac{ \sin\left(\pi \left(H-\frac12\right)\right) }{\pi}.
  \end{equation}
\end{proof}
\section{ $\mathfrak y$-indexed processes (with memory)
}\label{sec:branchingtindexed}
In this section we want to put the branching HS-model in a wider context. Using this general theory we will then be able to conclude the validity of Proposition~\ref{prop:condGauss} directly by Theorem~\ref{prop:iwtheorem11cor}. 
To this purpose we start with the following definition which generalizes the class of branching processes beyond the Markov property.
 
\begin{definition}\label{def:branchingprocess}
   Let $\mathfrak y$ be a binary branching tree  as before and $ X= \left( X(t) \right)_{t\in\R}$ a real-valued stochastic process with distribution $\nu$ and càdlàg-paths. Given the tree $\mathfrak y$, 
  assume that attached to each branch $b$ of the random tree $\mathfrak y$ there is a process
  $\left( X_b(t) \right)_{t\in\R}$ such that for all branches $b, \tilde b\in \mathcal B$ with $b \wedge \tilde b =s$
  \begin{equation}\label{eq:equal}
    X_b(\xi)=X_{\tilde b}(\xi) \qquad \forall \xi \leq s
  \end{equation}
  and for all $b\in \mathcal B$
  \begin{equation}\label{eq:equalbranchdist}
    X_b \overset{(d)}{=} X. 
  \end{equation}
  The process
  \begin{equation}
   X_{\mathfrak y}:= \left( \left( X_{b}(t) \right)_{t\in\R} \right)_{b\in \mathcal B}
  \end{equation}
  will be called a \emph{  $\mathfrak y$-indexed process    (with memory)  with basic law $\nu$}.
\end{definition}
\
 Such processes can be constructed inductively for all basic laws as long as the tree $\mathfrak y$ has not uncountably many branches.\ 

\begin{definition}
  $ X_{\mathfrak y}$ has the \emph{pairwise conditional independence property \eqref{eq:equaldist}} 
 iff for all branches $b,\tilde b\in \mathcal B$ 
  \begin{equation}\label{eq:equaldist}\tag{PCI}
    \mathcal L\left( \left. \left( X_{b}(t) \right)_{ t\geq b\wedge \tilde b } \right\vert  \left( X_{\tilde b}(t) \right)_{ t\in\R}   \right)=\mathcal L\left( \left. \left( X_{b}(t) \right)_{ t\geq  b\wedge \tilde b} \right\vert  \left( X_{\tilde b}(t) \right)_{ t\leq b\wedge \tilde b }   \right)
  \end{equation}
  holds.
\end{definition}
\
 Note that \eqref{eq:equaldist} is equivalent to $\left( X_{b}(t) \right)_{ t\geq b\wedge \tilde b }$ being conditionally independent of $\left( X_{\tilde b}(t) \right)_{ t\geq b\wedge \tilde b }$ given $\left( X_{b}(t) \right)_{ t\leq b\wedge \tilde b }$   for all pairs of branches $b, \tilde b$. \ \\
 See also \cite[Section~2]{treeGaussianLLN} for a similar approach.\\\ 
Assume that the process $\left(X_t\right)_{t\in\R}$ has second moments. Now indeed for two branches the covariance $\Cov\left[ X_{b}(t_1), X_{\tilde b}(t_2) \right]$ of a \emph{$\mathfrak y$-indexed processes with basic law $\mathcal L \left( \left(X_t\right)_{t\in\R} \right)$}  having the \emph{pairwise conditional independence property \eqref{eq:equaldist}} \ is uniquely determined: 
\begin{proposition}\label{prop:covdist}
  Let
\begin{equation}
    \left( \left( X_{b}(t) \right)_{t\in\R} \right)_{b\in \mathcal B}
  \end{equation}
  be a
  \emph{  $\mathfrak y$-indexed process   with basic law $\mathcal L \left( \left(X_t\right)_{t\in\R} \right)$}   and let \eqref{eq:equaldist} be fulfilled. Assume that $\left(X_t\right)_{t\in\R}$ has finite second moments.
  For branches $b,\tilde b \in \mathfrak y$ with $b\wedge \tilde b=s$ and $t_1, t_2\geq 0$
  \begin{equation}
     \mathcal L\left( \left( X_{b}(t_1), X_{\tilde b}(t_2)\right) \right)  
  \end{equation}
  is uniquely determined by the distribution of the process $X$ along one branch, i.e. by $\mathcal L \left( \left( X_b(t) \right)_{t\in\R} \right)$
  for any branch $b\in \mathcal B$. 
\end{proposition}

\begin{corollary}
 Let the basic law $\mathcal L \left( \left(X_t\right)_{t\in\R} \right)$ have finite variance for all $t$. Then in the setting of the above proposition \ 
  \begin{equation}
    \Cov\left[ X_{b}(t_1), X_{\tilde b}(t_2) \right]
  \end{equation}
   is uniquely determined by the distribution of the process $X$ along one branch, i.e. by $\mathcal L \left( \left( X_b(t) \right)_{t\in\R} \right)$  for any branch $b\in \mathcal B$. 
\end{corollary}
 
   Keeping in mind Definition~\ref{def:branchingprocess}  for branches $b, b_1,\ldots, b_m\in \mathfrak y$ and
  \begin{equation}
    \mathcal S:= \bigcup\limits_{i=1}^m b_i
  \end{equation}
  we introduce the notations 
  \begin{equation}
   X_{b}:= \left( X_b(t) \right)_{t\in\R}, \qquad X_{\mathcal S}:= \left(\left( X_{b_i}(t) \right)_{ t\in\R  }\right)_{b_i \in \mathcal S}, \qquad  X_{b \backslash \mathcal S }:= \left( X_b(t) \right)_{ t\geq \max\limits_{i=1,\ldots, m} b\wedge b_i  },
  \end{equation}
 as well as 
  \begin{equation}
    X_{ \mathcal S \backslash b }:= \left(\left( X_{b_i}(t) \right)_{ t\geq b\wedge b_i }\right)_{b_i \in \mathcal S}, \qquad  X_{\mathcal S \cap b}:= \left(\left( X_{b_i}(t) \right)_{ t\leq b\wedge b_i }\right)_{b_i \in \mathcal S}.
  \end{equation}
  Having defined this we can specify a property stronger than \eqref{eq:equaldist}:
  \begin{definition}\label{def:strong}
Let $X$ be a \emph{$\mathfrak y$-indexed process  with basic law $\mathcal L \left( \left(X_t\right)_{t\in\R} \right)$}.
We say that $X$ has the the \emph{  conditional independence property   \eqref{eq:strongprop}} iff
  \begin{equation}\label{eq:strongprop}\tag{CI}
    \mathcal L \left( \left. X_{ b\backslash \mathcal S } \right\vert X_{ \mathcal S }  \right) =  \mathcal L \left( \left. X_{ b\backslash \mathcal S } \right\vert X_{ \mathcal S \cap b }  \right)
  \end{equation}
   for all pairwise distinct branches $b, b_1, \ldots, b_m\in \mathfrak y$  with  
  \begin{equation}
     \mathcal S:= \bigcup\limits_{i=1}^m b_i
     .
  \end{equation}
\end{definition}

Let us now give a definition of Gaussian   $\mathfrak y$-indexed process es: 
\begin{definition}\label{def:branchingprocess_gauss}
  Let $\nu$ be the law of a Gaussian process $\left( X(t)\right)_{t\in\R}$. We call the corresponding   $\mathfrak y$-indexed process  $X_{\mathfrak y}$  a \emph{Gaussian   $\mathfrak y$-indexed process } if and only if for all branches $b_1,\ldots, b_m\in \mathfrak y$ and $t_1, \ldots, t_m$
  \begin{equation}
    \left( X_{b_1}(t_1),\ldots, X_{b_m}(t_m) \right)
  \end{equation}
  is multivariate Gaussian distributed.
\end{definition}
The following Lemma is proved in \appG.
\begin{lemma}\label{eq:corkallenberg}
 Every \textbf{Gaussian} \emph{  $\mathfrak y$-indexed process }  fulfilling \eqref{eq:equaldist}  \ has the \emph{  conditional independence property   \eqref{eq:strongprop}}.
\end{lemma}

Proposition~\ref{prop:covdist} implies another nice property of Gaussian \emph{  $\mathfrak y$-indexed processes } having the  \emph{pairwise conditional independence property \eqref{eq:equaldist}} \  namely that their distribution is uniquely determined by the distribution along one branch: 
\begin{corollary}\label{cor:uniquedistGauss}
   The distribution of a \textbf{Gaussian} \emph{  $\mathfrak y$-indexed process } having the  \emph{pairwise conditional independence property \eqref{eq:equaldist}} \ is uniquely determined by its basic law. 
\end{corollary}
 
\begin{remark}\label{def:bfbm}
  Let $\mathfrak y$ be as before and let $B$ be a fractional Brownian motion with Hurst parameter $H$. By Corollary~\ref{cor:uniquedistGauss} \emph{$\mathfrak y$-indexed fractional Brownian motion}
  is the (unique in distribution)  \emph{Gaussian   $\mathfrak y$-indexed process } with basic law $\mathcal L(B)$ having the  \emph{pairwise conditional independence property  \eqref{eq:equaldist}}.  \ 
 \end{remark}
 \
Furthermore, \eqref{eq:strongprop} implies the following two corollaries for Gaussian processes: 
  \begin{corollary}\label{cor:multivariateGaussianBP}
   Let $X_{\mathfrak y}$ be a \emph{  $\mathfrak y$-indexed process es}. Assume that the basic law of $X_{\mathfrak y}$ is Gaussian and that $X_{\mathfrak y}$ has the property \eqref{eq:strongprop}. Then $X_{\mathfrak y}$ is a \emph{Gaussian   $\mathfrak y$-indexed process }.
  \end{corollary}
  More generally: 
  \begin{corollary}\label{eq:corasymjointGauss}
   Let $X_{\mathfrak y}^{(n)}$ be a sequence of \emph{  $\mathfrak y$-indexed process es}. Assume that the basic law of $X_{\mathfrak y}^{(n)}$ is asymptotically Gaussian for $n\rightarrow \infty$ and that the \emph{  conditional independence property   \eqref{eq:strongprop}} holds  for each $n\in\N$. Then $X_{\mathfrak y}^{(n)}$  is asymptotically a \emph{Gaussian   $\mathfrak y$-indexed process }. 
 \end{corollary}
 The following can be seen as a standard example for the fact, that bivariate Gaussianity is not sufficient for multivariate Gaussianity. 
  \begin{example}
   We give an example why in Corollary~\ref{cor:multivariateGaussianBP} the property \eqref{eq:strongprop} instead of \eqref{eq:equaldist} is needed. This example also shows why we need the multivariate Gaussianity in Lemma~\ref{eq:corkallenberg}. Let $\mathfrak y$ have the branches $0, 0r, 0rs$ and let $Z_i$ be independent standard normal random variables. Set
    \begin{eqnarray*}
       X_{0}(t)&:=& \mathbf{1}_{ t\leq r } Z_0 + \mathbf{1}_{ r\leq t \leq s } Z_1 +  \mathbf{1}_{ s\leq t } Z_2,\\
      X_{0r}(t)& :=& \mathbf{1}_{ t\leq r } Z_0 + \mathbf{1}_{ r\leq t \leq s } Z_3 +  \mathbf{1}_{ s\leq t } Z_4,\\
      X_{0rs}(t) &:=& \mathbf{1}_{ t\leq r } Z_0 + \mathbf{1}_{ r\leq t \leq s } Z_3 +  \mathbf{1}_{ s\leq t } \mathrm{sign}\left( Z_2Z_4 \right) | Z_5|. 
    \end{eqnarray*}
    Observe that $\left( X_0(t), X_{0r}(t), X_{0rs}(t) \right), t\geq s$ is not jointly Gaussian, but the \eqref{eq:equaldist} holds while \eqref{eq:strongprop} does not hold since the sign of $X_{0rs}(t), t\geq s$ is determined by $X_{0r}(t), X_{0}(t)$. So in general \eqref{eq:equaldist} does not imply \eqref{eq:strongprop} and the joint Gaussianity of $X_{\mathfrak y}$ is not implied by $X$ being a Gaussian process.  
  \end{example}

   The following example is particularly relevant in the context of this paper. 
\begin{example}
  One easily verifies that the process 
\begin{equation}
  S^{(n)}_{\mathfrak y}:= \left( \left( S_b^{(n)}(t) \right)_{t\in\R} \right)_{b\in \mathcal B}
\end{equation}
defined through \eqref{eq:HSRWbranching} is a  \emph{  $\mathfrak y$-indexed process   with basic law}
\begin{equation}
  \mathcal L \left( \left( S_0^{(n)}(t) \right)_{t\in\R} \right),
\end{equation}
see \eqref{eq:Snbranch0} for the definition of $S_0^{(n)}$. Note that $ S^{(n)}_{\mathfrak y}$ has the \emph{  conditional independence property   \eqref{eq:strongprop}.}  By construction the HS-random walk $ S^{(n)}_{\mathfrak y}$ has the property \eqref{eq:strongprop}. So we obtain that the branching HS-random walk is also asymptotically jointly Gaussian. 
\end{example}

 \section{Proof of Theorem~\ref{th:dbfbm} and Proposition~\ref{prop:condGauss}}\label{sec:proofthdbfbm}
 We will now make use of the general theory in the above section to prove Theorem~\ref{th:dbfbm} and Proposition~\ref{prop:condGauss}.
\begin{proof}[Proof of Proposition~\ref{prop:condGauss}]
 Since $\left(\left( S_0^{(n)}(t) \right)_{t}\right)_{n\in\N}$ is a  sequence  of Gaussian processes and $S_{\mathfrak y}$ obeys the \emph{  conditional independence property  } \eqref{eq:strongprop} we can apply Corollary~\ref{eq:corasymjointGauss} and get the assertion of Proposition~\ref{prop:condGauss}. 
\end{proof}
 Now Theorem~\ref{th:dbfbm} is a direct consequence of Proposition~\ref{prop:condGauss}:
 \begin{proof}[Proof of Theorem~\ref{th:dbfbm}]
  Combining Corollary~\ref{cor:uniquedistGauss} and Proposition~\ref{prop:condGauss} gives the desired result since the property \eqref{eq:strongprop} holds.
\end{proof}
\section{Proof of Theorem~\ref{eq:THmax}}\label{sec:THmaxproof}
We start with a slightly weaker result than Theorem~\ref{eq:THmax}: Let $\eta$ denote the distribution of a (binary branching) Yule tree $\mathfrak y$ (or even a deterministic binary branching tree like in Theorem~\ref{eq:THmax}), then
\begin{equation}\label{eq:ThMaxWeak}
   \liminf\limits_{t\rightarrow \infty} \bigintss_{\mathbb{T}}  \WS_{ \mathfrak y }\left( \frac{ \max\limits_{b\in \mathcal B} B_{b}(t)  }{t^{H+\frac12}}  > a \right) \eta\left(\dif \mathfrak y\right) >0 \quad \mbox { for }a\in\R_{+}  
\end{equation}
and  for all $f(t)\gg t^{H+\frac12}$  
\begin{equation}\label{eq:ThMaxWeak2}
     \limsup\limits_{t\rightarrow \infty} \bigintss_{\mathbb{T}}  \WS_{ \mathfrak y }\left( \frac{ \max\limits_{b\in \mathcal B} B_{b}(t)  }{f(t)}  > a \right) \eta\left(\dif \mathfrak y\right) =0 \quad \mbox { for }a\in\R_{+}  
\end{equation}
 can be obtained by the following well known lemma. Details can be found in \appA.  
\begin{lemma}[{Slepian inequality, \cite[Section~2.10]{slepian}}]\label{th:slepian}
  Let $\sigma^2>0$ and $G$ be a Gaussian vector in $\R^l$ with covariance matrix
  \begin{equation}
    \Sigma_G(i,j)
    \begin{cases}
      =\sigma^2>0,& i=j\\
      \geq 0,&i\not=j,
    \end{cases}
  \end{equation}
  and $G'$ a Gaussian vector in $\R^l$ with covariance matrix
  \begin{equation}
     \Sigma_{G'}(i,j)
    \begin{cases}
      =\sigma^2>0,& i=j\\
      \geq \Sigma_{G}(i,j),&i\not=j.
    \end{cases}
  \end{equation}
  then for all $a=\left( a_1, \ldots a_l\right)\in\R^l $
  \begin{equation}
    \WS\left( G_i < a_i \quad \forall i  \right) \leq \WS\left( G_i' < a_i \quad \forall i  \right).
  \end{equation}
  This gives also
  for $a\in\R$
  \begin{equation}
    \WS\left( \max\limits_{i}G_i > a  \right) \geq \WS\left(\max\limits_{i}  G_i' >a  \right). 
  \end{equation}
\end{lemma}
Essentially the above lemma states that increasing already positive correlations decreases the maximum of a Gaussian vector.

We now prove Theorem~\ref{eq:THmax} by using results of \cite[Theorem~3.1]{GREM2} and connecting the maximum of branching fractional Brownian motion to a \emph{generalized random energy model }(GREM).  In order to connect our setting with  \cite[Theorem~3.1]{GREM2} we briefly describe our problem in their notation.
\begin{remark}
  Bovier and Kurkova \cite{GREM2} consider a Gaussian process $X_{\sigma}$ on the hypercube $ \{-1, 1\}^N$ whose covariance is given by
  \begin{equation}
    \EW\left[ X_{\sigma} X_{\sigma'} \right] = A \left( d_N(\sigma, \sigma') \right) \mbox{ for } d_N(\sigma, \sigma'):= \frac{1}{N}\left( \min\left\{ i: \sigma_i \not=\sigma_i' \right\}-1 \right)
  \end{equation}
  and some probability density function $A$ on $[0,1]$. In our setting the $X_{\sigma}$ will correspond to the $B_b(t)$, and the covariance will be given by $\rho\left( t, t, b\wedge \tilde b \right)$ for two branches $b$ and $\tilde b$, see \eqref{eq:rhoKtt}, \eqref{eq:covhsbranching} and  \eqref{eq:covequality} for the definition of $\rho$. 
\end{remark} 
\begin{proof}[Proof of Theorem~\ref{eq:THmax}]
  To prove Theorem~\ref{eq:THmax} we will apply \cite[Theorem~3.1]{GREM2}.
   We think of a tree $\mathfrak y_{\mathrm{bin}}$ like the one in figure~\ref{fig:fig3}, a tree in which every branch branches into two after time 1. If one now assigns to each part of each branch random variables $\Delta Z_{(b), r}$ like it is done in figure~\ref{fig:fig4} for $t=3$, such that
  \begin{equation}
    \Delta Z_{b,r} \sim \mathcal N \left( 0,  \rho\left( t, t, r  \right)- \rho\left( t, t, r-1  \right)  \right), \qquad r>0
  \end{equation}
  and
  \begin{equation}
    \Delta Z_{b,0} \sim \mathcal N \left( 0, \rho \left(t, t, 0\right) \right),
  \end{equation}
  we get that if we define $Z_{(b)}$ as the sum of the random variables along branch $b$, i.e. for $t=3$
  \begin{eqnarray*}
    Z_{013}&=& \Delta Z_{013,0}+ \Delta Z_{013,1}+\Delta Z_{013,2}+\Delta Z_{013,3}\\
             &=& \Delta Z_{0,0}+ \Delta Z_{01,1}+ \Delta Z_{01,2}+\Delta Z_{013,3},
  \end{eqnarray*}
  the following equality in distribution:
  \begin{equation}
    \left( B_b(t) \right)_{b\in \mathcal B} \overset{(d)}{=} \left( Z_{b} \right)_{b\in \mathcal B}.
  \end{equation}
  Note that $\rho(t,t,s)$ is increasing as a function of $s$, see end of Section~\ref{sec:rhoisinc}. 
  Now we are in the world of a classical \emph{generalized random energy model }(GREM), see e.g. \cite{GREM2}. So in essence we represented the values of the fractional Brownian motions $B_b$ at time $t$ via a GREM. 
  \begin{figure}[h]
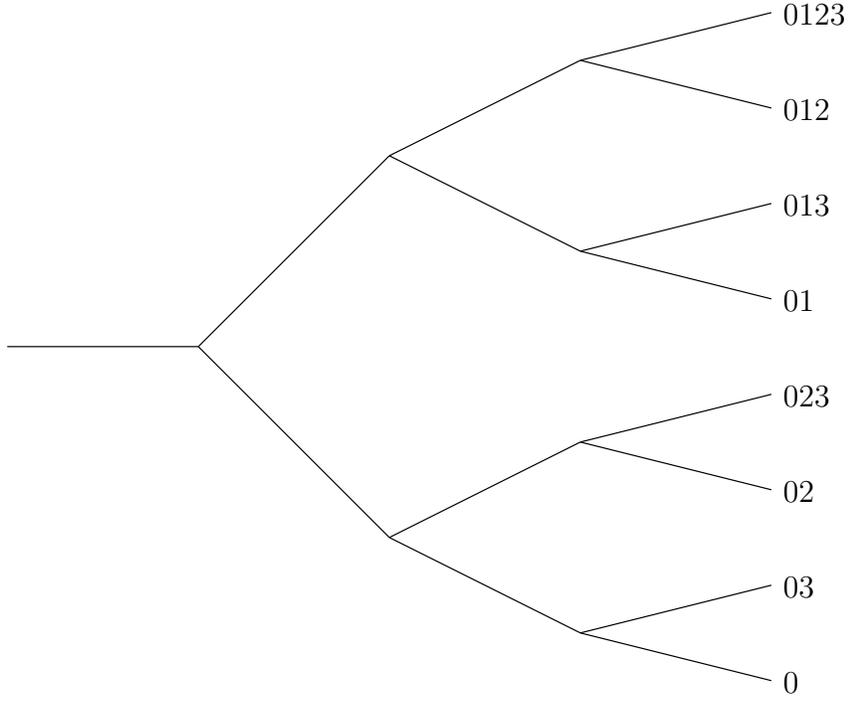

    \centering
    \scalebox{1}{
{\pgfkeys{/pgf/fpu/.try=false}%
\ifx\XFigwidth\undefined\dimen1=0pt\else\dimen1\XFigwidth\fi
\divide\dimen1 by 5649
\ifx\XFigheight\undefined\dimen3=0pt\else\dimen3\XFigheight\fi
\divide\dimen3 by 4749
\ifdim\dimen1=0pt\ifdim\dimen3=0pt\dimen1=3946sp\dimen3\dimen1
  \else\dimen1\dimen3\fi\else\ifdim\dimen3=0pt\dimen3\dimen1\fi\fi
\tikzpicture[x=+\dimen1, y=+\dimen3]
{\ifx\XFigu\undefined\catcode`\@11
\def\temp{\alloc@1\dimen\dimendef\insc@unt}\temp\XFigu\catcode`\@12\fi}
\XFigu3946sp
\ifdim\XFigu<0pt\XFigu-\XFigu\fi
\clip(1188,-7212) rectangle (6837,-2463);
\tikzset{inner sep=+0pt, outer sep=+0pt}
\pgfsetfillcolor{black}
\pgftext[base,left,at=\pgfqpointxy{6075}{-2775}] {\fontsize{12}{14.4}\usefont{T1}{ptm}{m}{n}$0123$}
\pgfsetlinewidth{+7.5\XFigu}
\pgfsetstrokecolor{black}
\draw (2400,-4800)--(3600,-3600);
\draw (2400,-4800)--(3600,-6000);
\draw (3600,-3600)--(4800,-3000);
\draw (3600,-3600)--(4800,-4200);
\draw (3600,-6000)--(4800,-5400);
\draw (3600,-6000)--(4800,-6600);
\draw (4800,-3000)--(6000,-2700);
\draw (4800,-3000)--(6000,-3300);
\draw (4800,-4200)--(6000,-3900);
\draw (4800,-4200)--(6000,-4500);
\draw (4800,-5400)--(6000,-5100);
\draw (4800,-5400)--(6000,-5700);
\draw (4800,-6600)--(6000,-6300);
\draw (4800,-6600)--(6000,-6900);
\pgfsetstrokecolor{white}
\draw (6825,-2475)--(6825,-7200);
\pgftext[base,left,at=\pgfqpointxy{6075}{-6975}] {\fontsize{12}{14.4}\usefont{T1}{ptm}{m}{n}$0$}
\pgftext[base,left,at=\pgfqpointxy{6075}{-6375}] {\fontsize{12}{14.4}\usefont{T1}{ptm}{m}{n}$03$}
\pgftext[base,left,at=\pgfqpointxy{6075}{-5775}] {\fontsize{12}{14.4}\usefont{T1}{ptm}{m}{n}$02$}
\pgftext[base,left,at=\pgfqpointxy{6075}{-5175}] {\fontsize{12}{14.4}\usefont{T1}{ptm}{m}{n}$023$}
\pgftext[base,left,at=\pgfqpointxy{6075}{-4575}] {\fontsize{12}{14.4}\usefont{T1}{ptm}{m}{n}$01$}
\pgftext[base,left,at=\pgfqpointxy{6075}{-3975}] {\fontsize{12}{14.4}\usefont{T1}{ptm}{m}{n}$013$}
\pgftext[base,left,at=\pgfqpointxy{6075}{-3375}] {\fontsize{12}{14.4}\usefont{T1}{ptm}{m}{n}$012$}
\pgfsetstrokecolor{black}
\draw (1200,-4800)--(2400,-4800);
\endtikzpicture}%
}
    \caption{Binary Branching GREM}
    \label{fig:fig3}
  \end{figure}
  \begin{figure}[h]
    \centering
    {\pgfkeys{/pgf/fpu/.try=false}%
\ifx\XFigwidth\undefined\dimen1=0pt\else\dimen1\XFigwidth\fi
\divide\dimen1 by 5577
\ifx\XFigheight\undefined\dimen3=0pt\else\dimen3\XFigheight\fi
\divide\dimen3 by 6252
\ifdim\dimen1=0pt\ifdim\dimen3=0pt\dimen1=3946sp\dimen3\dimen1
  \else\dimen1\dimen3\fi\else\ifdim\dimen3=0pt\dimen3\dimen1\fi\fi
\tikzpicture[x=+\dimen1, y=+\dimen3]
{\ifx\XFigu\undefined\catcode`\@11
\def\temp{\alloc@1\dimen\dimendef\insc@unt}\temp\XFigu\catcode`\@12\fi}
\XFigu3946sp
\ifdim\XFigu<0pt\XFigu-\XFigu\fi
\clip(1110,-7812) rectangle (6687,-1560);
\tikzset{inner sep=+0pt, outer sep=+0pt}
\pgfsetfillcolor{black}
\pgftext[base,left,at=\pgfqpointxy{5925}{-1725}] {\fontsize{12}{14.4}\usefont{T1}{ptm}{m}{n}3}
\pgfsetlinewidth{+7.5\XFigu}
\pgfsetstrokecolor{black}
\draw (2400,-4800)--(3600,-3600);
\draw (3600,-6000)--(4800,-5400);
\draw (4800,-3000)--(6000,-2700);
\draw (4800,-3000)--(6000,-3300);
\draw (4800,-5400)--(6000,-5100);
\draw (4800,-5400)--(6000,-5700);
\draw (4800,-6600)--(6000,-6300);
\draw (2400,-4800)--(3600,-6000);
\draw (3570,-3615)--(4770,-3015);
\draw (3600,-3600)--(4800,-4200);
\draw (4800,-4200)--(6000,-3900);
\draw (4800,-4200)--(6000,-4500);
\draw (3600,-6000)--(4800,-6600);
\pgfsetdash{{+15\XFigu}{+45\XFigu}}{+15\XFigu}
\draw (3600,-1800)--(3600,-7800);
\draw (4800,-1800)--(4800,-7800);
\draw (6000,-1800)--(6000,-7800);
\pgfsetdash{}{+0pt}
\draw (4800,-6600)--(6000,-6900);
\pgfsetstrokecolor{white}
\draw (6675,-2475)--(6675,-7200);
\pgfsetstrokecolor{black}
\pgfsetdash{{+15\XFigu}{+45\XFigu}}{+15\XFigu}
\draw (2400,-1800)--(2400,-7800);
\pgftext[base,left,at=\pgfqpointxy{3750}{-5550}] {\fontsize{12}{14.4}\usefont{T1}{ptm}{m}{n}$\Delta Z_{02,2}$}
\pgftext[base,left,at=\pgfqpointxy{3750}{-4125}] {\fontsize{12}{14.4}\usefont{T1}{ptm}{m}{n}$\Delta Z_{01,2}$}
\pgftext[base,left,at=\pgfqpointxy{3750}{-6525}] {\fontsize{12}{14.4}\usefont{T1}{ptm}{m}{n}$\Delta Z_{0,2}$}
\pgftext[base,left,at=\pgfqpointxy{5175}{-2625}] {\fontsize{12}{14.4}\usefont{T1}{ptm}{m}{n}$\Delta Z_{0123,3}$}
\pgftext[base,left,at=\pgfqpointxy{5175}{-3525}] {\fontsize{12}{14.4}\usefont{T1}{ptm}{m}{n}$\Delta Z_{012,3}$}
\pgftext[base,left,at=\pgfqpointxy{5175}{-3825}] {\fontsize{12}{14.4}\usefont{T1}{ptm}{m}{n}$\Delta Z_{013,3}$}
\pgftext[base,left,at=\pgfqpointxy{5175}{-4650}] {\fontsize{12}{14.4}\usefont{T1}{ptm}{m}{n}$\Delta Z_{01,3}$}
\pgftext[base,left,at=\pgfqpointxy{5175}{-5025}] {\fontsize{12}{14.4}\usefont{T1}{ptm}{m}{n}$\Delta Z_{023,3}$}
\pgftext[base,left,at=\pgfqpointxy{5175}{-5850}] {\fontsize{12}{14.4}\usefont{T1}{ptm}{m}{n}$\Delta Z_{02,3}$}
\pgftext[base,left,at=\pgfqpointxy{5175}{-6225}] {\fontsize{12}{14.4}\usefont{T1}{ptm}{m}{n}$\Delta Z_{03,3}$}
\pgftext[base,left,at=\pgfqpointxy{3675}{-3150}] {\fontsize{12}{14.4}\usefont{T1}{ptm}{m}{n}$\Delta Z_{012,2}$}
\pgftext[base,left,at=\pgfqpointxy{2325}{-4125}] {\fontsize{12}{14.4}\usefont{T1}{ptm}{m}{n}$\Delta Z_{01,1}$}
\pgftext[base,left,at=\pgfqpointxy{2325}{-5550}] {\fontsize{12}{14.4}\usefont{T1}{ptm}{m}{n}$\Delta Z_{0,1}$}
\pgftext[base,left,at=\pgfqpointxy{6000}{-2775}] {\fontsize{12}{14.4}\usefont{T1}{ptm}{m}{n}$Z_{0123}$}
\pgftext[base,left,at=\pgfqpointxy{6075}{-3375}] {\fontsize{12}{14.4}\usefont{T1}{ptm}{m}{n}$Z_{012}$}
\pgftext[base,left,at=\pgfqpointxy{6075}{-3975}] {\fontsize{12}{14.4}\usefont{T1}{ptm}{m}{n}$Z_{013}$}
\pgftext[base,left,at=\pgfqpointxy{6075}{-4575}] {\fontsize{12}{14.4}\usefont{T1}{ptm}{m}{n}$Z_{01}$}
\pgftext[base,left,at=\pgfqpointxy{6075}{-5175}] {\fontsize{12}{14.4}\usefont{T1}{ptm}{m}{n}$Z_{023}$}
\pgftext[base,left,at=\pgfqpointxy{6075}{-6375}] {\fontsize{12}{14.4}\usefont{T1}{ptm}{m}{n}$Z_{03}$}
\pgftext[base,left,at=\pgfqpointxy{6075}{-6975}] {\fontsize{12}{14.4}\usefont{T1}{ptm}{m}{n}$Z_{0}$}
\pgftext[base,left,at=\pgfqpointxy{5175}{-7050}] {\fontsize{12}{14.4}\usefont{T1}{ptm}{m}{n}$\Delta Z_{0,3}$}
\pgftext[base,left,at=\pgfqpointxy{6075}{-5775}] {\fontsize{12}{14.4}\usefont{T1}{ptm}{m}{n}$Z_{02}$}
\pgftext[base,left,at=\pgfqpointxy{1125}{-4650}] {\fontsize{12}{14.4}\usefont{T1}{ptm}{m}{n}$\Delta Z_{0,0}$}
\pgftext[base,left,at=\pgfqpointxy{2400}{-1725}] {\fontsize{12}{14.4}\usefont{T1}{ptm}{m}{n}0}
\pgftext[base,left,at=\pgfqpointxy{3525}{-1725}] {\fontsize{12}{14.4}\usefont{T1}{ptm}{m}{n}1}
\pgftext[base,left,at=\pgfqpointxy{4725}{-1725}] {\fontsize{12}{14.4}\usefont{T1}{ptm}{m}{n}2}
\pgfsetdash{}{+0pt}
\draw (1200,-4800)--(2400,-4800);
\endtikzpicture}%
    \caption{Binary Branching GREM}
    \label{fig:fig4}
  \end{figure}
  Since this is now precisely the setting of \cite[Theorem~3.1]{GREM2}, we get 
  that for a setting like this with
  \begin{equation}
    \Covbin\left[ Z_{b}, Z_{\tilde b} \right]= t^{2H} \left[1 -  \frac{ \sqrt{\pi} 2^{2H-1} }{ \Gamma(1-H)  \Gamma\left(H+\frac12\right) } \left(1- \frac{s}{t} \right)^{2H}\right],
  \end{equation}
  see \eqref{eq:formcov}, 
  for two branches $b$ and $\tilde b$ with $b\wedge \tilde b = s$
  we get that the leading order of the maximum for $t\rightarrow \infty$ is 
  \begin{equation}
    t^{H+\frac12}\sqrt{2} \sqrt{\log(2)} \int_0^1 \sqrt{ \bar{a}(x) } \dif{x},
  \end{equation}
  where
  \begin{equation}
    \bar{a}(x):= \frac{\dif}{\dif{x}} \Covbin\left[ Z_{(b)}, Z_{\left( b xt \right)} \right]. 
  \end{equation}
  So
  \begin{equation}
    \bar{a}(x)
    = \frac{ \sqrt{\pi} 2^{2H-1} }{ \Gamma(1-H)  \Gamma\left(H+\frac12\right) } 2H \left( 1-x \right)^{2H-1},
  \end{equation}
  which gives
  \begin{eqnarray*}
    &&\sqrt{2\log 2} \int_0^1 \sqrt{ \bar{a}(x) } \dif{x}\\
    &=& \sqrt{2\log 2 \frac{ \sqrt{\pi} 2^{2H-1} }{ \Gamma(1-H)  \Gamma\left(H+\frac12\right) } 2H} \int_0^1 \left(1-x\right)^{H-\frac12}\dif{x}\\
                                                       &=&\sqrt{ 2\log 2 \frac{ \sqrt{\pi} 2^{2H} }{ \Gamma(1-H)  \Gamma\left(H+\frac12\right) } H} \left[- \frac{1}{H+\frac12} \left(1-x\right)^{H+\frac12} \right]_{x=0}^1 \\
                                                       &=& \frac{\sqrt{2\log 2 \frac{ \sqrt{\pi} 2^{2H} }{ \Gamma(1-H)  \Gamma\left(H+\frac12\right) } H}}{H+\frac12}.
  \end{eqnarray*}
Since $\bar{a}(x)'<0$  it follows by
\cite[Theorem~3.1]{GREM2} 
that as stated
  \begin{equation}
    \EW_{ \mathfrak y_{\mathrm{bin}} }\left( \frac{ \max\limits_{b\in \mathcal B} B_{b}(t)  }{m(t) }   \right) \rightarrow 1 \quad \mbox { for }t\rightarrow \infty
  \end{equation}
  for $m(t)$ defined by \eqref{eq:mtdef}. 
\end{proof}
 \FloatBarrier

\section{Proof of Theorem~\ref{eq:THmax2}}\label{sec:THmax2proof}
\begin{proof}
We will approximate $\left( B_b(t) \right)_{b\in \mathcal B}$
  by a sequence of GREMs, cf. Section~\ref{sec:THmaxproof}.
  For this purpose we discretize time: Let $K\in\N$. For all $i=1,\ldots, K-1$ we shift all branching events in $\left[ \frac{i}{K} t , \frac{i+1}{K} t\right]$ of the Yule tree $\mathfrak y$ to $\frac{i}{K} t$, such that we can only have branching events at $0,\frac{t}{K}, \frac{2t}{K},\ldots, \frac{(K-1)t}{K}$.  This approximation to the Yule Tree  $\mathfrak y$ will be denoted by $\mathfrak y_{K,t}$, and its collection of branches by $\mathcal B_{K,t}$. 
    In the GREM we set 
    \begin{equation}
        Z_b^{(i)}:= \sum_{\ell=0}^i \Delta Z_b^{(\ell)},
    \end{equation}
    where the $\Delta Z_b^{(\ell)}$ are the increments along the branch $b$, see Figure~\ref{fig:figgremm}.
  We set $Z_b \equiv Z^{(K)}_b $.
For two branches $b$ and $\tilde b$ with $b\wedge \tilde b =s$ by
\eqref{eq:rhoKtt} we have
\begin{equation}
  \Cov\left[ B_{b}(t), B_{\tilde b}(t) \right] = t^{2H} \left[ 1- C_{\rho} \left(1- \frac{s}{t} \right)^{2H} \right]= \rho(t,t,s),
\end{equation}
where
\begin{equation}
  C_{\rho} = \frac{ \sqrt{\pi}2^{2H-1} }{ \Gamma\left( 1-H \right) \Gamma\left( H+\frac12 \right)}.
\end{equation}
Thus, in order to represent the $\mathfrak y_{K,t}$-indexed fractional Brownian motion as a GREM, we choose, for each $b\in \mathcal B_{K,t}$, the distribution of the increment $\Delta Z_{b}^{(i)}$
as
\begin{equation}\label{eq:defdeltaZ}
  \mathcal N \left( 0, \rho\left( t,t,\frac{i}{K}t  \right) -\rho\left( t,t,\frac{i-1}{K}t  \right)  \right) \qquad \mbox{if } i\ge 0
\end{equation}
and
\begin{equation}
  \Delta Z_{b}^{(0)}\equiv  Z_{b}^{(0)} \sim  \mathcal N \left( 0, \rho\left( 0,t,t  \right)  \right)=\mathcal N \left( 0, \left( 1- C_{\rho}\right)t^{2H}  \right).
\end{equation}
Remember that in a GREM  $\Delta Z_b^{(i)}$ and $\Delta Z_{\tilde b}^{(i)}$ coincide if the branches $b$ and $\tilde b$ did not separate till $\frac{iT}{K}$, meaning $b \wedge \tilde b > \frac{iT}{K}$.
  \begin{figure}
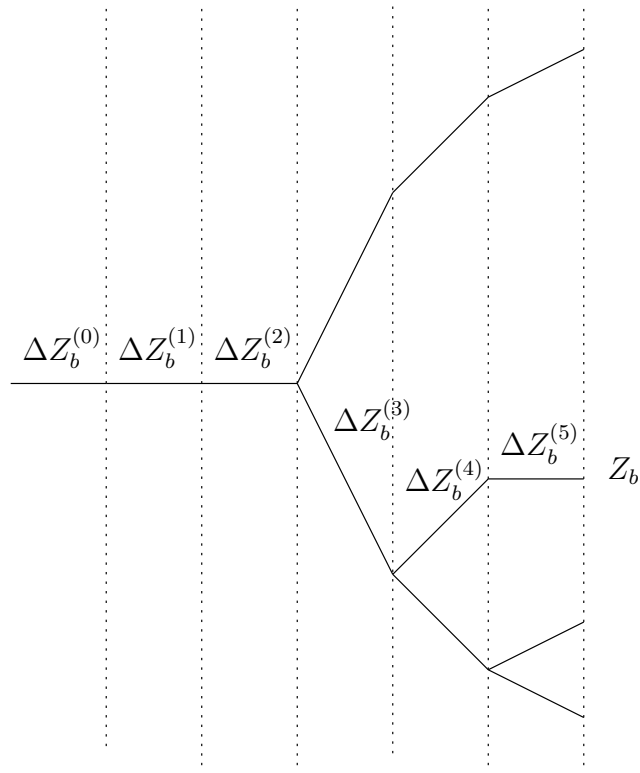

  \centering
  {\pgfkeys{/pgf/fpu/.try=false}%
\ifx\XFigwidth\undefined\dimen1=0pt\else\dimen1\XFigwidth\fi
\divide\dimen1 by 5424
\ifx\XFigheight\undefined\dimen3=0pt\else\dimen3\XFigheight\fi
\divide\dimen3 by 4899
\ifdim\dimen1=0pt\ifdim\dimen3=0pt\dimen1=3946sp\dimen3\dimen1
  \else\dimen1\dimen3\fi\else\ifdim\dimen3=0pt\dimen3\dimen1\fi\fi
\tikzpicture[x=+\dimen1, y=+\dimen3]
{\ifx\XFigu\undefined\catcode`\@11
\def\temp{\alloc@1\dimen\dimendef\insc@unt}\temp\XFigu\catcode`\@12\fi}
\XFigu3946sp
\ifdim\XFigu<0pt\XFigu-\XFigu\fi
\clip(-1212,-6012) rectangle (4212,-1113);
\tikzset{inner sep=+0pt, outer sep=+0pt}
\pgfsetfillcolor{black}
\pgftext[base,left,at=\pgfqpointxy{1425}{-3900}] {\fontsize{12}{14.4}\usefont{T1}{ptm}{m}{n}$\Delta Z_{b}^{(3)}$}
\pgfsetlinewidth{+7.5\XFigu}
\pgfsetstrokecolor{black}
\pgfsetdash{{+15\XFigu}{+45\XFigu}}{+15\XFigu}
\draw (1200,-1200)--(1200,-6000);
\draw (600,-1200)--(600,-6000);
\draw (3000,-1200)--(3000,-6000);
\pgfsetdash{}{+0pt}
\draw (1200,-3600)--(1800,-2400);
\draw (1200,-3600)--(1800,-4800);
\draw (1800,-2400)--(2400,-1800);
\draw (2400,-1800)--(3000,-1500);
\draw (1800,-4800)--(2400,-5400);
\draw (1800,-4800)--(2400,-4200);
\draw (2400,-5400)--(3000,-5700);
\draw (2400,-5400)--(3000,-5100);
\draw (2400,-4200)--(3000,-4200);
\pgfsetstrokecolor{white}
\draw (3975,-5700)--(3975,-5775);
\pgfsetstrokecolor{black}
\draw (0,-3600)--(-600,-3600);
\pgfsetstrokecolor{white}
\draw (4200,-1275)--(4200,-6000);
\pgfsetstrokecolor{black}
\pgfsetdash{{+15\XFigu}{+45\XFigu}}{+15\XFigu}
\draw (0,-1200)--(0,-5925);
\pgfsetstrokecolor{white}
\draw (-1200,-1200)--(-1200,-6000);
\pgfsetstrokecolor{black}
\draw (2400,-1200)--(2400,-6000);
\draw (1800,-1125)--(1800,-5925);
\pgftext[base,left,at=\pgfqpointxy{3150}{-4200}] {\fontsize{12}{14.4}\usefont{T1}{ptm}{m}{n}$Z_{b}$}
\pgftext[base,left,at=\pgfqpointxy{-525}{-3450}] {\fontsize{12}{14.4}\usefont{T1}{ptm}{m}{n}$\Delta Z_{b}^{(0)}$}
\pgftext[base,left,at=\pgfqpointxy{75}{-3450}] {\fontsize{12}{14.4}\usefont{T1}{ptm}{m}{n}$\Delta Z_{b}^{(1)}$}
\pgftext[base,left,at=\pgfqpointxy{675}{-3450}] {\fontsize{12}{14.4}\usefont{T1}{ptm}{m}{n}$\Delta Z_{b}^{(2)}$}
\pgftext[base,left,at=\pgfqpointxy{2475}{-4050}] {\fontsize{12}{14.4}\usefont{T1}{ptm}{m}{n}$\Delta Z_{b}^{(5)}$}
\pgftext[base,left,at=\pgfqpointxy{1875}{-4275}] {\fontsize{12}{14.4}\usefont{T1}{ptm}{m}{n}$\Delta Z_{b}^{(4)}$}
\pgfsetdash{}{+0pt}
\draw (0,-3600)--(1200,-3600);
\endtikzpicture}%
  \caption{A GREM on a discretized Yule tree.}
  \label{fig:figgremm}
\end{figure}
Since we shifted all branching events in the time intervals of length $\tfrac{1}{K}$ to the left by Slepian's Lemma, see Lemma~\ref{th:slepian}, the  order the maximum of the approximating GREM gives an upper bound for the order of the maximum of BFBM.

Now denote the rescaled standard variations of $\Delta Z_0^{(i)}, i\geq 1$ as
\begin{eqnarray}
  \Delta f_i&:=& \sqrt{ 2\frac{t}{K} } \sqrt{\rho\left( t ,t,\frac{i}{K}t \right) -\rho\left( t,t,\frac{i-1}{K}t\right)}\nonumber\\
                 &=&\sqrt{ 2\frac{t}{K} } t^H \left[ -C_{\rho} \left( 1-\frac{i}{K} \right)^{2H}+C_{\rho} \left( 1-\frac{i-1}{K} \right)^{2H} \right]^{\frac12}\label{eq:deltaf}            
\end{eqnarray}
and the rescaled standard variation of $\Delta Z _0^{(0)}$
 \begin{equation}
   \Delta f_0:= \sqrt{ \rho\left( t,t,0  \right)}.
 \end{equation}
\FloatBarrier
We now start to obtain an upper bound on the order of the maximum. 
Since $Z_b$ is equal in distribution to the main branch $Z_0$ we have almost surely
\begin{eqnarray*}
  &&\WS_{ \mathfrak y }\left(  \exists b\in \mathcal B_{K,t}:  Z_{b} >  (1+\varepsilon)\sum_{i=0}^K \Delta f_i 
  \right)\leq |\mathcal B_{K,t}| \WS_{ \mathfrak y }\left( \left.  Z_{0} >   (1+\varepsilon) \sum_{i=0}^K \Delta f_i  \right. \right)\nonumber\\
  &\leq& |\mathcal B_{K,t}| \sum_{ \{(l_i): \sum_{i=1}^K l_i <0 \}}  \WS_{ \mathfrak y }\left( \left.  \Delta Z_{0}^{(i)} \in  \left[ (\Delta f_i- l_i \varepsilon )\pm\varepsilon \right] \quad \forall i \right. \right)\nonumber.
\end{eqnarray*}
Now by independence we get
\begin{eqnarray*}
  &&|\mathcal B_{K,t}| \sum_{ \{(l_i): \sum_{i=1}^K l_i <0 \}}\WS_{ \mathfrak y }\left( \left.  \Delta Z_{0}^{(i)} \in   \left[ \left(\Delta f_i- l_i \varepsilon\right)\pm\varepsilon \right] \quad \forall i \right. \right)\nonumber\\
  &=& |\mathcal B_{t,K}| \sum_{ \{(l_i): \sum_{i=1}^K l_i <0 \}} \prod_{{i=1\normalcolor}}^K \WS_{ \mathfrak y }\left( \left.  \Delta Z_{0}^{(i)} \in  \left[ \left(\Delta f_i- l_i \varepsilon\right)\pm\varepsilon \right] \right. \right).\nonumber
\end{eqnarray*}
\normalcolor
By \eqref{eq:defdeltaZ} we obtain
\begin{eqnarray}
  && |\mathcal B_{t,K}| \sum_{ \{(l_i): \sum_{i=1}^K l_i <0 \}} \prod_{{i=1\normalcolor}}^K \WS_{ \mathfrak y }\left( \left.  \Delta Z_{0}^{(i)} \in \left[ \left(\Delta f_i- l_i \varepsilon\right)\pm\varepsilon \right] \right. \right).\nonumber\\
  &\leq& \mathrm{const}\cdot |\mathcal B_{t,K}| \sum_{ \{(l_i): \sum_{i=1}^K l_i <0 \}} \prod_{{i=1\normalcolor}}^K \exp\left( - \frac{ \left( \Delta f_i- l_i \varepsilon \right)^2 }{2 \left( \rho\left( \frac{i}{K}t,t,t  \right) -\rho\left( \frac{i-1}{K}t,t,t  \right) \right) } \right)  \nonumber\\
  &=& \mathrm{const} \cdot |\mathcal B_{t,K}| \sum_{ \{(l_i): \sum_{i=1}^K l_i <0 \}} \prod_{{i=1\normalcolor} }^K \exp\left( - \frac{t}{K}\cdot \frac{ \left( \Delta f_i- l_i \varepsilon \right)^2 }{ \Delta f_i^2  } \right)  \label{eq:sumoverall}.
\end{eqnarray}
\normalcolor
We now examine the product in \eqref{eq:sumoverall} in more detail:
\begin{equation}
    \prod_{{i=1\normalcolor}}^K \exp\left( - \frac{t}{K}\cdot \frac{ \left( \Delta f_i- l_i \varepsilon \right)^2 }{ \Delta f_i^2  } \right)  
    =\exp\left( - \frac{t}{K}\sum_{{i=1\normalcolor}}^K \frac{ \left( \Delta f_i- l_i \varepsilon \right)^2 }{ \Delta f_i^2  } \right)  \nonumber.
\end{equation}
Writing 
\begin{equation}
    g_i := \sqrt{2 C_{\rho}} \left[  - \left( 1-\frac{i}{K}\right)^{2H} +\left( 1-\frac{i-1}{K}\right)^{2H} \right]^{\frac12}
  \end{equation}
  we can write the exponent as
  \begin{eqnarray}
    &&-\frac{t}{K} \sum_{{i=1\normalcolor}}^K \left[ 1+ \frac{l_i^2 \varepsilon^2}{\Delta f_i^2} - \frac{2 l_i \varepsilon}{\Delta f_i}\right]\nonumber\\
    &=& -t - \frac{t\varepsilon^2 }{K}\sum_{{i=1\normalcolor}}^K\frac{l_i^2 }{\Delta f_i^2} + \frac{2t\varepsilon}{K}\sum_{{i=1\normalcolor}}^K  \frac{l_i}{\Delta f_i}\nonumber\\
    &=& - t - \frac{t}{K}\varepsilon^2  \sum_{{i=1\normalcolor}}^K l_i^2 \frac{K}{g_i^2 t^{2H+1}} + \frac{2t \varepsilon}{K} \sum_{{i=1\normalcolor}}^K l_i \frac{\sqrt{K}}{ t^{H+\frac12} g_i }\nonumber\\
    &=& -t - \frac{\varepsilon^2}{t^{2H}}\sum_{{i=1\normalcolor}}^K \frac{l_i^2}{g_i^2} +\frac{t^{ \frac12-H }}{K^{\frac12}} \varepsilon \sum_{{i=1\normalcolor}}^K \frac{l_i}{g_i}.\nonumber
  \end{eqnarray}
 For $\delta >0$ choose $\varepsilon = \delta t^{H+\frac12}$ and  $K\equiv t^{1-\xi}$ for some $1>\xi>0$  this becomes
  \begin{eqnarray}
    &&-t - t\delta^2 \sum_{{i=1\normalcolor}}^K  \frac{l_i^2}{g_i^2} + t^{\frac12 + \frac{\xi}{2}} \delta\sum_{{i=1\normalcolor}}^K \frac{l_i}{g_i}\nonumber\\
    &=&-t -  K^{ \frac{1}{1-\xi}} \delta^2\sum_{{i=1\normalcolor}}^K  \frac{l_i^2}{g_i^2} + \delta K^{ \frac12 \frac{ 1+\xi }{1-\xi}}\sum_{{i=1\normalcolor}}^K \frac{l_i}{g_i}.\nonumber
  \end{eqnarray}
  Plugging this into \eqref{eq:sumoverall} gives
  \begin{eqnarray}
      &&|\mathcal B_{t,K}| \sum_{ \{(l_i): \sum_{i=1}^K l_i <0 \}} \exp\left( -t -  K^{ \frac{1}{1-\xi}}  \delta^2\sum_{{i=1\normalcolor}}^K  \frac{l_i^2}{g_i^2} + K^{ \frac12 \frac{ 1+\xi }{1-\xi}} \delta\sum_{{i=1\normalcolor}}^K \frac{l_i}{g_i}\right)\nonumber\\
      &=& 
          |\mathcal B_{t,K}  |
    \sum_{ \{(l_i): \sum_{i=1}^K l_i <0 \}}\exp\left( -t -  K^{ \frac{1}{1-\xi}} \delta^2\sum_{{i=1\normalcolor}}^K  \frac{l_i^2}{g_i^2} + K^{ \frac12 \frac{ 1+\xi }{1-\xi}}\delta\sum_{{i=1\normalcolor}}^K \frac{l_i}{g_i}\right).\label{eq:sumsrewritten1}
  \end{eqnarray}
  Remembering the form of the $g_i$ and writing
  \begin{equation}
      h_i := \left( 4H C_{\rho} \left(1-\frac{i}{K}\right)^{2H-1} \right)^{\frac12}
  \end{equation}
  we get that \eqref{eq:sumsrewritten1} is asymptotically equivalent to
  \begin{equation}
    |\mathcal B_{t,K}  |  
    \sum_{ \{(l_i): \sum_{i=1}^K l_i <0 \}} \exp\left( -t -  K^{ \frac{2-\xi}{1-\xi}} \delta^2\sum_{{i=1\normalcolor}}^K  \frac{l_i^2}{h_i^2} + K^{  \frac{ 1 }{1-\xi}}\delta\sum_{{i=1\normalcolor}}^K \frac{l_i}{h_i}\right)\label{eq:sumsrewritten}.
  \end{equation}
  Since $ K^{ \frac{2-\xi}{1-\xi}} \gg K^{  \frac{ 1 }{1-\xi}}$ and $h_i^{-1}$ is bounded this is 
  logarithmically equivalent to
  \begin{equation}
    |\mathcal B_{t,K}  | 
    \sum_{ \{(l_i): \sum_{i=1}^K l_i <0 \}}\exp\left( -t -  K^{ \frac{2-\xi}{1-\xi}} \delta^2\sum_{{i=1\normalcolor}}^K  \frac{l_i^2}{h_i^2} 
      \right).\label{eq:sumallterms}
  \end{equation}
  This we can write as
  \begin{equation}
      \exp(-t)|\mathcal B_{t,K}  | \sum_{r=1}^K \sum_{m\geq 1} \sum_{ \substack{\{(l_i): \sum_{i=1}^K l_i <0,\\ \# \{i: l_i\not=0\}=r,\\
      \max |l_i|=m\}}} \exp\left( -  K^{ \frac{2-\xi}{1-\xi}} \delta^2\sum_{{i=1\normalcolor}}^K  \frac{l_i^2}{h_i^2} 
      \right).\label{eq:sumalltermssecond}
  \end{equation}
  Now observe that for fixed $r$ and $m$ the third sum in \eqref{eq:sumalltermssecond} consists at most of $\binom{K}{r}(2m)^r $ many summands. 
  By Stirling's formula 
  \begin{equation}
    \frac{ \#\left\{ (l_i): \sum_{i=1}^K l_i <0,\# \{i: l_i\not=0\}=r, 
      \max |l_i|=m\right\}}{ \exp\left(K \log(K) + \log(2m)r \right)} \rightarrow 0
  \end{equation}
  as $K\rightarrow \infty$.
  Since $\tfrac{1}{h_i^2}$ is bounded from below by $\tfrac{1}{4H C_{\rho}}$ and at least for one $i\in\{1,\ldots, K\}$ we have $l_i^2 = m^2$, the third sum in \eqref{eq:sumalltermssecond} is subexponentially small in $K$ and $m$ for all fixed $r$ and $m$. 
  We conclude that \eqref{eq:sumalltermssecond} vanishes for $K\rightarrow \infty$.\\
%
%
Finally observe that $\delta$ can be taken to zero such that the logarithmic equivalences of \eqref{eq:sumsrewritten} and \eqref{eq:sumallterms} still hold.
So by
 \begin{equation}
  \frac{ |\mathcal B_{t,K}| }{e^{t}} \Rightarrow \mathrm{Exp}(1)
\end{equation}
the distribution of $|\mathcal B_{t,K}|$ is sufficiently concentrated, such that
averaging \eqref{eq:sumallterms} over the random tree gives a term of order $o(1)$, such that in total  the leading order of the maximum is bounded by
\begin{equation}\label{eq:sumdeltai}
  \sum_{i=0}^K \Delta f_i
\end{equation}
from above.  Note that this is of the same order as $\sum_{i=1}^K \Delta f_i$.\normalcolor\\\\
To obtain the lower bound
we change the above construction, such that we shift all branching events in $\left[ \frac{i}{K} t , \frac{i+1}{K} t\right]$ of the Yule tree $\mathfrak y$ to $\frac{i+1}{K} t$, such that we can only have branching events at $\frac{t}{K},\ldots, \frac{Kt}{K}$. By Slepian's Lemma, see Lemma~\ref{th:slepian}, the order of the maximum of the corresponding GREM gives us a lower bound on the order of the maximum of BFBM. 
We now aim to apply the Paley–Zygmund inequality to show that for some function $\zeta: \R\rightarrow \R$, such that $\zeta(t)\rightarrow 1$ as $t\rightarrow \infty$, the probability to see a branch $b$ with $\Delta Z_b^{(i)} \geq \Delta f_i\cdot \zeta(t)$ for all $i\geq 1$ is positive.  Remember that we are allowed to ignore $\Delta Z_b^{(0)}$ since $\sum_{i=0}^K \Delta f_i \sim \sum_{i=1}^K \Delta f_i$. \normalcolor
Together with the classical concentration inequality by Talagrand, see \cite{Talagrand1995}, this gives that as a lower bound on the leading order of the maximum we obtain \eqref{eq:sumdeltai}: 
Choosing $\xi > \frac23, K=t^{1-\xi}$ and
\begin{equation}
    \zeta(t)= \left(  1+ \frac12\log\left(\frac{1}{4\pi}  \right)\cdot \frac{K}{t} + \frac12\log\left( \frac{K}{t}\right)\cdot \frac{K}{t}  {+\log(K) \cdot \frac{K}{t} } \right)^{\frac12},
\end{equation}
we have {$\zeta(t)\downarrow 1$} as $t\rightarrow \infty$. Now
by a standard Gaussian tail estimate we have
\begin{equation}
\WS_{ \mathfrak y }\left(\Delta Z_0^{(i)} \geq \Delta f_i\cdot \zeta(t) \right) \geq \mathrm{const} \frac{1}{\sqrt{2\pi} \cdot \sqrt{2} \cdot \zeta(t) }\cdot \left(\frac{K}{t}\right)^{\frac12}\cdot\exp\left( - \frac{t}{K}\cdot \zeta(t)^2 \right) =  \exp\left( - \frac{t}{K} \right) \frac{\mathrm{const}}{\zeta(t){ \cdot K }}.
\end{equation}
 Since we shifted all branching events a little bit to the right, this implies that the expected number of branches such that $\Delta Z_b^{(i)} \geq \Delta f_i \cdot \zeta(t)$ for all $i$ is bounded from below by $\mathrm{const} \cdot \zeta(t)^{-K} K^{-K} \exp\left(-t/K\right)$ \normalcolor for $t$ large enough. The probability that two branches $b$ and $\tilde b$ with $b\wedge \tilde b= \tfrac{l}{K}t $ fulfill $\Delta Z_b^{(i)},\Delta Z_{\tilde b}^{(i)} \geq \Delta f_i \cdot \zeta(t) $ for all $i$ is given by 
\begin{equation}
    \left[ \prod_{{i=1\normalcolor}}^{l}\WS_{ \mathfrak y }\left( \Delta Z_{0}^{(i)} \geq \Delta f_i\cdot \zeta(t) \right) \right]\cdot \left[ \prod_{i=l+1}^{K}\WS_{ \mathfrak y }\left( \Delta Z_{0}^{(i)} \geq \Delta f_i\cdot \zeta(t)\right) \right]^2.
\end{equation}
By the corresponding standard Gaussian tail estimate
\begin{equation}
\WS_{ \mathfrak y }\left(\Delta Z_0^{(i)} \geq \Delta f_i\cdot\zeta(t) \right) \leq \mathrm{const} \frac{1}{\sqrt{2\pi } \cdot \sqrt{2} \cdot \zeta(t) { \cdot K }}\cdot \left(\frac{K}{t}\right)^{\frac12}\cdot\exp\left( - \frac{t}{K}\cdot \zeta(t)^2 \right) =  \exp\left( - \frac{t}{K} \right) \frac{\mathrm{const}}{\zeta(t){ \cdot K } }
\end{equation}
 this is bounded from above by
 \begin{equation}
     \frac{\mathrm{const}}{\zeta(t)^{2K-l} { \cdot K^{2K-l} } } \exp\left( - l \frac{t}{K} - 2(K-l) \frac{t}{K}\right) = \frac{\mathrm{const}}{\zeta(t)^{2K-l} { \cdot K^{2K-l} }}\exp\left( - (2K-l) \frac{t}{K}\right).
 \end{equation}
 Denote by $\EW_{\mathcal P}$ the expectation with averages with respect to $\mathcal P$.
 Since now in expectation we have of order
 \begin{equation}
     \exp\left( \frac{t}{K}\left( 2K-l-1 \right)\right)\normalcolor
 \end{equation}
 many choices for two branches that separate at time $l \frac{t}{K}$ 
this directly gives that
\begin{equation}
    \EW_{\mathcal P}\left[\left( \# \left\{ b\in \mathcal B_{t,K}: \Delta Z_b^{(i)} \geq \Delta f_i \cdot \zeta(t) \qquad \forall i \right\}\right)^2 \right]
\end{equation}
is bounded from above by
\begin{equation}
    \mathrm{const}\cdot \sum_{l=0}^K  \exp\left( \frac{t}{K}\left( 2K-l -1\right)\right)  \frac{1}{\zeta(t)^{2K-l} { \cdot K^{2K-l} }}\exp\left( - (2K-l) \frac{t}{K}\right)  \leq \mathrm{const} \cdot \zeta(t)^{-K} K^{-K}\exp\left(-t/K\right).
\end{equation}

For $1>\xi > \frac12$ the factor $\zeta(t)^{-K}$ converges to 1
as $K\rightarrow \infty$. So in total
the Paley–Zygmund inequality gives that there exists a constant $c(H)$, such that for $t$ large enough
\begin{equation}
    \mathcal P\left( \# \left\{ b\in \mathcal B_{t,K}: \Delta Z_b^{(i)} \geq \Delta f_i \qquad \forall i \right\}>0\right)\geq c(H)>0. 
  \end{equation}
  Since $\max_b Z_b$ is a Lipschitz-function of the increments $\Delta Z_b^{(i)}$ we can apply the standard concentration inequality by Talagrand, see  \cite[(1.14)]{Talagrand1995},
such that as a lower bound on the leading order of the maximum we also obtain \eqref{eq:sumdeltai}. 
Now by \eqref{eq:deltaf} 
\begin{eqnarray}
  \Delta f_i &=&   \sqrt{ 2\frac{t}{K} } t^H \left[ -C_{\rho} \left( 1-\frac{i}{K} \right)^{2H}+C_{\rho} \left( 1-\frac{i-1}{K} \right)^{2H} \right]^{\frac12}\nonumber\\
             &=& t^{H+\frac12} \sqrt{2C_{\rho}}  \left[ \frac{-1}{K} \left[ \left( 1-\frac{i}{K} \right)^{2H}-    \left( 1-\frac{i-1}{K} \right)^{2H} \right]  \right]^{\frac12}\nonumber\\
             &\approx& t^{H+\frac12} \sqrt{2C_{\rho}} \sqrt{2H} \left[\frac{1}{K^2} \left( 1-\frac{i}{K} \right)^{2H-1}  \right]^{\frac12}\nonumber\\
             &=&2  t^{H+\frac12} \sqrt{C_{\rho} H}   \left[ \frac{1}{K^2}\left( 1-\frac{i}{K} \right)^{2H-1}  \right]^{\frac12}\nonumber\\
             &=& 2  t^{H+\frac12} \sqrt{C_{\rho} H} K^{-1} \left( 1-\frac{i}{K} \right)^{H-\frac12} \nonumber,
\end{eqnarray}
which for $t\equiv K$ large gives
\begin{eqnarray}
  f\left( \frac{l}{K}t \right)&:=&  \sum_{i=1}^{ l } \Delta f_i\label{eq:flk}=  2  t^{H+\frac12} \sqrt{C_{\rho} H}K^{-1}\sum_{i=1}^{ l }\left( 1-\frac{i}{K} \right)^{H-\frac12} \\
                              &\sim&  2  t^{H+\frac12} \sqrt{C_{\rho} H} \int_0^{ l/K} \left( 1-y \right)^{ H-\frac12 }\dif{y}.\nonumber
\end{eqnarray}
For $l=K$ the left hand side of \eqref{eq:flk} equals \eqref{eq:sumdeltai}. 
This gives the desired result. 
\end{proof}

\newpage
\appendix

 \section{Proof of \eqref{eq:ThMaxWeak} and \eqref{eq:ThMaxWeak2}}\label{sec:slepianapplication}
In this section we give a proof of \eqref{eq:ThMaxWeak} and \eqref{eq:ThMaxWeak2} to show how Slepian's Lemma, Lemma~\ref{th:slepian}, is already applicable to obtain the order of magnitude of $ \max\limits_{b\in \mathcal B} B_{b}(t)$ without the prefactor.  
\begin{proof}[Proof of \eqref{eq:ThMaxWeak} and \eqref{eq:ThMaxWeak2}]
   To prove  \eqref{eq:ThMaxWeak} and \eqref{eq:ThMaxWeak2}   we will make use of Lemma~\ref{th:slepian}. Given a Yule-Tree with $\mathcal N(t) \sim \mathrm{Geo}\left( e^{-t}  \right)$ many leaves at time $t$ we choose $s:= xt, 0<x<1$  arbitrary  and couple a process in the following way: All branching events after time $s$ get deleted. All branching events before time $s$ get shifted to exactly $s$, such that there is one branch till $s$ and at time $s$ exactly $\mathcal N(t)$ branch off simultaneously. (So we only consider a modified version of a subtree. To keep notation consistent we do not modify the names of the branches while modifying the tree.) On this tree we build a BFBM $\tilde B$ with Hurst parameter $H$, covariance structure $\rho$ and branching rate $1$. Obviously for all branches $\tilde b_1, \tilde b_2$ in the new tree and all branches $b_1, b_2$ in $\mathfrak y$ which start before time $s$, we get
  \begin{equation}
    \rho\left(t, t, s  \right)= \Cov\left[ \tilde{B}_{ \tilde b_1 }(t),\tilde{B}_{ \tilde b_2}(t)  \right] \geq \Cov\left[ {B}_{  b_1 }(t),{B}_{  b_2}(t)  \right]. 
  \end{equation}
  By Lemma~\ref{th:slepian} the order of the maximum of our new process is lower than the one of the original process, giving us a lower bound on the maximum of our BFBM. Since
  \begin{equation}
    \tilde B_{b} (t) \overset{(d)}{=} \sqrt{\rho(t,t,s)} Z + \sqrt{ t^{2H} - \rho(t,t,s) }Z_{b}
  \end{equation}
  for $Z, \left( Z_b \right)_{b\in \mathcal B}$ independent and standard Gaussian, we see that the order of the maximum over $\tilde B_b(t)$ is of the same order as
  \begin{equation}
    \max_{b\in \mathcal B}\sqrt{ t^{2H} - \rho(t,t,s) }Z_{b}. 
  \end{equation}
  By a standard tail estimate we get 
  \begin{eqnarray*}
    &&\WS_{ \mathfrak y }\left( r t^H Z_b > Ct^{H+\frac12}+a_t \right)\\
    &=& \WS_{ \mathfrak y }\left( Z_b > \frac{ c t^{\frac12} }{r} + \frac{a_t}{r t^H } \right)\\
    &\sim& \frac{r}{\sqrt{2\pi}C } \exp\left( - \frac12 \log(t) - \frac12 \left( \frac{C^2}{r^2}t + \frac{a_t^2}{r^2 t^{2H}} + \frac{2C}{r} t^{\frac12 - H} a_t \right) \right)\\
    &=& \frac{r}{\sqrt{4\pi }} \exp\left( - \frac12 \log(t) - \frac12 \left( 2 t + o(1) -\log(t) \right) \right)\\
    &=& \frac{1}{\sqrt{4\pi }} \exp\left( -  t \right)\left(1+o(1)\right)
  \end{eqnarray*}
  for
  \begin{equation}
    C= \sqrt{2} r
  \end{equation}
  and
  \begin{equation}
    a_t = - t^{H-\frac12} \log(t) \frac{\rho}{2 \sqrt{2}}.
  \end{equation}
  So
  since $ e^{- t} \mathcal N(t)$ is a martingale with
  \begin{equation}
    \mathrm{Exp}(1)\sim E_1 = \lim\limits_{t\rightarrow \infty} e^{- t} \mathcal N(t),
  \end{equation}
  we immediately get  \eqref{eq:ThMaxWeak} .
  This shows the lower bound. The upper bound follows analogously for $\mathcal N(t)$ independent fractional Brownian motions by simply using a first moment bound. 
\end{proof}
 
\section{Proofs of some properties of branching, $\mathfrak y$-indexed processes  }\label{sec:gaussianstuff}
We start by giving a proof of Lemma~\ref{eq:corkallenberg}:
\begin{proof}[Proof of Lemma~\ref{eq:corkallenberg}]
For all $b' \in \left\{b_1,\ldots, b_m\right\}$ and $t, t'> b'\wedge b$ we have
  \begin{equation}
    \Cov\left[ \left. X_{b}(t), X_{b'}(t')  \right\vert X_{ b \cap b'}  \right]=0
  \end{equation}
  since $X_{\mathfrak y}$ has the \emph{pairwise branching property \eqref{eq:equaldist}}, see \eqref{eq:equaldist}. By construction one can choose $b'$ such that
  \begin{equation}
    X_{\mathcal S \cap b} = X _{b' \cap b}.
  \end{equation}
  This implies that for all branches $\hat b \in \left\{b_1,\ldots, b_m\right\}$ and $t, \hat t > \hat b \wedge b$
   \begin{equation}
     \Cov\left[ \left. X_{b}(t), X_{\hat b}(\hat t)  \right\vert X_{ b \cap \mathcal S}  \right]=
     \Cov\left[ \left. X_{b}(t), X_{\hat b}(\hat t)  \right\vert X_{ b \cap b'}  \right].
   \end{equation}
   Now
   \begin{equation}
     \Cov\left[ \left. X_{b}(t), X_{\hat b}(\hat t)  \right\vert X_{ b \cap b'}  \right]= \Cov\left[ \left. X_{b}(t), X_{\hat b}(\hat t)  \right\vert X_{ \left( b \cap b' \right)\backslash \hat b}, X_{ b\cap \hat b }  \right].
   \end{equation}
   But since
   \begin{equation}
     \left( X_b(t), X_{ \left( b\cap b' \right)\backslash \hat b } \right) \quad \mbox{ and }\quad X_{\hat b}(\hat t)
   \end{equation}
    are conditionally independent given
   \begin{equation}
      X_{ b\cap \hat b }
    \end{equation}
    and
    \begin{equation}
      X_{\hat b}(\hat t), \qquad  X_{ \left( b\cap b' \right)\backslash \hat b }
    \end{equation}
    are conditionally independent given
    \begin{equation}
      X_{b\cap \hat b}
    \end{equation}
    we obtain
      \begin{eqnarray*}
     &&\mathcal L \left( \left.  \left( X_b(t), X_{ \left( b\cap b' \right)\backslash \hat b } \right),  X_{\hat b}(\hat t)  \right\vert X_{ b\cap \hat b }  \right)\\
     &=& \mathcal L \left( \left.  \left( X_b(t), X_{ \left( b\cap b' \right)\backslash \hat b } \right)  \right\vert X_{ b\cap \hat b }  \right)\otimes \mathcal L \left( \left.  X_{\hat b}(\hat t)  \right\vert X_{ b\cap \hat b }  \right)\\
     &=& \left[ \mathcal L \left( \left.  \left( X_b(t) \right)  \right\vert  X_{ b\cap \hat b }  \right)\otimes
\mathcal L \left( \left.  X_{ \left( b\cap b' \right)\backslash \hat b }  \right\vert  X_{ b\cap \hat b }  \right)
             \right] \\
     &&\otimes \mathcal L \left( \left.  X_{\hat b}(\hat t)  \right\vert X_{ b\cap \hat b }  \right)\\
     &=& \left[ \mathcal L \left( \left.  \left( X_b(t) \right)  \right\vert X_{ \left( b\cap b' \right)\backslash \hat b },  X_{ b\cap \hat b }  \right)\otimes
\mathcal L \left( \left.  X_{ \left( b\cap b' \right)\backslash \hat b }  \right\vert  X_{ b\cap \hat b }  \right)
             \right]\\
     && \otimes \mathcal L \left( \left.  X_{\hat b}(\hat t)  \right\vert X_{ \left( b\cap b' \right)\backslash \hat b },  X_{ b\cap \hat b }  \right)\\
   \end{eqnarray*}
   giving us
   \begin{equation}
     \Cov\left[ \left. X_{b}(t), X_{\hat b}(\hat t)  \right\vert X_{ \left( b \cap b' \right)\backslash \hat b}, X_{ b\cap \hat b }  \right]
     =  \Cov\left[ \left. X_{b}(t), X_{\hat b}(\hat t)  \right\vert X_{ b\cap \hat b }  \right]=0. 
   \end{equation}

  Now since $X_{\mathfrak y}$ is Gaussian this implies that $X_{b\backslash \mathcal S }$ and $X_{ \mathcal S \backslash b }$ are  conditionally independent given $X_{\mathcal S \cap b}$. This implies that $X_{\mathfrak y}$ has the \emph{branching property \eqref{eq:strongprop}}.
\end{proof}
\begin{proof}[Proof of Corollary~\ref{cor:multivariateGaussianBP}]
      Let $t_1\leq t_2$, then
    \begin{equation}
      \left( X_b(t) \right)_{t_1 \leq t \leq t_2}
    \end{equation}
    is a Gaussian process for all $b\in \mathcal B$.  Set $\mathcal S:=\bigcup_{i=1}^m b_i$ for  branches $b_1,\ldots, b_m$.
    Then by the \emph{branching property \eqref{eq:strongprop}}
    \begin{eqnarray*}
      \mathcal L \left( X_{ b\cup  \mathcal S } \right)&=& \mathcal L \left(\left. \mathcal L \left(X_{ b \backslash \mathcal S }, X_{\mathcal S \backslash b}  \right\vert X_{ b\cap \mathcal S } \right)\otimes \delta_{ X_{ b\cap \mathcal S } } \right)\\
      &=&\mathcal L \left( \mathcal L \left(\left. X_{ b \backslash \mathcal S }  \right\vert X_{ b\cap \mathcal S } \right)\otimes \mathcal L \left(  \left.X_{\mathcal S \backslash b} \right\vert X_{ b\cap \mathcal S } \right) \otimes \delta_{ X_{ b\cap \mathcal S } } \right).
    \end{eqnarray*}
    Note that $ \mathcal L \left(\left. X_{ b \backslash \mathcal S }  \right\vert X_{ b\cap \mathcal S } \right)$ is the distribution of a Gaussian process conditioned on its past, as is $ \mathcal L \left(  \left.X_{\mathcal S \backslash b} \right\vert X_{ b\cap \mathcal S } \right)$ by iteration, such that in total $ \mathcal L \left( X_{ b\cup  \mathcal S } \right)$ is a Gaussian process. 
\end{proof}
The proof of Corollary~\ref{eq:corasymjointGauss} is now very similar to the one of Corollary~\ref{cor:multivariateGaussianBP}, but Gaussianity gets exchanged with asymptotic Gaussianity. 
 \begin{proof}[Proof of Corollary~\ref{eq:corasymjointGauss}]
     Let $t_1\leq t_2$, then
    \begin{equation}
      \left( X_b^{(n)}(t) \right)_{t_1 \leq t \leq t_2}
    \end{equation}
    is an asymptotically Gaussian process for all $b\in \mathcal B$ and by the \emph{branching property \eqref{eq:strongprop}}  for branches $b_1,\ldots, b_m$ with $\max\limits_{i=1,\ldots, m} b\wedge b_i = s$ 
    \begin{equation}
      \left( X^{(n)}_b(t) \right)_{t \geq s }, \qquad  \left( \left( X^{(n)}_{b_i}(t) \right)_{t\geq b_i \wedge b} \right)_{i=1,\ldots, m}
    \end{equation}
    are conditionally independent given
    \begin{equation}\label{eq:anothergaussianone}
      \left( \left( X^{(n)}_{b_i}(t) \right)_{t\leq b_i \wedge b} \right)_{i=1,\ldots, m}.
    \end{equation}
    By similar arguments as in the proof of Corollary~\ref{cor:multivariateGaussianBP} it follows that the process \eqref{eq:anothergaussianone} is also asymptotically Gaussian. 
 This already gives the assertion. 
\end{proof}
\begin{proof}[Proof of Proposition~\ref{prop:covdist}]
  First observe that by \eqref{eq:equal}
\begin{equation}
    \mathcal L\left( \left. \left( X_{\tilde b} \right)_{ t\geq s} \right\vert  \left( X_{b} \right)_{ t\in\R}   \right)=\mathcal L\left( \left. \left( X_{b} \right)_{ t\geq s} \right\vert  \left( X_{ b} \right)_{ t\leq s}   \right) = \mathcal L\left( \left. \left( X_{\tilde b} \right)_{ t\geq s} \right\vert  \left( X_{ \tilde b} \right)_{ t\leq s}   \right)
  \end{equation}
  is implied by \eqref{eq:equaldist}.
  Now for branches $b, \tilde b \in \mathfrak y$ with $b\wedge \tilde b = s$ by definition we get for $t_1, t_2 > s$
  \begin{eqnarray*}
    &&\Cov\left[ X_{b}(t_1), X_{\tilde b}(t_2) \right]\\
    &=& \EW_{ \mathfrak y }\left[ X_{b}(t_1) X_{\tilde b}(t_2) \right]- \EW_{ \mathfrak y }\left[ X_{b}(t_1) \right]\EW_{ \mathfrak y }\left[X_{\tilde b}(t_2) \right].
  \end{eqnarray*}
  By \eqref{eq:equalbranchdist} we get
  \begin{equation}\label{eq:secondew}
    \EW_{ \mathfrak y }\left[ X_{b}(t_1) \right]\EW_{ \mathfrak y }\left[X_{\tilde b}(t_2) \right]=\EW_{ \mathfrak y }\left[ X_{\tilde b}(t_1) \right]\EW_{ \mathfrak y }\left[X_{ \tilde b}(t_2) \right].
  \end{equation}
  Using \eqref{eq:equaldist} we obtain
  \begin{eqnarray*}
    &&\EW_{ \mathfrak y }\left[ X_{b}(t_1) X_{\tilde b}(t_2) \right]\\
    &=& \EW_{ \mathfrak y }\left[ \EW_{ \mathfrak y }\left[\left. X_{b}(t_1) X_{\tilde b}(t_2) \right\vert \sigma\left( X_{\tilde b}(t), t\leq t_2 \right)  \right] \right]\\
    &=& \EW_{ \mathfrak y }\left[ \EW_{ \mathfrak y }\left[\left. X_{b}(t_1) \right\vert \sigma\left( X_{\tilde b}(t), t\leq t_2 \right)  \right]  X_{\tilde b}(t_2) \right]\\
     &=& \EW_{ \mathfrak y }\left[ \EW_{ \mathfrak y }\left[\left. X_{b}(t_1) \right\vert \sigma\left( X_{\tilde b}(t), t\leq s \right)  \right]  X_{\tilde b}(t_2) \right].
  \end{eqnarray*}
  And by \eqref{eq:equalbranchdist} we can exchange $X_{b\normalcolor}(t_1)$ and $X_{\tilde b\normalcolor}(t_1)$, meaning
   \begin{eqnarray}
     &&\EW_{ \mathfrak y }\left[ \EW_{ \mathfrak y }\left[\left. X_{b\normalcolor}(t_1) \right\vert \sigma\left( X_{\tilde b}(t), t\leq s \right)  \right]  X_{\tilde b}(t_2) \right]\nonumber \\
     &&= \EW_{ \mathfrak y }\left[ \EW_{ \mathfrak y }\left[\left. X_{\tilde b\normalcolor}(t_1) \right\vert \sigma\left( X_{\tilde b}(t), t\leq s \right)  \right]  X_{\tilde b}(t_2) \right].\label{eq:firstew}
   \end{eqnarray}
  Together \eqref{eq:firstew} and \eqref{eq:secondew} give
  \begin{eqnarray*}
    &&\Cov\left[ X_{b}(t_1), X_{\tilde b}(t_2) \right]\\
    &&= \EW_{ \mathfrak y }\left[ \EW_{ \mathfrak y }\left[\left. X_{\tilde b}(t_1) \right\vert \sigma\left( X_{\tilde b}(t), t\leq s \right)  \right]  X_{\tilde b}(t_2) \right]\\
    &&- \EW_{ \mathfrak y }\left[ X_{\tilde b}(t_1) \right]\EW_{ \mathfrak y }\left[X_{ \tilde b}(t_2) \right].
  \end{eqnarray*}
  This gives that for branches $b,\tilde b \in \mathfrak y $ with $b\wedge \tilde b=s$ and $t_1, t_2> s$
  \begin{equation}
    \Cov\left[ X_{b}(t_1), X_{\tilde b}(t_2) \right]
  \end{equation}
  is uniquely determined by the distribution of the process along one branch, i.e. by
  \begin{equation}
    \mathcal L \left( \left( X_b(t) \right)_{t\in\R} \right).
  \end{equation}
  If now for all branches $b_1, \ldots b_m$ and $t_1,\ldots t_m\in\R$
  \begin{equation}
    \left( X_{b_1}(t_1), \ldots, X_{b_m}(t_m) \right)
  \end{equation}
  is jointly Gaussian we get that the distribution of the \emph{tree-indexed branching stochastic process}
   \begin{equation}
    \left( \left( X_{b}(t) \right)_{t\in\R} \right)_{b\in \mathcal B}
  \end{equation}
  is uniquely determined by the distribution along one branch, i.e. by
   \begin{equation}
    \mathcal L \left( \left( X_b(t) \right)_{t\in\R} \right)
  \end{equation}
  for any branch $b\in \mathcal B$. 
\end{proof}

\section{Covariances of different branches}\label{sec:proof:th:dbfbm}
To prove Theorem~\ref{th:dbfbm} by Proposition~\ref{prop:condGauss} it is enough to analyze
\begin{equation}\label{eq:covges}
  \Cov\left[ \sum_{l=1}^{m_1} Y_{b,l} , \sum_{l=1}^{m_2} Y_{\tilde b,l}  \right]\quad \mbox{ for } m_1 = t_1 n, m_2 = t_2 n, t_1>t_2>s, b\in \mathcal B
\end{equation}
for two branches $b, \tilde b$ with $b\wedge \tilde b = s$         
since by \cite[Proposition~1.3]{IW} tightness is given and asymptotic joint Gaussianity is given by Proposition~\ref{prop:condGauss}. 
We now calculate the covariance between $S^{(n)}_{b}(t_1)$ and $S^{(n)}_{\tilde b}(t_2)$ for $t_1 > t_2 > s$. 

First set
\begin{equation}\label{eq:defc1}
  C_1 := C_2 \frac{ \Gamma(1-2\alpha) }{\Gamma(\alpha)\Gamma(1-\alpha)^3},
\end{equation}
for
\begin{equation}\label{eq:defc2}
  C_2 := \frac{1}{ \sum_{l\geq 0}q_l^2 } >0,
\end{equation}
and
\begin{equation}\label{eq:defc3}
  C_3 := \frac{1}{\alpha(2\alpha+1)C_1} = \frac{C_2}{\alpha(2\alpha+1)} \cdot \frac{\Gamma(1-2\alpha)}{\Gamma(\alpha)\Gamma(1-\alpha)^3}.
\end{equation}
Recalling \eqref{eq:coalprob} we get
\begin{equation}\label{eq:coalasymp}
  C_1 n^{2\alpha-1} \sim \WS_{ \mathfrak y }\left( 0 \sim n \right),
\end{equation}
see also \cite[Proposition~2.1]{IW}. 
Denote by $A_0$ the ancestral line of zero. Set
\begin{equation}\label{eq:Cq}
  C_q:= \frac{1}{\Gamma(\alpha)\Gamma(1-\alpha)},
\end{equation}
then
\begin{equation}\label{eq:qnasymp}
  C_q n^{\alpha-1} \sim q_n := \WS_{ \mathfrak y }\left( -n \in A_0 \right), 
\end{equation}
see \eqref{qas} and for further reference \cite[(1.10)]{IW}. 
Recalling \eqref{asvar}, see also \cite[(1.8)]{IW}, we get
\begin{equation}\label{eq:varnasymp}
  C_3 n^{2\alpha+1} \sim \Var\left[ \sum_{l=1}^n Y_{l} \right].
\end{equation}
Let's start by using bilinearity and the fact that $ Y_{(b,i)}= Y_{(\tilde b,i)}=: Y_i$ for $i\leq sn$: 
\begin{eqnarray}
  && \Cov\left[ \sum_{l=1}^{m_1} Y_{b,l} , \sum_{l=1}^{m_2} Y_{\tilde b,l}  \right]\nonumber\\
  &=& \Cov\left[ \sum_{l=1}^{sn} Y_{b,l} +\sum_{l=sn+1}^{m_1} Y_{b,l} , \sum_{l=1}^{sn} Y_{\tilde b,l} +\sum_{l=sn+1}^{m_2} Y_{\tilde b,l}  \right]\nonumber\\
  &=& \Cov\left[ \sum_{l=1}^{sn} Y_{b,l}, \sum_{l=1}^{sn} Y_{\tilde b, l}  \right]+  \Cov\left[ \sum_{l=1}^{sn} Y_{b,l} ,\sum_{l=sn+1}^{m_2} Y_{\tilde b, l}  \right]\nonumber\\
  &&+ \Cov\left[ \sum_{l=sn+1}^{m_1} Y_{b,l} , \sum_{l=1}^{sn} Y_{\tilde b, l} \right] \nonumber\\
  &&+ \Cov\left[\sum_{l=sn+1}^{m_1} Y_{b,l} ,\sum_{l=sn+1}^{m_2} Y_{\tilde b, l}  \right]\nonumber\\
  &=& \Var\left[ \sum_{l=1}^{sn} Y_{l}  \right]+  \Cov\left[ \sum_{l=1}^{sn} Y_{l} ,\sum_{l=sn+1}^{m_2} Y_{\tilde b, l}  \right]\nonumber\\
  &&+ \Cov\left[ \sum_{l=sn+1}^{m_1} Y_{(b,l)} , \sum_{l=1}^{sn} Y_{b,l} \right] \nonumber\\
  &&+ \Cov\left[\sum_{l=sn+1}^{m_1} Y_{b,l} ,\sum_{l=sn+1}^{m_2} Y_{\tilde b, l}  \right]\nonumber\\
  &=& \Var\left[ \sum_{l=1}^{sn} Y_{l}  \right]\label{eq:cov2}\\
  &&+  \Cov\left[ \sum_{l=1}^{sn} Y_{l} ,\sum_{l=sn+1}^{m_1} Y_{b,l}  \right] 
  +  \Cov\left[ \sum_{l=1}^{sn} Y_{l} ,\sum_{l=sn+1}^{m_2} Y_{b,l}  \right]\label{eq:cov2b}
  \\
  &&+ \Cov\left[\sum_{l=sn+1}^{m_1} Y_{b,l} ,\sum_{l=sn+1}^{m_2} Y_{\tilde b, l}  \right]. \label{eq:cov3}
\end{eqnarray}
We now need to check the asymptotics of \eqref{eq:cov2}, \eqref{eq:cov2b} and \eqref{eq:cov3}. These summands can be interpreted in the following way:
\begin{enumerate}[label=(\alph*)]
\item The summand in \eqref{eq:cov2} is the variance of the random walk along one branch.
\item The first and second summand in \eqref{eq:cov2b} represent the covariance between the increments before and after a branching event.
\item  \eqref{eq:cov3} represents the covariance between the increments of two branches after a branching event. 
\end{enumerate}
\paragraph*{\textbf{Variance of the random walk accumulated before the branching event.}}
For the first summand in \eqref{eq:cov2} we get
\begin{equation}
  \Var\left[ \sum_{l=1}^{sn} Y_{l}  \right] \sim  C_3 (sn)^{2\alpha+1},
\end{equation}
see \eqref{eq:varnasymp}.
\paragraph*{\textbf{Covariance between the increments of the random walk before and after the branching event.}}
We now analyze the second summand in \eqref{eq:cov2}. 
By $Y_l^2=1$ we get
\begin{equation}\label{eq:covidcoal}
  \Cov\left[Y_l, Y_{b,k}\right]= \WS_{ \mathfrak y }\left(l \sim (b,k)\right).
\end{equation}
This gives 
\begin{eqnarray}
  &&\Cov\left[ \sum_{l=1}^{sn} Y_{l} ,\sum_{l=sn+1}^{m_1} Y_{b,l}  \right]\nonumber\\
  &=&  \sum_{l=1}^{sn}\sum_{k=sn+1}^{m_1}  \Cov\left[  Y_{l} , Y_{b,k}  \right]\nonumber\\
  &=&  \sum_{l=1}^{sn}\sum_{k=sn+1}^{m_1} \WS_{ \mathfrak y }\left( l \sim k \right)\nonumber.
\end{eqnarray}
Plugging in the asymptotics \eqref{eq:coalasymp} for $n\rightarrow \infty$ gives
\begin{eqnarray}
      &&\sum_{l=1}^{sn}\sum_{k=sn+1}^{m_1} \WS_{ \mathfrak y }\left( l \sim k \right)\nonumber\\
    &\sim & C_1  \sum_{l=1}^{sn}\sum_{k=sn+1}^{m_1} (k-l)^{2\alpha-1}\nonumber\\
    &=& C_1 n^{2\alpha-1}\cdot n^2 \cdot \frac{1}{n}  \sum_{l=1}^{sn}\frac1n \sum_{k=sn+1}^{{x_{1}}n} \left( \frac{k}{n}-\frac{l}{n} \right)^{2\alpha-1}\nonumber\\
    &\sim& C_1 n^{2\alpha+1} \int_0^s \int_s^{{x_{1}}} \left( y_2-y_1 \right)^{2\alpha-1} \dif{y_2} \dif{y_1}.\label{eq:int1}
\end{eqnarray}
The integral in \eqref{eq:int1} can be computed explicitly: The inner integral is given by
\begin{eqnarray*}
  &&\int_s^{x_{1}} \left( y_2-y_1 \right)^{2\alpha-1} \dif{y_2}\\
  &=& \left[ \frac{1}{2\alpha}  \left( y_2-y_1 \right)^{2\alpha}  \right]_{ y_2=s }^{x_{1}}\\
  &=& \frac{1}{2\alpha} \left[ ({x_{1}}-y_1)^{2\alpha}- (s-y_1)^{2\alpha} \right].
\end{eqnarray*}
For the two summands we compute the outer integral: The first one gives
\begin{eqnarray*}
  &&\int_0^s \left( {x_{1}}-y_1 \right)^{2\alpha} \dif{y_1}\\
  &=& \left[ - \frac{1}{2\alpha+1} \left( {x_{1}}-y_1 \right)^{2\alpha+1} \right]_{y_1=0}^s \\
  &=& \frac{1}{2\alpha+1} \left[ {x_{1}}^{2\alpha+1} - ({x_{1}}-s)^{2\alpha+1} \right].
\end{eqnarray*}
The second gives
\begin{eqnarray*}
  &&\int_0^s \left(s-y_1 \right)^{2\alpha} \dif{y_2}\\
  &=& \left[ - \frac{1}{2\alpha+1} \left(s-y_1\right)^{2\alpha+1} \right]_{y_1=0}^s \\
  &=& \frac{ s^{2\alpha+1} }{2\alpha+1}.
\end{eqnarray*}
So in total we get
\begin{equation}
  \int_0^s \int_s^{{x_{1}}} \left( y_2-y_1 \right)^{2\alpha-1} \dif{y_2} \dif{y_1} = \frac{1}{2\alpha \left( 2\alpha+1 \right)} \left[ {x_{1}}^{2\alpha+1} - ({x_{1}}-s)^{2\alpha+1} -s^{2\alpha+1} \right]. 
\end{equation}
For the second summand in \eqref{eq:cov2} this gives
\begin{eqnarray*}
  && \Cov\left[ \sum_{l=1}^{sn} Y_{l} ,\sum_{l=sn+1}^{m_1} Y_{b,l}  \right]\\
  &\sim& n^{2\alpha+1} C_1  \frac{1}{2\alpha \left( 2\alpha+1 \right)} \left[ {x_{1}}^{2\alpha+1} - ({x_{1}}-s)^{2\alpha+1} -s^{2\alpha+1} \right]. 
\end{eqnarray*}
Analogously to \eqref{eq:cov2} we get
\begin{eqnarray*}
  && \Cov\left[ \sum_{l=1}^{sn} Y_{l} ,\sum_{l=sn+1}^{m_2} Y_{\tilde b,l}  \right]\\
  &\sim& n^{2\alpha+1} C_1  \frac{1}{2\alpha \left( 2\alpha+1 \right)} \left[ {x_{2}}^{2\alpha+1} - ({x_{2}}-s)^{2\alpha+1} -s^{2\alpha+1} \right]. 
\end{eqnarray*}
\paragraph*{\textbf{Covariance of the increments of the two branches}}
We now aim to compute \eqref{eq:cov3}, the covariance of the increments of two branches, namely
\begin{equation}
  \Cov\left[\sum_{l=sn+1}^{m_1} Y_{b,l} ,\sum_{l=sn+1}^{m_2} Y_{\tilde b, l}  \right],
\end{equation}
by analyzing $ \WS_{ \mathfrak y }\left( (b,l) \sim (\tilde b, k) \right)$.
Using \eqref{eq:covidcoal} we can write \eqref{eq:cov3} as
\begin{eqnarray*}
  &&\Cov\left[\sum_{l=sn+1}^{m_1} Y_{b,l} ,\sum_{l=sn+1}^{m_2} Y_{\tilde b, l}  \right]\\
  &=&\sum_{l=sn+1}^{m_1} \sum_{k=sn+1}^{m_2} \Cov\left[ Y_{b,l} , Y_{\tilde b, k}  \right]\\
  &=&\sum_{l=sn+1}^{m_1} \sum_{k=sn+1}^{m_2}  \WS_{ \mathfrak y }\left( (b,l) \sim (\tilde b, k) \right).
\end{eqnarray*}
For further analysis we will make use of  Proposition~\ref{prop:coalbranches}.
Proposition~\ref{prop:coalbranches} now gives
\begin{eqnarray*}
  &&\sum_{l=sn+1}^{m_1} \sum_{k=sn+1}^{m_2}  \WS_{ \mathfrak y }\left( (b,l) \sim (\tilde b, k) \right)\\
  &=& C_2 \sum_{l=sn+1}^{m_1} \sum_{k=sn+1}^{m_2} \sum_{r\geq 0} q_{l-sn+r} q_{k-sn+r}\\
  &=& C_2 \sum_{l=1}^{(t_1-s)n} \sum_{k=1}^{(t_2-s)n} \sum_{r \geq 0} q_{l+r}q_{k+r}.
\end{eqnarray*}
By plugging in \eqref{eq:qnasymp} we obtain
\begin{eqnarray*}
  &&C_2 \sum_{l=1}^{(t_1-s)n} \sum_{k=1}^{(t_2-s)n} \sum_{r \geq 0} q_{l+r}q_{k+r}\\
  &\sim& C_2 C_q^2 \sum_{l=1}^{(t_1-s)n} \sum_{k=1}^{(t_2-s)n} \sum_{r \geq 0} (l+r)^{\alpha-1} (k+r)^{\alpha-1}\\
  &=& C_2 C_q^2 n^{2\alpha-2} \cdot n^{3} \cdot \frac1n \sum_{l=1}^{(t_1-s)n} \frac1n  \sum_{k=1}^{(t_2-s)n} \frac1n  \sum_{r \geq 0} \left( \frac{l+r}{n} \right)^{\alpha-1} \left( \frac{k+r}{n} \right)^{\alpha-1}\\
  &\sim& C_2 C_q^2 n^{2\alpha+1} \int\limits_0^{t_1-s} \int\limits_0^{t_2-s} \int\limits_0^\infty (y_3+y_1)^{\alpha-1} (y_3+y_2)^{\alpha-1} \dif{y_3} \dif{y_2} \dif{y_1}.
\end{eqnarray*}
Finally we get
\begin{eqnarray*}
  &&\Cov\left[\sum_{l=sn+1}^{m_1} Y_{b,l} ,\sum_{l=sn+1}^{m_2} Y_{\tilde b, l}  \right]\\
  &\sim& C_2 C_q^2 n^{2\alpha+1} \int\limits_0^{t_1-s} \int\limits_0^{t_2-s} \int\limits_0^\infty (y_3+y_1)^{\alpha-1} (y_3+y_2)^{\alpha-1} \dif{y_3} \dif{y_2} \dif{y_1}. 
\end{eqnarray*}
\paragraph*{\textbf{Putting the pieces together}}\label{sec:rhoisinc}
Putting it all together 
gives 
\begin{eqnarray*}
  && \Cov\left[ \sum_{l=1}^{m_1} Y_{b,l} , \sum_{l=1}^{m_2} Y_{\tilde b, l}  \right]\\
  &\sim&n^{2\alpha+1}  \left.\Bigg[C_3 s^{2\alpha+1}\right.\\
  &&\left.+  C_1  \frac{1}{2\alpha \left( 2\alpha+1 \right)} \left[ t_1^{2\alpha+1} - (t_1-s)^{2\alpha+1} -s^{2\alpha+1} \right] \right.\\
  &&\left.+  C_1  \frac{1}{2\alpha \left( 2\alpha+1 \right)} \left[ t_2^{2\alpha+1} - (t_2-s)^{2\alpha+1} -s^{2\alpha+1} \right] \right.\\
  && \left. +  C_2 C_q^2  \int_0^{t_1-s} \int_0^{t_2-s} \int_0^\infty (y_3+y_1)^{\alpha-1} (y_3+y_2)^{\alpha-1} \dif{y_3} \dif{y_2} \dif{y_1}    \right.\Bigg].
\end{eqnarray*}
Plugging in \eqref{eq:defc1}, \eqref{eq:defc2} and \eqref{eq:defc3}, namely
\begin{equation}
  C_1 = C_2 \frac{ \Gamma(1-2\alpha) }{\Gamma(\alpha)\Gamma(1-\alpha)^3},
\end{equation}
and
\begin{equation}
  C_3 = \frac{1}{\alpha (2\alpha+1)} C_1 = \frac{C_2 }{\alpha (2\alpha+1)}  \frac{ \Gamma(1-2\alpha) }{\Gamma(\alpha)\Gamma(1-\alpha)^3}
\end{equation}
and \eqref{eq:Cq}, namely
\begin{equation}
  C_q = \frac{1}{\Gamma(\alpha)\Gamma(1-\alpha)},
\end{equation}
we can write
\begin{eqnarray*}
  &&\rho^{(n)}(t_1, t_2,s)\\
  &:=& \Cov\left[ \sum_{l=1}^{t_1 n} Y_{b,l} , \sum_{l=1}^{t_2 n} Y_{\tilde b, l}  \right]\\
  &\sim&n^{2\alpha+1}  C_2 \left[ s^{2\alpha+1} \frac{1}{\alpha (2\alpha+1)}  \frac{ \Gamma(1-2\alpha) }{\Gamma(\alpha)\Gamma(1-\alpha)^3}  \phantom{ \frac{ \int\limits_0^{t_1-s}}{\Gamma(aaa)^2 } }  \right.\\
  &&+\left.   \frac{ \Gamma(1-2\alpha) }{\Gamma(\alpha)\Gamma(1-\alpha)^3}  \frac{1}{2 \alpha \left( 2\alpha+1 \right)} \left[ t_1^{2\alpha+1} - (t_1-s)^{2\alpha+1} -s^{2\alpha+1} \right] \right.\\
  &&+\left.   \frac{ \Gamma(1-2\alpha) }{\Gamma(\alpha)\Gamma(1-\alpha)^3}  \frac{1}{2\alpha \left( 2\alpha+1 \right)} \left[ t_2^{2\alpha+1} - (t_2-s)^{2\alpha+1} -s^{2\alpha+1} \right] \right.\\
  &&+ \left.    \frac{ \int\limits_0^{t_1-s} \int\limits_0^{t_2-s} \int\limits_0^\infty (y_3+y_1)^{\alpha-1} (y_3+y_2)^{\alpha-1} \dif{y_3} \dif{y_2} \dif{y_1} }{\Gamma(\alpha)^2 \Gamma(1-\alpha)^2 }      \right].
\end{eqnarray*}
However, writing it this way makes it hard to see that
\begin{equation}\label{eq:covinc}
  \Cov\left[ \sum_{l=1}^{m_1} Y_{b,l} , \sum_{k=1}^{m_2} Y_{\tilde b, k}  \right]
\end{equation}
is indeed increasing in $s$ for fixed $m_1, m_2$.
Writing
\begin{equation}
  \Cov\left[ \sum_{l=1}^{m_1} Y_{b,l} , \sum_{k=1}^{m_2} Y_{\tilde b, k}  \right]=\sum_{l=1}^{m_1}  \sum_{k=1}^{m_2} \WS_{ \mathfrak y }\left( (b,l)\sim (\tilde b, k) \right)
\end{equation}
gives us that \eqref{eq:covinc} is increasing in $s$ for fixed $m_1, m_2$ since 
\begin{equation}
  \WS_{ \mathfrak y }\left( (b,l)\sim (\tilde b, k) \right)
\end{equation}
is increasing in $s$.
%
This is all we needed to show for the assertion of Theorem~\ref{th:dbfbm}.
\section{Proof of Proposition~\ref{prop:coalbranches}}\label{sec:proofprop:coalbranches}
The proof of Proposition~\ref{prop:coalbranches} is similar to the one in \cite[Proposition~2.1]{IW}.
The idea is based on considering two independent ancestral lines, which means lines that do not coalesce even if they meet. For this two lines we estimate the number of common ancestors and decompose it with respect to their most common ancestor.
%
\begin{proof}[Proof of Proposition~\ref{prop:coalbranches}]
  Assume without loss of generality $j>i$ and let $b, \tilde b$ be two branches with $b\wedge \tilde b = s$. Denote by $\tilde A_{(b,i)}$ and $\tilde A_{(\tilde b, j)}$ two independent (non-coalescing) ancestral lines starting in $(b,i)$ and in $(\tilde b, j)$ respectively.
  Since those can not meet before $sn$ we can write
  \begin{equation}
    \EW_{ \mathfrak y }\left[ \left\vert\tilde A_{(b,i)} \cap\tilde A_{(\tilde b, j)}   \right\vert  \right] = \sum_{ r \geq i-sn } q_r q_{r+j-i}.
  \end{equation}
  By a decomposition with respect to the most recent common ancestor we can write
  \begin{equation}
    \EW_{ \mathfrak y }\left[ \left\vert\tilde A_{(b,i)} \cap\tilde A_{(\tilde b, j)}   \right\vert  \right] =  \WS_{ \mathfrak y }\left( (b,i) \sim (\tilde b, j) \right) \sum_{r\geq 0}q_r^2.
  \end{equation}
  Giving us
  \begin{equation}
    \WS_{ \mathfrak y   }\left( (b,i) \sim (\tilde b, j) \right) = C_2 \sum_{ r \geq (i\wedge j)-sn } q_r q_{r+|j-i|} = C_2 \sum_{ r\geq 0 } q_{ i-sn +r }q_{j-sn+r} .
  \end{equation}
\end{proof}
\section{Some analytic identities}\label{sec:analyticidentities}
In this section we want to state some analytic identities which are a consequence of \eqref{eq:covequality}.
First recall the definitions
\begin{equation}
  C_H:= \left( - \frac{ 2^{-2H} \Gamma(-H) \Gamma\left(H+\frac12\right) }{\sqrt{\pi}}  \right)^{\frac12},
\end{equation}
and
\begin{eqnarray}
  \rho^{\mathrm{HS}}\left(t_1, t_2, s\right)
  &:=&\frac12 \left[ t_1^{2\alpha+1} - (t_1-s)^{2\alpha+1}
         +     t_2^{2\alpha+1} -  (t_2-s)^{2\alpha+1}\right] \nonumber\\
    &&+   \frac{ \int\limits_0^{t_1-s} \int\limits_0^{t_2-s} \int\limits_0^\infty (y_3+y_1)^{\alpha-1} (y_3+y_2)^{\alpha-1} \dif{y_3} \dif{y_2} \dif{y_1} }{ \Gamma(1-2\alpha) \Gamma(\alpha) \Gamma(1-\alpha)^{-1} (\alpha (2\alpha+1))^{-1} }.\nonumber
\end{eqnarray}
as well as
\begin{eqnarray*}
   \rho^{\mathrm{K}}\left(t_1, t_2, s\right)
  &:=&  \frac{1}{C_H^2}\left(\int_{-\infty}^0 \left((t_1-\xi)^{H-\frac12}-(-\xi)^{H-\frac12}  \right)\left((t_2-\xi)^{H-\frac12}-(-\xi)^{H-\frac12}  \right) \dif{\xi} \right) \\
  && + \frac{1}{C_H^2}\left( \int_0^{s} (t_1-\xi)^{H-\frac12} (t_2-\xi)^{H-\frac12} \dif{\xi}   \right).
\end{eqnarray*}
We already know that for
 $H=\alpha + \frac12$ 
\begin{equation}\label{eq:covequal}
  \rho^{\mathrm{HS}}(t_1, t_2,s) = \rho^{\mathrm{K}}\left( t_1, t_2, s \right).
\end{equation}
This gives the following analytic identity: 
\begin{proposition}\label{prop:id1}
  For all $t_1, t_2 >s \geq 0$ we have the equality
  \begin{eqnarray*}
  &&  \frac12 \left[ t_1^{2\alpha+1} - (t_1-s)^{2\alpha+1}\nonumber
     +     t_2^{2\alpha+1} -  (t_2-s)^{2\alpha+1}\right]\nonumber\\
    &&+   \frac{ \int\limits_0^{t_1-s} \int\limits_0^{t_2-s} \int\limits_0^\infty (y_3+y_1)^{\alpha-1} (y_3+y_2)^{\alpha-1} \dif{y_3} \dif{y_2} \dif{y_1} }{ \Gamma(1-2\alpha) \Gamma(\alpha) \Gamma(1-\alpha)^{-1} (\alpha (2\alpha+1))^{-1} }\nonumber\\
    &=&  \frac{1}{C_{\alpha+\frac12}^2}\left(\int_{-\infty}^0 \left((t_1-\xi)^{\alpha}-(-\xi)^{\alpha}  \right)\left((t_2-\xi)^{\alpha}-(-\xi)^{\alpha}  \right) \dif{\xi} \right)  + \frac{1}{C_{\alpha+\frac12}^2}\left( \int_0^{s} (t_1-\xi)^{\alpha} (t_2-\xi)^{\alpha} \dif{\xi}   \right).
  \end{eqnarray*}
\end{proposition}
 The special case of $t_1 =t_2 \equiv t >s$ then gives the following identity: 
 \begin{corollary}\label{cor:id2}
   Let $t>s>0$ then
  \begin{eqnarray*}
  &&t^{2\alpha +1} -(t-s)^{2\alpha +1} \frac{ \sqrt{\pi} 2^{2\alpha} }{ \Gamma(1-H)  \Gamma\left(\alpha\right) }\\
  &=& t^{2\alpha +1}- (t-s)^{2\alpha +1}+ \frac{\alpha (2\alpha +1) \Gamma(1-\alpha)}{\Gamma(\alpha)\Gamma(1-2\alpha)} \int_0^{t-s}\int_0^{t-s} \int_0^\infty (y_3+y_1)^{\alpha-1}(y_3+y_2)^{\alpha-1} \dif{y_3} \dif{y_2} \dif{y_1}.
\end{eqnarray*}
\end{corollary}
Choosing  $t_1 \equiv t>t_2 \equiv 1 $  and $s<1$ in Proposition~\ref{prop:id1} shows a beautiful connection to the generalized hypergeometric function, which is  defined by 
\begin{equation}
  {}_2F_1\left( a,b;c;z \right)=\sum_{n\geq 0} \frac{ (a)_n (b)_n }{(c)_n}\frac{z^n}{n!}, \qquad |z|<1.
\end{equation}
\begin{corollary}\label{cor:id3}
  For $t>s$
  \begin{eqnarray*}
  && \frac12 \left[ t^{2\alpha+1} - (t-s)^{2\alpha+1}
         +     1 -  (1-s)^{2\alpha+1}\right] \nonumber\\
  &&+   \frac{ \int\limits_0^{t-s} \int\limits_0^{1-s} \int\limits_0^\infty (y_3+y_1)^{\alpha-1} (y_3+y_2)^{\alpha-1} \dif{y_3} \dif{y_2} \dif{y_1} }{ \Gamma(1-2\alpha) \Gamma(\alpha) \Gamma(1-\alpha)^{-1} (\alpha (2\alpha+1))^{-1} }\nonumber\\
  &=&\lim\limits_{y\rightarrow \infty}\left[\frac{1}{\alpha C_H^2} \left[ {- \frac{ x^{\alpha+1}(t+x)^{\alpha} \left( \frac{t+x}{t} \right)^{-\alpha} {}_2F_1\left( 1+\alpha, \alpha; 2+\alpha ; - \frac{x}{t} \right) }{ \frac{1}{\alpha}+1 }}  \right.\right.\\
  &&+ {\frac{ \left(x+1\right)^{\alpha+1} \left(t+x\right)^{\alpha+1} \left( \frac{t+x}{t-1} \right)^{-\alpha}  {}_2F_1\left( 1+\alpha, \alpha; 2+\alpha ; \frac{x+1}{1-t} \right) }{\frac{1}{\alpha}+1}} \\
  &&- \frac{ x^{\alpha+1}  {}_2F_1\left( 1+\alpha, \alpha; 2+\alpha ; -x \right)  }{\frac{1}{\alpha}+1} \left.\left.+ {\frac{x^{1+\alpha} }{\frac{1}{\alpha}+2}} \right] \right]_{x=0}^{x=y}\\
   &&- \frac{1}{1+\alpha}\left(1-s\right)^{\alpha+1} \left( t-1 \right)^{\alpha} {}_2F_1\left( -\alpha, \alpha+1; \alpha+2; \frac{s-1}{t-1} \right)\\
  &&+\frac{1}{1+\alpha}\left(t-1\right)^{-\alpha}  {}_2F_1\left( -\alpha, \alpha+1; \alpha+2; \frac{-1}{t-1} \right).
\end{eqnarray*}
\end{corollary}
We start with the proof of Proposition~\ref{prop:id1}: 
\begin{proof}[Proof of Proposition~\ref{prop:id1}]
  We know that \eqref{eq:covequal} holds. Since
   \begin{eqnarray*}
  \rho^{\mathrm{HS}}(t_1, t_2,s)
  &=&  \frac12 \left[ t_1^{2\alpha+1} - (t_1-s)^{2\alpha+1}\nonumber
         +     t_2^{2\alpha+1} -  (t_2-s)^{2\alpha+1}\right] \nonumber\\
     &&+   \frac{ \int\limits_0^{t_1-s} \int\limits_0^{t_2-s} \int\limits_0^\infty (y_3+y_1)^{\alpha-1} (y_3+y_2)^{\alpha-1} \dif{y_3} \dif{y_2} \dif{y_1} }{ \Gamma(1-2\alpha) \Gamma(\alpha) \Gamma(1-\alpha)^{-1} (\alpha (2\alpha+1))^{-1} }\nonumber
   \end{eqnarray*}
  and 
   \begin{eqnarray*}
     \rho^{\mathrm{K}}\left( t_1, t_2, s \right)
      &=&  \frac{1}{C_H^2}\left(\int_{-\infty}^0 \left((t_1-\xi)^{H-\frac12}-(-\xi)^{H-\frac12}  \right)\left((t_2-\xi)^{H-\frac12}-(-\xi)^{H-\frac12}  \right) \dif{\xi} \right) \\
                                                            && + \frac{1}{C_H^2}\left( \int_0^{s} (t_1-\xi)^{H-\frac12} (t_2-\xi)^{H-\frac12} \dif{\xi}   \right),
   \end{eqnarray*}
   we have the desired result for $H=\alpha + \frac12$. 
 \end{proof}
 \begin{proof}[Proof of Corollary~\ref{cor:id2}]
  Since for $t_1=t_2 \equiv t > s$ we have
\begin{eqnarray*}
  &&\rho^{\mathrm{K}}\left( t, t, s \right)\\
  &=&  \frac{1}{C_H^2} \left(\int_{-\infty}^0 \left((t-\xi)^{H-\frac12}-(-\xi)^{H-\frac12}  \right)^2 \dif{\xi} \right)  \\
  &&+ \frac{1}{C_H^2} \left( \int_0^{s} (t-\xi)^{2H-1}  \dif{\xi}   \right)\\
  &=& t^{2H} -(t-s)^{2H} \frac{ \sqrt{\pi} 2^{2H-1} }{ \Gamma(1-H)  \Gamma\left(H+\frac12\right) }
\end{eqnarray*}
and
\begin{eqnarray*}
  &&\rho^{\mathrm{HS}}(t, t, s)\\
  &=& t^{2\alpha +1}- (t-s)^{2\alpha +1}\\
  &&+ \frac{\alpha (2\alpha +1) \Gamma(1-\alpha)}{\Gamma(\alpha)\Gamma(1-2\alpha)} \int_0^{t-s}\int_0^{t-s} \int_0^\infty (y_3+y_1)^{\alpha-1}(y_3+y_2)^{\alpha-1} \dif{y_3} \dif{y_2} \dif{y_1}.
\end{eqnarray*}
this gives the nice identity
\begin{eqnarray*}
  &&\rho^{\mathrm{K}}\left( t, t, s \right)\\
  &=&t^{2H} -(t-s)^{2H} \frac{ \sqrt{\pi} 2^{2H-1} }{ \Gamma(1-H)  \Gamma\left(H+\frac12\right) }\\
  &=& t^{2\alpha +1}- (t-s)^{2\alpha +1}\\
  &&+ \frac{\alpha (2\alpha +1) \Gamma(1-\alpha)}{\Gamma(\alpha)\Gamma(1-2\alpha)} \int_0^{t-s}\int_0^{t-s} \int_0^\infty (y_3+y_1)^{\alpha-1}(y_3+y_2)^{\alpha-1} \dif{y_3} \dif{y_2} \dif{y_1}\\
  &=&\rho^{\mathrm{HS}}(t, t, s).
\end{eqnarray*}
Plugging in $H=\alpha + \frac12$ gives the desired result. 
\end{proof}
\begin{proof}[Proof of Corollary~\ref{cor:id3}]
For $t_1=t, t_2 = 1, s< 1$ we have for $H=\alpha + \frac12$
\begin{eqnarray*}
  &&\rho^{\mathrm{K}}\left( t_1, t_2, s \right)\\
  &=&  \frac{1}{C_H^2}\left(\int_{-\infty}^0 \left((t_1-\xi)^{H-\frac12}-(-\xi)^{H-\frac12}  \right)\left((t_2-\xi)^{H-\frac12}-(-\xi)^{H-\frac12}  \right) \dif{\xi} \right) \\
  && + \frac{1}{C_H^2}\left( \int_0^{s} (t_1-\xi)^{H-\frac12} (t_2-\xi)^{H-\frac12} \dif{\xi}   \right)\\
   &=&\lim\limits_{y\rightarrow \infty}\left[\frac{1}{\alpha C_H^2} \left[ {- \frac{ x^{\alpha+1}(t+x)^{\alpha} \left( \frac{t+x}{t} \right)^{-\alpha} {}_2F_1\left( 1+\alpha, \alpha; 2+\alpha ; - \frac{x}{t} \right) }{ \frac{1}{\alpha}+1 }}  \right.\right.\\
  &&+ {\frac{ \left(x+1\right)^{\alpha+1} \left(t+x\right)^{\alpha+1} \left( \frac{t+x}{t-1} \right)^{-\alpha}  {}_2F_1\left( 1+\alpha, \alpha; 2+\alpha ; \frac{x+1}{1-t} \right) }{\frac{1}{\alpha}+1}} \\
  &&- \frac{ x^{\alpha+1}  {}_2F_1\left( 1+\alpha, \alpha; 2+\alpha ; -x \right)  }{\frac{1}{\alpha}+1}\\
  && \left.\left.+ {\frac{x^{1+\alpha} }{\frac{1}{\alpha}+2}} \right] \right]_{x=0}^{x=y}\\
   &&- \frac{1}{1+\alpha}\left(1-s\right)^{\alpha+1} \left( t-1 \right)^{\alpha} {}_2F_1\left( -\alpha, \alpha+1; \alpha+2; \frac{s-1}{t-1} \right)\\
  &&+\left(t-1\right)^{-\alpha} \frac{ {}_2F_1\left( -\alpha, \alpha+1; \alpha+2; \frac{-1}{t-1} \right) }{\alpha+1}.
\end{eqnarray*}
 To obtain this equality we apply the identity, which can be checked with every well-known CAS, for example \texttt{Mathematica},
\begin{eqnarray*}
&&\int \left((x+1)^{\frac{1}{c}}-x^{\frac{1}{c}}\right)
   \left((b+x)^{\frac{1}{c}}-x^{\frac{1}{c}}\right)\\
   &=& c \left(-\frac{x^{\frac{1}{c}+1} (b+x)^{\frac{1}{c}}
   \left(\frac{b+x}{b}\right)^{-1/c} {}_2F_1\left(1+\frac{1}{c},-\frac{1}{c};2+\frac{1}{c};-\frac{x}{b}\right)}{c+1}\right.\\
   &&\left.+\frac{(x+1)^{\frac{1}{c}+1} (b+x)^{\frac{1}{c}} \left(\frac{b+x}{b-1}\right)^{-1/c} 
   {}_2F_1\left(1+\frac{1}{c},-\frac{1}{c};2+\frac{1}{c};\frac{x+1}{1-b}\right)}{c+1}\right.\\
   &&\left.-\frac{x^{\frac{1}{c}+1} \,
   _2F_1\left(1+\frac{1}{c},-\frac{1}{c};2+\frac{1}{c};-x\right)}{c+1}+\frac{x^{\frac
   {c+2}{c}}}{c+2}\right)  , \qquad b>1,\, c>0
\end{eqnarray*}
with $c=\tfrac{1}{\alpha},\, b=t$, 
%
and transform the second integral, such that we can make use of
\cite[\href{https://dlmf.nist.gov/15.6\#E1}{(15.6.1)}]{NISTDLMF}.

By \eqref{eq:covequal} 
this is equal to
\begin{eqnarray*}
 && \rho^{\mathrm{HS}}(t, 1,s)\\
  \\&=&  \frac12 \left[ t^{2\alpha+1} - (t-s)^{2\alpha+1}\nonumber
         +     1 -  (1-s)^{2\alpha+1}\right] \nonumber\\
    &&+   \frac{ \int\limits_0^{t-s} \int\limits_0^{1-s} \int\limits_0^\infty (y_3+y_1)^{\alpha-1} (y_3+y_2)^{\alpha-1} \dif{y_3} \dif{y_2} \dif{y_1} }{ \Gamma(1-2\alpha) \Gamma(\alpha) \Gamma(1-\alpha)^{-1} (\alpha (2\alpha+1))^{-1} }\nonumber.
\end{eqnarray*}
This gives the desired identity.
\end{proof}
\newpage
\bibliography{bfbm} 
\bibliographystyle{habbrv}
\end{document}